\documentclass[a4paper,12pt]{elsarticle}
\usepackage[utf8]{inputenc} 
\usepackage[T1]{fontenc}   
\usepackage{hyperref}     
\usepackage{url}           
\usepackage{booktabs}      
\usepackage{amsfonts}       
\usepackage{nicefrac}      
\usepackage{microtype}     
\usepackage{lipsum}
\usepackage{amssymb}
\usepackage{amsmath}
\usepackage{textcomp}
\usepackage{eurosym}
\usepackage{verbatim}
\usepackage{amsthm}
\usepackage{commath}
\usepackage{mathtools}
\usepackage{graphicx}
\usepackage{caption}
\usepackage{subcaption}
\usepackage{dsfont}
\usepackage{xcolor}
\usepackage[numbers]{natbib}
\usepackage{float}
\usepackage{wrapfig}
\usepackage{mathtools}
\usepackage{bigints}
\usepackage{breqn}

\theoremstyle{plain}
\newtheorem{thm}{Theorem}[section]

\newtheorem{coro}[thm]{Corollary}

\newtheorem{defn}[thm]{Definition}
\newtheorem{prop}[thm]{Proposition}
\theoremstyle{definition}
\newtheorem{ex}[thm]{Example}
\newtheorem{rem}[thm]{Remark}

\newcommand{\N}{\mathbb{N}}
\newcommand{\R}{\mathbb{R}}

\newcommand{\ind}{\perp\!\!\!\perp}
\def\e{\mathrm{e}}
\def\E{\mathbb{E}}
\def\d{\, \mathrm{d}}
\def\1{\mathds{1}}
\def\R{\mathbb{R}}
\def\P{\mathbb{P}}

\def\E{\mathbb{E}}
\def\N{\mathbb{N}}

\def\X{\mathcal{X}}
\def\Y{\mathcal{Y}}
\def\K{\textup{\textbf{K}}}

\def\|{\, | \,}

\def\rd{\,\mathrm{d}}

\def\var{\mathrm{Var}}
\def\cov{\mathrm{Cov}}

\def\e{\mathrm{e}}

\def\F{\mathcal{F}}

\def\nn{\nonumber}

\DeclareMathOperator{\sech}{sech}
\def\eql{\overset{\mathcal L}{=}}
\def\eqas{\overset{a.s.}{=}}
\numberwithin{equation}{section}
\allowdisplaybreaks
\begin{document}
\begin{frontmatter}
\title{\texorpdfstring{Arcade Processes for \\ Informed Martingale Interpolation \tnoteref{footnote1}}{}}
\author{\texorpdfstring{Georges Kassis$^{a}$ and Andrea Macrina$^{a,b}$\, \fnref{footnote2}}{}}  
\affiliation{organization={Department of Mathematics, University College London},
            addressline={Gower Street}, 
            city={London},
            postcode={WC1E 6BT},
            country={United Kingdom}}
\affiliation{organization={AIFMRM, University of Cape Town},
            addressline={Private Bag X3}, 
            city={Rondebosch, Cape Town},
            postcode={7701},
            country={South Africa}}
\fntext[footnote2]{Corresponding author: a.macrina@ucl.ac.uk}
\begin{abstract}
Arcade processes are a class of continuous stochastic processes that interpolate in a strong sense, i.e., omega by omega, between zeros at fixed pre-specified times. Their additive randomisation allows one to match any finite sequence of target random variables, indexed by the given fixed dates, on the whole probability space. The randomised arcade processes (RAPs) can thus be interpreted as a generalisation of anticipative stochastic bridges. The filtrations generated by these processes are utilised to construct a class of martingales that interpolate between the given target random variables. These so-called filtered arcade martingales (FAMs) are almost-sure solutions to the martingale interpolation problem and reveal an underlying stochastic filtering structure. In the special case of conditionally Markov randomised arcade processes, the dynamics of FAMs are informed by Bayesian updating. The same ideas are applied to filtered arcade reverse-martingales, which are constructed in a similar fashion, using reverse-filtrations of RAPs, instead. Several explicit examples for RAPs and FAMs are provided and simulated. This paper concludes with an outlook on potential connections between FAMs and martingale optimal transport, and related applications.
\end{abstract}
\end{frontmatter}

\section{Introduction}
The problem of constructing martingales which match given marginales has been the subject of many works in probability theory and applications thereof. The origin of the problem dates back to Strassen \cite{Strassen}, with preliminary work by Blackwell \cite{Blackwell}, Sherman \cite{Sherman}, and Stein \cite{Stein}, see \cite{Lecam} for a review. It was shown that a discrete-time martingale $(M_n)_{n\in \N}$ can match a sequence of real probability measures $(\mu_n)_{n\in \N}$ if and only if the measures $(\mu_n)_{n\in \N}$ are convexly ordered, i.e., $M_n \sim \mu_n$ is possible for all $n \in \N$ if and only if $ \int f(x) \d\mu_n(x)$ is an increasing sequence in $n$ for any positive convex function $f$. This result also extends to the continuous-time setting, see Kellerer's theorem \cite{Kellerer}. 

Strassen's and Kellerer's theorems do not provide a method for constructing the matching martingales. Many techniques inspired by Kellerer's setting, where one looks to construct a martingale that matches a given peacock \cite{Hirsch}: Skorokhod's embedding problem, Brownian random time-changes, local volatility models, scaling peacock methods, etc., see \cite{Madan}. Since the target is a dynamical measure, the matching is produced in a weak sense, and the martingale evolves in continuous time. On the other hand, the matching in Strassen's setting can be weak or almost sure, depending on the nature of the targets (measures or random variables), and the martingale is in discrete time. A well-known approach for this task is martingale optimal transport (MOT), see, e.g., \cite{Beiglbock} and \cite{Beiglbock2}. 

Our interest lies at the intersection of Strassen's and Kellerer's settings: how to construct a continuous-time martingale that matches, in law or almost surely, a given finite set of convexly ordered random variables, at pre-specified dates? This is for instance of interest in a financial setting, where knowledge of the prices of vanilla call and put options provides an implied distribution for the underlying asset price at future fixed dates under a no-arbitrage requirement. This is an example of what is called the martingale interpolation problem, since one wishes to interpolate with a continuous-time martingale between the components of a discrete-time martingale. If one put the martingale condition aside, then bridging the gap between a coupling and an interpolating measure is usually a manageable proposition. Imposing the martingale property on the interpolating process is the challenging part of this problem.   

In this paper, inspired by the information-based approach, see \cite{BHM1} and \cite{BHM2}, we construct a class of martingales that match, almost surely, any set of convexly ordered random variables, indexed by an arbitrary set of fixed times. The construction of such interpolating martingales, called \textit{filtered arcade martingales} (FAMs), needs two main ingredients: a discrete-time martingale (for example, a solution to an MOT problem) and an interpolating process that we call the \textit{randomised arcade process} (RAP).

Before we proceed with a section-by-section summary of the content of this paper, we next emphasize the mathematical and scientific contributions of the new approach for stochastic interpolation proposed here. 
\begin{enumerate}
\item A new class of sample-continuous stochastic processes is constructed that interpolate strongly between zeros, that is, omega by omega, called {\it arcade processes}. In general, these are no Markov processes. However, sufficient conditions are established under which an AP is a Markov process. 
\item Strong interpolation between a vector of random variables with arbitrary distribution is obtained by an additive extension of arcade processes---the randomised arcade processes. These processes feature a so-called signal process and an arcade process, the latter being the noise process. RAPs are no Markov processes. However, a special class of RAPs is identified for which a conditionally Markov property is introduced. Conditionally Markov randomised arcade processes is a class of processes that allows for closed-form expressions for the dynamics of the ensuing martingale interpolation.
\item {Filtered arcade martingales}, a class of processes which almost surely interpolate between the random variables (in convex order) linked by the randomised arcade processes, are constructed, whereby the FAMs are adapted to the filtration generated by the associated randomised arcade process. In the case that the randomised arcade process posses the conditional-Markov property, it is shown that the filtered arcade martingale can be expressed in closed form by an application of the Bayes formula. A similar construction method is produced to build filtered arcade reverse-martingales (FARMs), which are constructed in a similar fashion, using reverse-filtrations of RAPs.

\item We establish the connection between filtered arcade martingales and stochastic filtering problems, for which an explicit solution can be found in the case that the underlying randomised arcade process is conditionally Markov with respect to its natural filtration.
\item The almost-sure construction method of the interpolating filtered arcade martingales allows for a level of flexibility to search for the optimal martingale coupling of the random variables in-between which the filtered arcade martingale interpolates. This property of FAMs' leads to the investigation of filtered arcade martingales in the context of martingale optimal transport which we discuss at the end of this paper in the outlook pointing to separate, ongoing research.
\end{enumerate}
Following this introduction, in Section 2, we begin with constructing the {\it arcade processes} (AP). Defined as a functional of a given stochastic process, called the driver, APs are sample-continuous stochastic processes that interpolate between zeros on the whole probability space, i.e., omega by omega. To achieve the interpolation, APs rely on deterministic functions, the interpolating coefficients. APs may be viewed as multi-period anticipative stochastic bridges. We study their properties and focus on Gaussian APs, which play an important role in the construction of FAMs. Starting from a Gauss-Markov driver, we show that it is always possible to construct a Markovian AP by utilising the covariance structure of the driver. The resulting AP from this procedure is called a standard AP. This process is not unique, since one can construct infinitely many Markovian APs that are driven by the same Gauss-Markov process.

Section 3 treats an additive randomisation of APs, that is,  an AP added to a stochastic process that interpolates deterministically between the given random variables. The randomised arcade processes (RAPs) can be thought of as a sum of a signal function and a noise process. By construction, a RAP can match any random variables on the whole probability space at an arbitrary sequence of fixed times, i.e, it is a stochastic interpolator between target random variables in the strong sense. The notion of the Markov property does not suit RAPs in general, since their natural filtration contains the $\sigma$-algebras generated by each previously matched target random variables. For this reason, we introduce a weaker notion, the conditionally-Markov property, and go on to show under which conditions a RAP is conditionally Markov.

In the fourth section, we introduce the filtered arcade martingales (FAMs). A FAM is defined by the conditional expectation of the final target random variable, given the information generated by a RAP and, hence, inherits the filtering framework of the information-based approach that appears in \cite{Macrina}, \cite{BHM1}, \cite{BHM2}: the signal is the final target random variable and can only be observed through a noisy process, the RAP. FAMs are tractable, thanks to the conditionally-Markov property enjoyed by the underlying RAPs, and can be simulated using Bayes' formula. Applying Itô's formula, we derive the stochastic differential equation satisfied by a FAM, which reveals the structure of the underlying innovations process that is adapted to the filtration generated by the associated RAP.  The same ideas are applied to introduce the {\it filtered arcade reverse-martingales} (FARMs), constructed using reverse-filtrations of RAPs. Finally, as an application of this theory, we briefly discuss a similar problem to martingale optimal transport that incorporates noise in the optimisation process, inspired by the entropic regularisation of optimal transport and Schrödinger's problem. This approach selects the martingale coupling that maximises the expectation of the cumulative cost in time between a target random variable and its associated FAM, for a given RAP.

Throughout this paper, unless specified otherwise, let $n\in \N_0$ and consider the collection of fixed dates $\{T_i \in \R \| i=0, 1, \ldots, n\}$ such that $0\leqslant T_0 <T_{1}<T_{2}<\ldots < T_n < \infty$, and introduce the ordered sets 
\begin{eqnarray*}
\{T_0,T_n\}_*&=&\{T_0, T_1, \ldots,T_n\},\\
(T_0,T_n)_*&=&\bigcup\limits_{i=0}^{n-1}(T_i,T_{i+1}).
\end{eqnarray*}
Let $(\Omega, \F, \mathbb{P})$ be a probability space, $\left(D_{t}\right)_{t \in [T_0,T_n]}$ a sample-continuous stochastic process such that for all $c \in \R$, $\P[D_t=c]<1$ for all $t\in (T_0,T_n)_*$, and $X$ an $\R^{n+1}$-valued random vector that is independent of $(D_t)$.

\section{Arcade processes}
We construct stochastic processes on $[T_0,T_n]$, as a functional of $(D_t)$, which match exactly $0$ (for all $\omega \in \Omega$) at the given times $\{T_0,T_n\}_*$. The first step is to introduce deterministic functions called interpolating coefficients. 
\begin{defn} \label{interpolating}
The functions $f_0,f_1,\ldots, f_n$ are interpolating coefficients on $\{T_0,T_n\}_*$ if $f_0, $ $f_1,\ldots, f_n \in C^0\left([T_0,T_n],\R\right), f_i(T_i)=1,$ and $f_i(T_j)=0 \text{ for }i,j=0,\ldots,n,\, i \neq j$.
\end{defn}
We can now give the definition of what we call an \emph{arcade process}.
\begin{defn} \label{AP}
An arcade process (AP) on the partition $\{T_0,T_n\}_*$, denoted $(A_{t}^{(n)})_{t \in [T_0,T_n]}$, is a stochastic process of the form
\begin{equation}
A_{t}^{(n)}:=D_{t}-\sum_{i=0}^n f_i(t)D_{T_i},
\end{equation}
where $f_0,\ldots, f_n$ are  interpolating coefficients on $\{T_0,T_n\}_*$. The process $(D_t)_{t \in [T_0,T_n]}$ is the driver of the AP. We denote by $(\F^A_t)_{t \in [T_0,T_n]}$ the filtration generated by $(A_{t}^{(n)})_{t \in [T_0,T_n]}$.
\end{defn}
We observe that $A_{T_0}^{(n)}=A_{T_1}^{(n)}=\ldots=A_{T_n}^{(n)}=0$ by construction for all $\omega \in \Omega$. That is, $(A_{t}^{(n)})_{t \in [T_0,T_n]}$ strongly interpolates between zeros at times $t \in \{T_0,T_n\}_*$.
\begin{ex} \label{ABB}
For $n=1$, ${f_0(t)= \frac{t-T_1}{T_0-T_1}} $, $ {f_1(t)= \frac{t-T_0}{T_1-T_0}} $, the AP driven by a standard Brownian motion $\left(B_{t}\right)_{t \geqslant 0}$ is the anticipative Brownian bridge on $[T_0,T_1]$,
\begin{equation}
A_{t}^{(1)}= B_t - \frac{T_1-t}{T_1-T_0} B_{T_0} - \frac{t-T_0}{T_1-T_0} B_{T_1}.
\end{equation}
\end{ex}

Two generalisations of the Brownian bridge $(A_{t}^{(1)})_{t \in [T_0,T_1]}$, shown in Example \ref{ABB}, to an $n$-arc arcade process $(A_{t}^{(n)})_{t \in [T_0,T_n]}$ are given in \ref{appendixA}, where their properties are investigated and other examples are provided. We call the generalisation 
\begin{equation}
    A_t^{(n)}=D_t - \sum_{i=0}^n \prod\limits_{k=0,k\neq i}^n\frac{ T_k-t}{ T_k-T_i} D_{T_i},
\end{equation}
driven by a stochastic process $(D_t)$, the Lagrange AP due to the form of the interpolating coefficients.

One might wonder why we choose the timeline $T_0,T_1,\ldots,T_n$ and the interpolating coefficients $f_0,f_1,\ldots,f_n$ to be deterministic rather than stochastic. While the definition of APs is easily extended to a stochastic timeline and stochastic interpolating coefficients, APs are fundamentally designed as a building bloc for the martingales introduced in Section \ref{FAMchapter}. Inspired by several applications, these martingales interpolate between random variables at fixed times. APs play a key role in their construction, which relies entirely on a deterministic timeline and deterministic interpolating coefficients. Furthermore, we want APs to be understood as a natural generalisation of anticipative bridges, which also use deterministic times and interpolating coefficients. 

The expectation and covariance of an AP, when they exist, are fully characterised by the driver and the interpolating coefficients. 
\begin{prop} \label{covform}
If the driver $(D_t)$ has a mean function $\mu_D$, a variance function $\sigma_D^2$, and a covariance function $K_D$, we have:
\begin{align}
&\mu_A(t):=\E [A_{t}^{(n)}]=  \mu_D(t) - \sum\limits_{i=0}^n f_i(t) \mu_D(T_i), \label{expect} \\  
&\sigma_A^2(t):=\var{[A_{t}^{(n)}]} = \sigma_D^2(t) + \sum\limits_{i=0}^n f_i^2(t) \sigma_D^2(T_i)-2 f_i(t) K_D(t,T_i) \nn \\ 
&\hspace{1.55cm}+ 2  \sum\limits_{i=0}^n  \sum\limits_{j=i+1}^{n} f_i(t) f_j(t) K_D(T_i,T_j), \label{variance} \\  
&K_A(s,t):=\cov[{A_{s}^{(n)}, A_{t}^{(n)}}] = K_D(s,t)  - \sum\limits_{i=0}^n f_i(t)K_D(s,T_i) +f_i(s)K_D(t,T_i)   \nn \\ 
&\hspace{1.55cm}+ \sum\limits_{i=0}^n\sum\limits_{j=0}^{n} f_i(s)f_j(t) K_D(T_i, T_j). \label{cov}
\end{align}
\end{prop}
\begin{proof}
Eq. (\ref{expect}) follows from the linearity property of the expectation and Eq. (\ref{variance}) follows from Eq. (\ref{cov}). Hence, it is enough to prove Eq. (\ref{cov}). 
\begin{align}
    &\cov[{A_{s}^{(n)}, A_{t}^{(n)}}]\notag\\ 
    &= \cov\left[ D_{s}-\sum_{i=0}^n f_i(s)D_{T_i},  D_{t}-\sum_{i=0}^n f_i(t)D_{T_i} \right] \\
    &= K_D(s,t) - \cov\left[ D_{s},  \sum_{i=0}^n f_i(t)D_{T_i} \right] - \cov\left[ \sum_{i=0}^n f_i(s)D_{T_i},  D_{t} \right] \notag\\
    &\hspace{6.71cm}+ \cov\left[\sum_{i=0}^n f_i(s)D_{T_i}, \sum_{i=0}^n f_i(t)D_{T_i} \right]  \\ 
    &= K_D(s,t) - \sum_{i=0}^n f_i(t) K_D(s,T_i) - \sum_{i=0}^n f_i(s)K_D(T_i,t) + \sum_{i=0}^n\sum_{j=0}^n  f_i(s)f_j(t) K_D(T_i,T_j) \\
    &=K_D(s,t)  - \sum\limits_{i=0}^n f_i(t)K_D(s,T_i) +f_i(s)K_D(t,T_i)+ \sum\limits_{i=0}^n\sum\limits_{j=0}^{n} f_i(s)f_j(t) K_D(T_i, T_j).
\end{align}
\end{proof}
Markov APs will play a crucial role in the construction of filtered arcade martingales. A useful property of these processes is the following.
\begin{prop} \label{ind}
Let $(A_t^{(n)})_{t \in [T_0,T_n]}$ be an AP on $\{T_0,T_n\}_*$ that is Markov with respect to $(\F^A_t)_{t \in [T_0,T_n]}$. If $t>T_i$ for $T_i \in \{T_0,T_n\}_*$, then $A_t^{(n)}$ is independent of $\F^A_{T_{i}}$.
\end{prop}
\begin{proof}
    Since $(A_t^{(n)})$ is $0$ on $\{T_0,T_n\}_*$, the proof follows immediately from the Markov property.
\end{proof}
In particular, if $(A_t^{(n)})$ is Markov, and it has a covariance function, then $\cov (A_s^{(n)}, A_t^{(n)})= 0 $ whenever $s\in[T_i,T_{i+1}]$, $t\in[T_j,T_{j+1}]$ and $j \neq i$.

If APs are considered being noise processes, it is natural to study their Gaussian subclass. Therefore, we introduce terminology that avoids confusion when referring to the driver of an AP that is Gaussian, since the driver of an AP is not uniquely determined. Consider the process $Y_t=B_t + tY$, where $(B_t)$ is a Gaussian process, $Y$ a non-Gaussian random variable, and the AP $(A_t^{(1)})_{t \in [0,1]}$ on $\{0,1\}$ is given by
\begin{align}
    A_t^{(1)}&=Y_t - (1-t) Y_{0} - t Y_{1} \nn \\
    &= B_t + tY - (1-t) B_0 - t (B_1 + Y) \nn \\
    &= B_t - (1-t) B_0 - tB_1.
\end{align}
Both $(Y_t)$ and $(B_t)$ can be considered the driver of $(A_t^{(1)})$, which is a Gaussian process, but only one driver is Gaussian.
\begin{defn}
An AP $(A_t^{(n)})_{t \in [T_0,T_n]}$ is said to be a Gaussian AP if its driver $(D_t)$ is a Gaussian stochastic process on $[T_0,T_n]$.
\end{defn}
Under this definition, we know that $(Y_t)$ cannot be referred to as the driver of the Gaussian AP $(A_t^{(1)})$ defined above, since it is not Gaussian. Next, we give a first result on the Markov property of Gaussian APs.
\begin{thm} \label{thmmarkov}
Let $(A_t^{(n)})_{t \in [T_0,T_n]}$ be a Gaussian AP on $\{T_0,T_n\}_*$ with covariance function $K_A$. Then $(A_t^{(n)})$ is Markov with respect to its own filtration if and only if for all $ (r,s,t) \in (T_0,T_n)_*^3$ satisfying $r \leqslant s < t$, there exists $a(s,t) \in \R_0$ such that
\begin{equation} \label{form1}
    K_A(r,t)=\left\{
    \begin{array}{cc}
         0, &\mbox{ if } t\in [T_i,T_{i+1}], r\in [T_j,T_{j+1}], i\neq j,   \\ \\
          K_A(r,s)a(s,t),  &\mbox{ otherwise}.
    \end{array}
    \right. 
\end{equation}
\end{thm}
\begin{proof}
Suppose $K_A$ is of the form Eq. (\ref{form1}), we shall show that $(A_t^{(n)})$ is Markov. Let $k>1$ and $(s_1 , s_2 , \ldots , s_k , t ) \in [T_0,T_n] ^{k+1}$ such that $s_1 < s_2 < \ldots < s_k < t$. Then, $(A_t^{(n)})$ is Markov if and only if
\begin{equation}
\mathbb{P}\left[A_{t}^{(n)} \in \cdot \mid   A_{s_1}^{(n)}, \ldots, A_{s_{k}}^{(n)}\right] = \mathbb{P}\left[A_{t}^{(n)} \in \cdot \mid  A_{s_k}^{(n)} \right],
\end{equation}
see Theorem 1.12 in \cite{Lipster}. Since $K_A(s,t)=0$ unless $s$ and $t$ are in the same sub-interval, we can assume without loss of generality that $(s_1 , s_2 , \ldots , s_k , t ) \in  (T_m,T_{m+1})^{k+1}$ for a fixed $m \in \{0,\ldots,n-1\}$. Define
\begin{equation}
    \Delta_q=\sum_{i=1}^{k}  c_{i,q}  A_{s_i}^{(n)}, \quad q=1,\ldots,k-1,
\end{equation}
where the coefficients $(c_{i,q})$ are real numbers satisfying
\begin{equation} \label{coeffrel}
\sum\limits_{i=1}^{k}  c_{i,q}  K_A(s_i,t)=0 \quad \mbox{ and } \quad \mathbb \det
\begin{pmatrix}
c_{1,1}  & \ldots & c_{1,k-1}  \\
\vdots &  & \vdots  \\
c_{k-1,1}  & \ldots & c_{k-1,k-1}   \\
\end{pmatrix} \neq 0.
\end{equation}
We notice that 
\begin{equation}
    \sigma (A_{s_1}^{(n)}, \ldots, A_{s_{k}}^{(n)}) = \sigma (\Delta_1, \ldots, \Delta_{k-1}, A_{s_k}^{(n)}) \iff \mathbb \det
\begin{pmatrix}
c_{1,1}  & \ldots & c_{1,k-1}  \\
\vdots &  & \vdots  \\
c_{k-1,1}  & \ldots & c_{k-1,k-1}   \\
\end{pmatrix} \neq 0.
\end{equation}
It remains to be shown that $A_{t}^{(n)} \ind (\Delta_1, \ldots, \Delta_{k-1}) \text{ and } A_{s_k}^{(n)} \ind (\Delta_1, \ldots, \Delta_{k-1}).$ Equivalently, because we are treating the Gaussian case, we shall show that $ \cov[A_{t}^{(n)}, \\ \Delta_q]=0 \mbox{ and } \cov[A_{s_k}^{(n)},\Delta_q]=0 \mbox{ for } q=1,\ldots,k-1.$
Expanding these covariances, we get 
\begin{equation}
    \cov[A_{t}^{(n)},\Delta_q]= 
    \sum_{i=1}^{k}  c_{i,q} K_A(s_i,t)=0,
 \end{equation}
which is guaranteed by Eq. (\ref{coeffrel}), and 
 \begin{equation}
    \cov[A_{s_k}^{(n)},\Delta_q]= 
    \sum_{i=1}^{k}  c_{i,q} K_A(s_k,s_i)=0.
 \end{equation}
The first equation implies the second because $K_A(s_i,t) = 0 \implies K_A(s_k,s_i) =0$ by Eq. (\ref{form1}), and, when $K_A(s_i,t) \neq 0$, we have $K_A(s_i,t) = a(s_k,t)K_A(s_k,s_i)$, so that
\begin{equation}
\sum\limits_{i=1}^{k} c_{i,q} K_A(s_i,t)=0 \implies \sum\limits_{i=1}^{k}  c_{i,q} a(s_k,t)K_A(s_k,s_i)=0 \implies \sum\limits_{i=1}^{k}  c_{i,q} K_A(s_k,s_i)=0.
\end{equation}
This concludes one implication.

For the converse, suppose without loss of generality that the driver has mean $0$. We observe that $K_A(x,y)=0$ if $x\in [T_i,T_{i+1}], y\in [T_j,T_{j+1}], i\neq j$. Let $(r,s,t) \in (T_i,T_{i+1})^3$ be such that $r\leqslant s <t$. Since $(A_{t}^{(n)})$ is Gaussian, we have 
\begin{equation}
    \E[A_{t}^{(n)} \mid A_{s}^{(n)} ] = \frac{K_A(s,t)}{K_A(s,s)}A_{s}^{(n)}.
\end{equation}
Using the Markov property of $(A_t^{(n)})$, we get 
\begin{equation} \label{blabla}
    K_A(r,t) = \E [ A_{r}^{(n)}A_{t}^{(n)}   ] = \E [ \E [ A_{r}^{(n)}A_{t}^{(n)}    \mid A_{s}^{(n)} ] ] = \E[\E [A_{r}^{(n)} \mid A_{s}^{(n)} ] \E [ A_{t}^{(n)} \mid A_{s}^{(n)} ] ],
\end{equation}
which implies 
\begin{equation}
    K_A(r,t) = \E \left[\frac{K_A(s,t)}{K_A(s,s)}A_{s}^{(n)} \frac{K_A(r,s)}{K_A(s,s)}A_{s}^{(n)} \right] = \frac{K_A(s,t)K_A(r,s)}{K_A(s,s)}.
\end{equation}
Hence $\displaystyle{a(s,t) = \frac{K_A(s,t)}{K_A(s,s)}}$.
\end{proof}
We can simplify the statement of Theorem \ref{thmmarkov} in the following way.
\begin{coro} \label{coromarkov}
A Gaussian AP $(A_t^{(n)})_{t \in [T_0,T_n]}$ is Markov with respect to its own filtration if and only if there exist real functions $A_1 :  [T_0,T_n]  \rightarrow \R $ and $A_2  : [T_0,T_n] \rightarrow \R $ such that 
\begin{equation} 
    K_A(s,t)=\sum_{i=0}^{n-1} A_1(\min(s,t))A_2(\max(s,t)) \1_{(T_i,T_{i+1})}(s,t).
\end{equation}
\end{coro}
\begin{proof}
If
\begin{equation}
    K_A(s,t)=\sum_{i=0}^{n-1} A_1(\min(s,t))A_2(\max(s,t)) \1_{(T_i,T_{i+1})}(s,t),
\end{equation}
then $\forall \, (r,s,t) \in (T_0,T_n)_*^3$ such that $r \leqslant s < t$, one has
\begin{equation}
    K_A(r,t)=\left\{
    \begin{array}{cc}
         0, &\mbox{ if } t\in [T_i,T_{i+1}], r\in [T_j,T_{j+1}], i\neq j,   \\ \\
          K_A(r,s) \frac{A_2(t)}{A_2(s)},  &\mbox{ otherwise}.
    \end{array}
    \right. 
\end{equation}
Hence, $(A_t^{(n)})$ is Markov. 

Conversely, suppose that $(A_t^{(n)})$ is Markov, and let $T_m \in \{T_0,T_{n-1}\}_*$. Then, 
\begin{equation}
    \frac{K_A(r,t)}{K_A\left(r,\frac{T_m + T_{m+1}}{2}\right)} = \frac{K_A\left(\frac{T_m + T_{m+1}}{2},t\right)}{K_A\left(\frac{T_m + T_{m+1}}{2},\frac{T_m + T_{m+1}}{2}\right)} 
\end{equation}
for any $(r,t)\in (T_m,T_{m+1})^2$ such that $r<\frac{T_m + T_{m+1}}{2} < t$ by Theorem \ref{thmmarkov}. Hence, if
\begin{equation}
    A_1(x) = \sum_{i=0}^{n-1} K_A\left(x, \frac{T_i + T_{i+1}}{2}\right) \1_{(T_i,T_{i+1})} (x),
\end{equation}
 and 
 \begin{equation}
      A_2(x) = \sum_{i=0}^{n-1} \frac{K_A\left(\frac{T_i + T_{i+1}}{2}, x\right)} {K_A\left( \frac{T_i + T_{i+1}}{2},  \frac{T_i + T_{i+1}}{2}\right )} \1_{(T_i,T_{i+1})} (x),
 \end{equation}
we have 
\begin{equation}
    K_A(s,t)=\sum_{i=0}^{n-1} A_1(\min(s,t))A_2(\max(s,t)) \1_{(T_i,T_{i+1})}(s,t).
\end{equation}
\end{proof}
If $T_i \in \{T_0,T_{n-1}\}_*$ and $(s,t) \in (T_i,T_{i+1})^2$ such that $s<t$, then $\lim\limits_{s\rightarrow T_i} A_1(s)=0$ or $\lim\limits_{t\rightarrow T_{i+1}}  A_2(t)=0$ by continuity of $K_A$, and $A_1/A_2(t)$ is positive and non-decreasing on each interval $(T_i,T_{i+1})$, since $K_A$ is a covariance function.

Starting from a Gauss-Markov driver $(D_t)$, it is always possible to construct a Markovian $(A_t^{(n)})$ by applying the following procedure.
\begin{thm}
For any Gauss-Markov driver $(D_t)_{t\in [T_0,T_n]}$, there exists an AP $(A_t^{(n)})_{t \in [T_0,T_n]}$, driven by $(D_t)$, that is Markovian.
\end{thm}
\begin{proof}
Let $T_m \in \{T_0,T_{n-1}\}_*$ and $(s,t)\in (T_m,T_{m+1})^2$ be such that $s<t$. Recall Definition \ref{interpolating} and choose the interpolating coefficients $f_0,\ldots, f_n$ so that they satisfy
\begin{equation}
\begin{pmatrix}
K_D(T_0,T_0) & \ldots & K_D(T_0,T_n)\\
\vdots & \ddots & \vdots \\
K_D(T_n,T_0) & \ldots & K_D(T_n,T_n)
\end{pmatrix}
\begin{pmatrix}
f_0(\cdot)\\
\vdots \\
f_n(\cdot)
\end{pmatrix} = 
\begin{pmatrix}
K_D(\cdot,T_0)\\
\vdots \\
K_D(\cdot,T_n)
\end{pmatrix}.
\end{equation}
Then,
\begin{equation} \label{rel}
    \sum_{j=0}^n f_j(\cdot) K_D(T_i,T_j) - K_D(\cdot,T_i)=0, \quad \forall i=0,\ldots, n,
\end{equation}
which implies
\begin{equation}
    K_A(x,y)= K_D(x,y) - \sum_{i=0}^n f_i(x) K_D(y,T_i) = K_D(x,y) - \sum_{i=0}^n f_i(y) K_D(x,T_i),
\end{equation}
$\forall (x,y) \in [T_0,T_n]^2$. Let $T_m \in \{T_0,T_{n-1}\}_*$ and $(s,t)\in (T_m,T_{m+1})^2$ such that $s<t$. Recalling that $(D_t)$ is Gauss-Markov, there exist two functions, $H_1: [T_m,T_{m+1}] \rightarrow \R $ and $H_2 : [T_m,T_{m+1}] \rightarrow \R$, such that $K_D(s,t)=H_1(s)H_2(t)$. So, one can write
\begin{align}
K_A(s,t)&=H_1(s)H_2(t) - \sum_{i=0}^m f_i(s) H_1(T_i)H_2(t) - \sum_{i=m+1}^n f_i(s) H_1(t)H_2(T_i) \nn \\ 
&=H_2(t)\left(H_1(s)-  \sum_{i=0}^m f_i(s) H_1(T_i)\right) - H_1(t) \sum_{i=m+1}^n f_i(s) H_2(T_i)\nn \\
&= \left(\sum_{i=m+1}^n f_i(s) H_2(T_i)\right) \left( \frac{H_1(T_{m+1})}{H_2(T_{m+1})} H_2(t) - H_1(t)  \right)
\end{align}
where we used Eq. (\ref{rel}) with $i=m+1$. Hence,
\begin{align}
    A_1(x)\1_{(T_m,T_{m+1})}(x) &= \left(\sum_{i=m+1}^n f_i(x) H_2(T_i)\right) \1_{(T_m,T_{m+1})}(x), \label{A1} \\
    A_2(x) \1_{(T_m,T_{m+1})}(x) &= \left( \frac{H_1(T_{m+1})}{H_2(T_{m+1})} H_2(x) - H_1(x)  \right) \1_{(T_m,T_{m+1})}(x), \label{A2}
\end{align}
and $(A_t^{(n)})$ is Markov by Corollary \ref{coromarkov}.
\end{proof}
We recall that the property is imposed on all drivers $(D_t)$ that for all $c \in \R$, $\P[D_t=c]<1$ for all $t\in (T_0,T_n)_*$. If we extend this property to $[T_0,T_n]$ instead, then the above construction of a Markovian AP becomes explicit.
\begin{coro} \label{coroexplicit}
If $(D_t)$ is a Gauss-Markov process such that for all $c \in \R$, $\P[D_t=c]<1$ for all for $t\in [T_0,T_n]$, with $K_D(s,t)=H_1(\min(s,t)) H_2 (\max(s,t))$ for all $(s,t) \in [T_0,T_n]^2$ for some real functions $H_1$ and $H_2$, then the solution to Eq. (\ref{rel}) is given by 
\begin{align}
    f_0(x)&= \frac{H_1(T_{1})H_2(x)-H_1(x)H_2(T_1)}{H_1(T_{1})H_2(T_0)-H_1(T_0)H_2(T_1)} \1_{[T_0, T_1]} (x),\\
    f_i(x)&=\frac{H_1(x)H_2(T_{i-1})-H_1(T_{i-1})H_2(x)}{H_1(T_{i})H_2(T_{i-1})-H_1(T_{i-1})H_2(T_{i})} \1_{[T_{i-1}, T_i]}  (x) \nn \\ 
    & \quad + \frac{H_1(T_{i+1})H_2(x)-H_1(x)H_2(T_{i+1})}{H_1(T_{i+1})H_2(T_{i})-H_1(T_{i})H_2(T_{i+1})} \1_{(T_{i}, T_{i+1}]} (x) , \quad \text{for } i=1,\ldots, n-1,  \\
    f_n(x)&= \frac{H_1(x)H_2(T_{n-1})-H_1(T_{n-1})H_2(x)}{H_1(T_n)H_2(T_{n-1})-H_1(T_{n-1})H_2(T_n)} \1_{(T_{n-1}, T_n]} (x).
\end{align}
\end{coro}
\begin{proof}
Let ($T_{m^-},T_m,T_{m^+}) \in \{T_0,T_{n-1}\}_*^3$ such that $T_{m^-}\leqslant T_m<T_{m^+}$, and $x \in (T_m,T_{m+1})$. Then,
\begin{align}
    &\sum_{j=0}^n f_j(x) K_D(T_{m^-},T_j) \nn \\
    &= f_m(x) H_1(T_{m^-}) H_2 (T_m) + f_{m+1}(x) H_1(T_{m^-}) H_2 (T_m) \nn \\
    &= \frac{H_1(T_{m+1})H_2(x)-H_1(x)H_2(T_{m+1})}{H_1(T_{m+1})H_2(T_{m})-H_1(T_{m})H_2(T_{m+1})} H_1(T_{m^-}) H_2 (T_m) \nn \\
    &\hspace{1.1cm}+ \frac{H_1(x)H_2(T_{m})-H_1(T_{m})H_2(x)}{H_1(T_{m+1})H_2(T_{m})-H_1(T_{m})H_2(T_{m+1})}  H_1(T_{m^-}) H_2 (T_{m+1}) \nn \\ \nn \\
    &= \frac{H_1(x) (H_1(T_{m^-})H_2(T_m) H_2 (T_{m+1}) -  H_1(T_{m^-}) H_2 (T_m)   H_2(T_{m+1}))}{H_1(T_{m+1})H_2(T_{m})-H_1(T_{m})H_2(T_{m+1})} \nn \\
    &\hspace{1.1cm}+\frac{H_2(x)(H_1(T_{m+1}) H_1(T_{m^-}) H_2 (T_m) - H_1(T_{m}) H_1(T_{m^-}) H_2 (T_{m+1}) ) }{H_1(T_{m+1})H_2(T_{m})-H_1(T_{m})H_2(T_{m+1})}\nn \\ \nn \\
    &=H_1(T_{m^-}) H_2(x) = K_D(T_{m^-},x)
\end{align}
The same argument applies to show $\displaystyle \sum_{j=0}^n f_j(x) K_D(T_{m^+},T_j) = K_D(x,T_{m^+}) $. Hence,
\begin{equation}
\begin{pmatrix}
K_D(T_0,T_0) & \ldots & K_D(T_0,T_n)\\
\vdots & \ddots & \vdots \\
K_D(T_n,T_0) & \ldots & K_D(T_n,T_n)
\end{pmatrix}
\begin{pmatrix}
f_0(\cdot)\\
\vdots \\
f_n(\cdot)
\end{pmatrix} = 
\begin{pmatrix}
K_D(\cdot,T_0)\\
\vdots \\
K_D(\cdot,T_n)
\end{pmatrix}.
\end{equation}
Next, recalling the steps following Eq. (\ref{rel}) concludes the proof.
\end{proof}
\begin{rem}
Notice that if there is a $c \in \R$ such that $\P[D_t=c]=1$ for some $t\in \{T_0,T_n\}_*$, one can still use the construction of the above Markovian AP by removing all the rows and columns of zeros in the matrix 
\begin{equation}
\begin{pmatrix}
K_D(T_0,T_0) & \ldots & K_D(T_0,T_n)\\
\vdots & \ddots & \vdots \\
K_D(T_n,T_0) & \ldots & K_D(T_n,T_n)
\end{pmatrix}.
\end{equation}
This will determine all the interpolating coefficients but the ones that are associated with the times $T_i \in \{T_0,T_n\}_*$ where $\P[D_{T_i}=c]=1$. Since the driver already matches $c$ at these times, their associated interpolating coefficients can be freely chosen without affecting the Markov property of the AP. This is illustrated by taking the Brownian motion with initial value $1$ as a driver with $T_0=0$. Then $f_0$ does not matter for the Markov property of the AP since it is multiplied by the deterministic function with constant value $1$ in the expression of the AP.
\end{rem}
\begin{rem}
Inserting the interpolating coefficients given in Corollary \ref{coroexplicit} into Eqs. (\ref{A1}) and (\ref{A2}) yields
\begin{align}
    A_1(x)\1_{(T_m,T_{m+1})}(x) &=\frac{H_1(x)H_2(T_{m})-H_1(T_{m})H_2(x)}{H_1(T_{m+1})H_2(T_{m})-H_1(T_{m})H_2(T_{m+1})} H_2(T_{m+1}) \1_{(T_m,T_{m+1})}(x),\\
    A_2(x) \1_{(T_m,T_{m+1})}(x) &= \left( \frac{H_1(T_{m+1})}{H_2(T_{m+1})} H_2(x) - H_1(x)  \right) \1_{(T_m,T_{m+1})}(x),
\end{align}
for $m=1,\ldots, n-1$.
\end{rem}
This method to construct Markovian APs presented here is not unique, but is natural, for the covariance structure of the driver completely and uniquely determines the interpolating coefficients. The resulting APs are called standard. They are anticipative representations of processes whose law is determined by the conditioned driver, i.e., if $(A_t^{(n)})$ is standard, then $\P[A_t^{(n)} \in \cdot \,]=\P[D_t \in \cdot \mid D_{T_0}=0, D_{T_1}=0, \ldots, D_{T_n}=0]$.
\begin{defn} \label{standard}
A standard AP $(A_t^{(n)})_{t \in [T_0,T_n]}$ is an AP, driven by a Gauss-Markov process $(D_t)$, of the form
\begin{equation}
A_{t}^{(n)}  = \left\{
    \begin{array}{ll}
         D_t - \frac{H_1(T_{1})H_2(t)-H_1(t)H_2(T_1)}{H_1(T_{1})H_2(T_0)-H_1(T_0)H_2(T_1)}  D_{T_0} \\ \\ \quad \quad \quad \quad \quad \quad \quad
         - \frac{H_1(t)H_2(T_{0})-H_1(T_0)H_2(t)}{H_1(T_{1})H_2(T_{0})-H_1(T_0)H_2(T_{1})} D_{T_1} & \mbox{for } t \in [T_0,T_{1}), \\ \\
         D_t - \frac{H_1(T_2)H_2(t)-H_1(t)H_2(T_2)}{H_1(T_2)H_2(T_1)-H_1(T_1)H_2(T_2)}  D_{T_1} \\ \\ \quad \quad \quad \quad \quad \quad \quad
         - \frac{H_1(t)H_2(T_{1})-H_1(T_{1})H_2(t)}{H_1(T_2)H_2(T_{1})-H_1(T_{1})H_2(T_{2})} D_{T_2} & \mbox{for } t \in [T_1,T_{2}), \\ \\
         \, \vdots \\ \\
         D_t - \frac{H_1(T_n)H_2(t)-H_1(t)H_2(T_n)}{H_1(T_n)H_2(T_{n-1})-H_1(T_{n-1})H_2(T_n)}  D_{T_{n-1}}  & \\ \\ \quad \quad \quad \quad \quad \quad \quad  -\frac{H_1(t)H_2(T_{n-1})-H_1(T_{n-1})H_2(t)}{H_1(T_n)H_2(T_{n-1})-H_1(T_{n-1})H_2(T_{n})} D_{T_n} & \mbox{for } t \in [T_{n-1},T_{n}],
    \end{array}
\right.
\end{equation}
where $H_1$ and $H_2$ are the real functions that appear in the covariance function of the driver: $K_D(s,t)=H_1(\min(s,t)) H_2 (\max(s,t))$ for all $ (s,t) \in [T_0,T_n]^2$.
\end{defn}
\begin{ex}
If $(D_t)$ is an Ornstein-Uhlenbeck process with parameters $\theta>0, \sigma>0, \mu \in \R$ and initial value $d_0\in \R$, that is, the solution to 
\begin{equation}
    \rd D_{t}=\theta\left(\mu-D_{t}\right) \rd t+\sigma \rd W_{t}, \quad D_0=d_0,
\end{equation}
then $K_D(s,t)=\sigma^2/(2 \theta) \exp{[\theta \min(s,t)]} \exp{[-\theta \max(s,t)]}$. The standard AP driven by $(D_t)$ is 
\begin{equation}
A_{t}^{(n)}  = \left\{
    \begin{array}{ll}
         D_t - \frac{\exp{[\theta (T_1-t)]}-\exp{[-\theta (T_1-t)]}}{\exp{[\theta (T_1-T_0)]}-\exp{[-\theta (T_1-T_0)]}}  D_{T_0}  \\ \\ \quad \quad \quad \quad \quad \quad \quad -\frac{\exp{[\theta (t-T_0)]}-\exp{[-\theta (t-T_0)]}}{\exp{[\theta (T_1-T_0)]}-\exp{[-\theta (T_1-T_0)]}}  D_{T_1} \quad &\mbox{for } t \in [T_0,T_{1}), \\ \\
         D_t -\frac{\exp{[\theta (T_2-t)]}-\exp{[-\theta (T_2-t)]}}{\exp{[\theta (T_2-T_1)]}-\exp{[-\theta (T_2-T_1)]}}  D_{T_1} \\ \\ \quad \quad \quad \quad \quad \quad \quad - \frac{\exp{[\theta (t-T_1)]}-\exp{[-\theta (t-T_1)]}}{\exp{[\theta (T_2-T_1)]}-\exp{[-\theta (T_2-T_1)]}} D_{T_2} \quad &\mbox{for } t \in [T_1,T_{2}), \\ \\
         \, \vdots \\ \\ 
         D_t - \frac{\exp{[\theta (T_n-t)]}-\exp{[-\theta (T_n-t)]}}{\exp{[\theta (T_n-T_{n-1})]}-\exp{[-\theta (T_n-T_{n-1})]}} D_{T_{n-1}} \\ \\
         \quad \quad \quad \quad \quad \quad \quad-\frac{\exp{[\theta (t-T_{n-1})]}-\exp{[-\theta (t-T_{n-1})]}}{\exp{[\theta (T_n-T_{n-1})]}-\exp{[-\theta (T_n-T_{n-1})]}} D_{T_n} \quad &\mbox{for } t \in [T_{n-1},T_{n}].
    \end{array}
\right.
\end{equation}
\end{ex}
There are infinitely many Markovian APs driven by the same Gauss-Markov driver. In general, when $T_m \in \{T_0,T_{n-1}\}_*$ and $(s,t)\in (T_m,T_{m+1})^2$ with $s<t$, we have 
\begin{align} 
    K_A(s,t)=&\left(H_1(s)-\sum_{i=0}^m f_i(s) H_1(T_i) \right)\left(H_2(t)-\sum_{i=m+1}^n f_i(t) H_2(T_i) \right) \nn \\ 
    &- \left(\sum_{i=m+1}^{n} f_i(s) H_2(T_i) \right) \left(H_1(t)-\sum_{i=0}^{m+1} f_i(t) H_1(T_i) \right)\nn \\
    &+ \left(\sum_{i=m+1}^{n} f_i(s) H_1(T_i) \right)\sum_{i=m+2}^{n} f_i(t) H_2(T_i)  \nn \\
    &- \left(H_2(s)-\sum_{i=0}^m f_i(s) H_2(T_i) \right)\sum_{i=0}^{m} f_i(t) H_1(T_i) ,  \label{KA}\
\end{align}
where we use the convention that an empty sum is equal to zero. There are as many Markovian APs driven by $(D_t)$ as there are ways to express the covariance function $K_A(s,t)$ as a product of a function of $s$ and a function of $t$.
\begin{ex} \label{nonstarcade}
If $(D_t)$ is a standard Brownian motion, applying Eq. (\ref{rel}) to find appropriate interpolating coefficients yields the stitched Brownian AP, see Example \ref{stitched}. However, there are other Markovian APs driven by standard Brownian motion. For instance, in the two-period case, we may choose
\begin{align}
f_0(t)&= \frac{T_1-t}{T_1-T_0} \1_{\left[T_0,\frac{T_1+T_2}{2}\right]}(t) - \frac{T_2-t}{T_1-T_0} \1_{\left(\frac{T_1+T_2}{2},T_2\right]}(t), \label{f0} \\
f_1(t)&= \frac{t-T_0}{T_1-T_0} \1_{[T_0,T_1]}(t) +  \frac{T_2-t}{T_2-T_1} \1_{(T_1,T_2]}(t), \\
f_2(t)&= \frac{t-T_1}{T_2-T_1} \1_{\left[T_1,T_2\right]}(t).
\end{align}
It is straightforward to verify that these are interpolating coefficients. Let $(A_t^{(2)})_{t \in [T_0,T_2]}$ be the AP constructed with these interpolating coefficients and driven by a standard Brownian motion. Then,
\begin{align}
&K_A(s,t)\nn \\
&= \frac{(\min(s,t) -T_0)(T_1-\max(s,t))}{T_1-T_0}\1_{(T_0,T_1)}(s,t) \nn \\ \nn \\
&+ (\min(s,t)-T_1)\left(\frac{T_0 (\max(s,t) + T_0) - 3 T_0 T_1 + T_1^2}{(T_1-T_0)^2} + \frac{T_1-\max(s,t)}{T_2-T_1} \right)\nn \\
&\hspace{8.5cm} \times \1_{\left(T_1,\frac{T_1+T_2}{2}\right]}(s,t)\nn \\ \nn \\
&+ \frac{(\min(s,t) -T_1)(T_2-\max(s,t))}{T_2-T_1} \frac{T_0^2 + T_1^2 + T_0 (T_2 -3 T_1)}{(T_1-T_0)^2}\nn \\ &\hspace{6.5cm}\times\1_{\left(T_1,\frac{T_1+T_2}{2}\right]}(\min(s,t)) \1_{\left(\frac{T_1+T_2}{2},T_2\right]}(\max(s,t))\nn \\ \nn \\
&+ \frac{\min(s,t) (T_0^2 + T_1^2 - T_0 (T_1 + T_2))  - (T_1 - T_0)^2 T_1 + T_0 T_2 (T_2 - T_1)}{(T_1-T_0)^2} \frac{T_2-\max(s,t)}{T_2-T_1} \nn \\ 
&\hspace{8.5cm}\times\1_{\left(\frac{T_1+T_2}{2},T_2\right]}(s,t).
\end{align}
Thus,
\begin{align}
    A_1(x)&= (x -T_0) \1_{\left[T_0,T_1\right]}(x) + (x -T_1) \1_{\left(T_1,\frac{T_1+T_2}{2}\right]}(x)  \\
    &+ \frac{x (T_0^2 + T_1^2 - T_0 (T_1 + T_2))  - (T_1 - T_0)^2 T_1 + T_0 T_2 (T_2 - T_1)}{T_0^2 + T_1^2 + T_0 (T_2 -3 T_1)}\1_{\left(\frac{T_1+T_2}{2},T_2\right]}(x),\nn
\end{align}
\begin{align}
    A_2(x)= \frac{T_1-x}{T_1-T_0}\1_{[T_0,T_1]}(x) &+ \left(\frac{T_0 (x + T_0) - 3 T_0 T_1 + T_1^2}{(T_1-T_0)^2} + \frac{T_1-x}{T_2-T_1} \right)\1_{\left(T_1,\frac{T_1+T_2}{2}\right]}(x) \nn \\
    &+\frac{(T_2-x)}{T_2-T_1} \frac{T_0^2 + T_1^2 + T_0 (T_2 -3 T_1)}{(T_1-T_0)^2}\1_{\left(\frac{T_1+T_2}{2},T_2\right]}(x).
\end{align}
Recalling Theorem \ref{thmmarkov}, this AP is Markovian. This process is a slight modification of the stitched Brownian AP, where $f_0$ is not $0$ on $(T_1,T_2)$. Hence, $B_{T_0}$ still has a negative influence on the paths of the AP on $(T_1,T_2)$, since $f_0$ is negative on $(T_1,T_2)$. Similarly, we can modify the stitched Brownian arcade by setting $f_2\neq 0$ on $(T_0,T_1)$: 
\begin{align}
f_0(t)&= \frac{T_1-t}{T_1-T_0} \1_{[T_0,T_1]}(t),\\
f_1(t)&= \frac{t-T_0}{T_1-T_0} \1_{[T_0,T_1]}(t) +  \frac{T_2-t}{T_2-T_1} \1_{(T_1,T_2]}(t),\\
f_2(t)&=\frac{T_0-t}{T_2-T_1}\1_{\left[T_0,\frac{T_0+T_1}{2}\right]}(t) + \frac{t-T_1}{T_2-T_1} \1_{\left(\frac{T_0+T_1}{2},T_2\right]}(t). \label{f2}
\end{align}
For this choice of interpolating coefficients, $B_{T_2}$ has an influence on the paths of the AP on $(T_0,T_1)$. Combining the interpolating coefficients $f_0$ from Eq. (\ref{f0}) and $f_2$ from Eq. (\ref{f2}), we can find an interpolating coefficient $f_1$, such that a Brownian AP with these interpolating coefficients is Markovian:
\begin{align}
f_0(t)&= \frac{T_1-t}{T_1-T_0} \1_{\left[T_0,\frac{T_1+T_2}{2}\right]}(t) + \frac{t-T_2}{T_1-T_0} \1_{\left(\frac{T_1+T_2}{2},T_2\right]}(t),\\   \nn \\
f_1(t) &= \frac{(t-T_0)(T_2-T_0)}{(T_2-T_1)(T_1-T_0)}\1_{\left[T_0,\frac{T_0+T_1}{2}\right]}(t)  \nn \\
&\quad + \frac{T_1^2-T_0T_2 + t(T_0-2T_1+T_2)}{(T_1-T_0)(T_2-T_1)} \1_{\left(\frac{T_0+T_1}{2},T_1\right]}(t)  \nn  \\  \nn \\
&\quad + \frac{T_1 (T_0 (T_1 - 2 T_2) + T_1 T_2) + t(T_0T_2-T_1^2)}{(T_2-T_1)(T_1-T_0)T_1} \1_{\left(T_1,\frac{T_1+T_2}{2}\right]}(t)\nn \\  \nn \\
&\quad +\frac{(T_2-t)(T_1^2+T_0(T_2-2T_1))}{(T_2-T_1)(T_1-T_0)T_1}\1_{\left(\frac{T_1+T_2}{2},T_2\right]}(t), \\  \nn \\
f_2(t)&=\frac{T_0-t}{T_2-T_1} \1_{\left[T_0,\frac{T_0+T_1}{2}\right]}(t)+ \frac{t-T_1}{T_2-T_1} \1_{\left(\frac{T_0+T_1}{2},T_2\right]}(t).
\end{align}
The key to building non-standard APs is to break each interval into several sub-intervals, and to define the interpolating coefficients piecewise on these sub-intervals while guaranteeing that they remain continuous, and that the expression of $K_A$ in Eq. (\ref{KA}) has separable variables.
\end{ex}
\section{Randomized arcade processes}\label{RAPsection}
We extend the construction of arcade processes to allow interpolation between the components of a random vector $X$ instead of interpolation between zeros. Two sets $\{f_0,\ldots,f_n\}$ and $\{g_0,\ldots,g_n\}$ of interpolating coefficients, see Definition \ref{interpolating}, are needed to ensure the matching of the target random variables.
We recall that the $\R^{n+1}$-valued random vector $X=(X_0,\ldots, X_n)$ is independent of the stochastic driver $(D_t)$, while the random variables $X_0,\ldots, X_n$ may be mutually dependent. 
\begin{defn} \label{RAP}
An $X$-randomised arcade process ($X$-RAP) on the partition $\{T_0,T_n\}_*$, denoted $(I_t^{(n)})_{t \in [T_0,T_n]}$, is a stochastic process of the form 
\begin{equation}
    I_t^{(n)}:= D_{t}-\sum_{i=0}^n \left(f_i(t)D_{T_i}-g_i(t)X_{i}\right),
\end{equation}
where $f_0,\ldots, f_n$ and $g_0,\ldots, g_n$ are interpolating coefficients on $\{T_0,T_n\}_*$. Writing $I_t^{(n)}=S_t^{(n)}+  A_t^{(n)}$, we refer to 
\begin{equation}
    S_t^{(n)}=\sum\limits_{i=0}^n g_i(t)X_{i}
\end{equation}
as the signal function of $I_t^{(n)}$ and to 
\begin{equation}
    A_t^{(n)}=D_{t}-\sum\limits_{i=0}^n f_i(t)D_{T_i}
\end{equation}
as the noise process of $I_t^{(n)}$. We denote by $(\F^I_t)_{t \in [T_0,T_n]}$ the filtration generated by $(I_t^{(n)})$.
\end{defn}
We emphasise that $I_{T_0}^{(n)}= X_{0}, \ldots, I_{T_n}^{(n)}= X_{n}$, so $(I_t^{(n)})$ is a strong stochastic interpolator between the random variables $X_{0},\ldots, X_{n}$. The signal function $(S_t^{(n)})$ is independent of the noise process $(A_t^{(n)})$ since $X$ is assumed to be independent of $(D_t)_{t\in [T_0,T_n]}$. 
\begin{rem}
A related but distinct class of processes, introduced in \cite{Menguturk} and called random $n$-bridges, shares some characteristics with RAPs. These processes, defined weakly, match given probability measures instead of random variables. They are constructed in a similar fashion as stochastic bridges, that is, by conditioning a stochastic process to match given distributions at given times. In special cases, the law of a RAP satisfies the conditions for the RAP to be a random $n$-bridge. For instance, the RAP obtained by randomising the stitched Brownian bridge, using the same interpolating coefficients for the signal function as the ones used in the noise process, has a law that satisfies the conditions for the RAP to be a random $n$-bridge. However, any other RAP driven by Brownian motion is not a random $n$-bridge. Conversely, certain random $n$-bridges cannot have the same law as a RAP. In this work, we chose not to consider the case in which the driver underlying a RAP jumps. But if we did, just for the sake of comparison, a random $n$-bridge built using a gamma process would not match the law of any RAP driven by a gamma process, since APs are sums, not products. The development of alternative arcade processes driven by (pure) jump processes will be pursued in future research. It requires a different construction that does not rely on the covariance function, similar to the change needed to go from Gaussian bridges to Lévy bridges \cite{Hoyle} or Markov bridges \cite{Macrina2} .
\end{rem}
The paths of an $X$-RAP depend on the joint distribution of the random vector $X$, not only on its marginal distributions. This fact is illustrated in the following example.
\begin{ex} \label{ex33}
Let $X=(X_0,\ldots,X_5)$ be a vector of independent and uniform $\mathcal U\,(\{-1,1\})$ random variables and $Y=(Y_0,\ldots,Y_5)$ be another vector of random variables such that $Y_0 \sim \mathcal U\,(\{-1,1\})$, $Y_i=-Y_{i-1}$ for $i=1,\ldots, 5$. Let $(A_t^{(5)})$ be an AP with elliptic interpolation coefficients driven by Brownian motion, $g_i=f_i$ for $i=0,\ldots, 5$, and $(I_t^{(5)})$, $(\tilde I_t^{(5)})$ its associated $X$-RAP and $Y$-RAP, respectively. Although the same driver and interpolating coefficients for both RAPs are the same, and the vectors $X$ and $Y$ have the same marginal distributions, the paths of $(I_t^{(5)})$ and $(\tilde I_t^{(5)})$ are different.
\begin{figure}[H]
  \centering
 \begin{minipage}[t]{0.49\textwidth}
    \includegraphics[width=\textwidth]{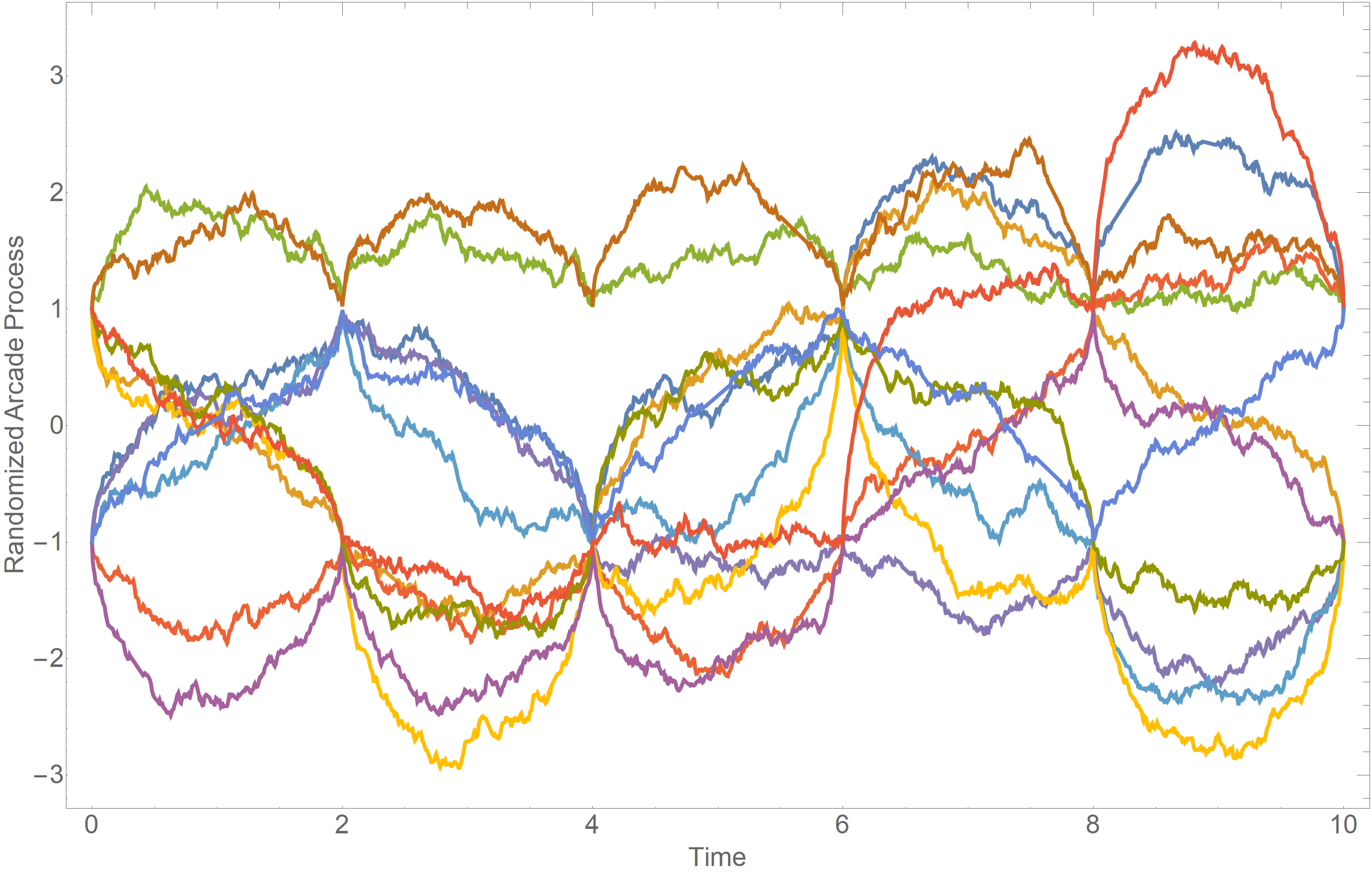}
    \caption{Paths simulation of $(I_t^{(5)})$ on $[0,10]$ using the equidistant partition $\{T_i = 2i \mid i=0,1,\ldots, 5\}$, cf., Example \ref{ex33}.}
  \end{minipage}
  \hfill
  \begin{minipage}[t]{0.49\textwidth}
    \includegraphics[width=\textwidth]{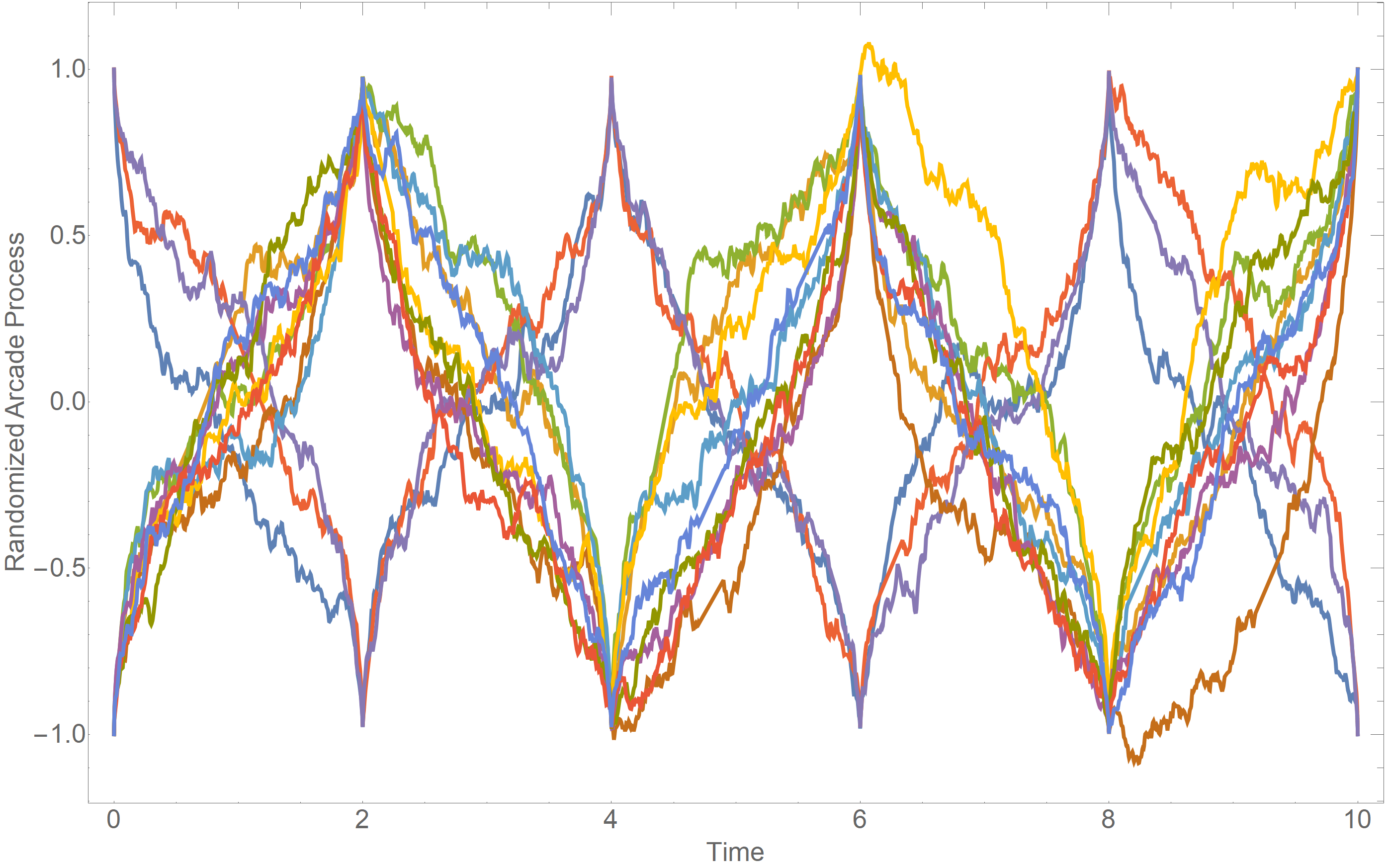}
    \caption{Paths simulation of $(\tilde I_t^{(5)})$ on $[0,10]$ using the equidistant partition $\{T_i = 2i \mid i=0,1,\ldots, 5\}$, cf., Example \ref{ex33}.}
  \end{minipage}
\end{figure}
\end{ex}
As a side remark, we note here that besides its main purpose of interpolating in the strong sense, a RAP can also be used to mimic a stochastic process. Let $(Y_t)_{t \in [T_0,T_n]}$ be a sample-continuous stochastic process. For instance, if $\{T_0,T_n\}_*$ is the equidistant partition of an interval $[a,b]$, $X=(Y_{T_0}, \ldots, Y_{T_n})$, and $\{f_0,\ldots,f_n\}=\{g_0,\ldots,g_n\}$ are the piecewise linear interpolating coefficients, then for almost all $\omega \in \Omega$, $\sup\limits_{t \in [a,b]} A_t^{(n)}(\omega) \rightarrow 0$,  and $\sup\limits_{t \in [a,b]} \abs{S_t^{(n)}(\omega) - Y_t(\omega)} \rightarrow 0$ as $n \rightarrow \infty$. Hence, in this case,
\begin{equation}
    \sup\limits_{t \in [a,b]} \abs{I_t^{(n)}-Y_t}=\sup\limits_{t \in [a,b]} \abs{S_t^{(n)}+  A_t^{(n)}-Y_t}  \rightarrow 0,
\end{equation}
with probability one.

\begin{prop}
Let $(S_t^{(n)})$ and $(A_t^{(n)})$ have mean functions $\mu_S$ and $\mu_A$, variance functions $\sigma_S^2$ and $\sigma_A^2$, and covariance functions $K_S$ and $K_A$, respectively. Then
\begin{align}
    &\mu_I(t):=\E [I_{t}^{(n)}]=  \mu_S(t) + \mu_A(t), \\
    &\sigma_I^2(t):=\var{[I_{t}^{(n)}]} = \sigma^2_S(t) +  \sigma^2_A(t), \\
    &K_I(s,t):=\cov[{I_{s}^{(n)}, I_{t}^{(n)}}] \nn \\
    &\hspace{1.5cm} = K_S(s,t) + K_A(s,t) =\sum\limits_{i=0}^n \sum\limits_{j=0}^n g_i(t)g_j(s)\cov[X_{i},X_j]  + K_A(s,t).
\end{align}
\end{prop}
The proof follows from the proof of Prop. \ref{covform}. We introduce terminologies similar to those in the previous section.
\begin{defn} \label{standardRAP}
Let $(I_t^{(n)})$ be a RAP. 
\begin{enumerate}
    \item $(I_t^{(n)})$ is said to be a Gaussian RAP if its stochastic driver $(D_t)$ is a Gaussian process.
    \item $(I_t^{(n)})$ is said to be a standard RAP if its noise process $(A_t^{(n)})$ is a standard AP, and if for all $x\in[T_0,T_n]$ and $j=1,\ldots,n$ it holds that $g_j(x) \1_{[T_0,T_{j-1}]}(x)=0$, and $g_j(x) \1_{[T_{j-1},T_{j}]}(x)= f_j(x)\1_{[T_{j-1},T_{j}]}(x)$.
\end{enumerate}
\end{defn}
\section{Conditionally Markov RAPs} \label{conditonallymarkov}
Aside of simple cases, an $X$-RAP is not Markovian since no constraints are applied to the joint distribution of $X$. Hence, we introduce a relaxed notion of the Markov property that is better suited to RAPs.
\begin{defn} \label{nearlymarkov}
Let $\mathcal I  \subseteq \R^+$ be a real interval and $\tau_0<\tau_1 < \ldots  < \infty$ such that $\tau=\{\tau_0, \tau_1, \ldots\} \subset \mathcal I$. The set $\tau$ may be finite or contain countably many elements. A stochastic process $(Y_t)_{t \in \mathcal I}$ is called $\tau$-conditionally Markov if 
\begin{equation}
    \P\left[Y_t \in \cdot \mid  \F^Y_s \right] = \P\left[Y_{t} \in \cdot \mid Y_{\tau_0},\ldots,Y_{\tau_{m(s)}}  ,  Y_s \right]
\end{equation}
for any $(s,t) \in \mathcal I^2$ such that $s \leqslant t$, and $\tau_{m(s)} = \max\limits_{i\in \N} \{\tau_i \mid \tau_i \leqslant s\}$.
\end{defn}
Now, we give a result on the conditionally-Markov property of Gaussian RAPs similar to the one on the Markov property of Gaussian APs.
\begin{thm} \label{thmsemimarkov}
Let $(I_t^{(n)})_{t \in [T_0,T_n]}=(S_t^{(n)} + A_t^{(n)})_{t \in [T_0,T_n]}$ be a Gaussian $X$-RAP on $\{T_0,T_n\}_*$. Then $(I_t^{(n)})$ is $\{T_0,T_n\}_*$-conditionally Markov if the following conditions are all satisfied:
\begin{enumerate}
    \item The AP $(A_t^{(n)})$ is Markov, i.e., $K_A(s,t)\\ =\sum\limits_{i=0}^{n-1} A_1(\min(s,t))A_2(\max(s,t)) \1_{(T_i,T_{i+1})}(s,t)$.
    \item For all $ j=1,\ldots, n$, and for all $ x \in [T_0,T_n]$, 
    \begin{equation} \label{subcond1}
        g_j(x) \1_{[T_0,T_{j-1}]}(x)=0,
    \end{equation}
    \begin{equation} \label{subcond2}
        g_j(x)A_1(T_j) \1_{[T_{j-1},T_{j}]}(x)= A_1(x)\1_{[T_{j-1},T_{j}]}(x).
    \end{equation}
\end{enumerate}
\end{thm}
\begin{proof}
Consider Theorem 1.12 in \cite{Lipster}, and let $k>1$ and $(s_1 , s_2 , \ldots , s_k , t ) \in [T_0,T_n] ^{k+1}$ such that $s_1 < s_2 < \ldots < s_k < t$. Then, $(I_t^{(n)})$ is $\{T_0,T_n\}_*$-conditionally Markov if and only if
\begin{equation}
\mathbb{P}\left[I_{t}^{(n)} \in \cdot \mid X_0, \ldots, X_{m(s_k)},  I_{s_1}^{(n)}, \ldots, I_{s_{k}}^{(n)}\right] = \mathbb{P}\left[I_{t}^{(n)} \in \cdot \mid X_0, \ldots, X_{m(s_k)}, I_{s_k}^{(n)} \right].
\end{equation}
where $m(s_k):= \max\limits\{i \in \N \mid T_i \leqslant s_k\}$. In the following, we will refer to $m(s_k)$ by $m$ since $s_k$ is fixed. 

We first show that $s_1,\ldots, s_k$ can be selected to all be in the sub-interval $(T_m,T_{m+1})$. To see this, assume there is an integer $j \in \{1,\ldots, k\}$ such that $s_j <T_m$ and $T_m<s_{j+1}$. Then 
\begin{equation}
\sigma( X_0, \ldots, X_{m},  I_{s_1}^{(n)}, \ldots, I_{s_{k}}^{(n)}) = \sigma( X_0, \ldots, X_{m},  A_{s_1}^{(n)}, \ldots, A_{s_j}^{(n)}, I_{s_{j+1}}^{(n)}, \ldots, I_{s_{k}}^{(n)})
\end{equation}
by Eq. (\ref{subcond1}). We also know that $(A_{s_1}^{(n)}, \ldots, A_{s_j}^{(n)}) \ind (X_0, \ldots, X_{m})$ by the definition of the $X$-RAP, and therefore, in particular, $(A_{s_1}^{(n)}, \ldots, A_{s_j}^{(n)}) \ind (S_{s_{j+1}}^{(n)} , \ldots, S_{s_{k}}^{(n)}, S_{t}^{(n)})$. Furthermore, since $(A_t^{(n)})$ is Markov and $A_{T_m}^{(n)}=0$, we have $(A_{s_1}^{(n)}, \ldots, A_{s_j}^{(n)}) \ind (A_{s_{j+1}}^{(n)} , \ldots, A_{s_{k}}^{(n)}, A_{t}^{(n)})$, and hence $(A_{s_1}^{(n)}, \ldots, A_{s_j}^{(n)}) \ind (I_{s_{j+1}}^{(n)}, \ldots, I_{s_{k}}^{(n)}, I_{t}^{(n)} ) =(S_{s_{j+1}}^{(n)} + A_{s_{j+1}}^{(n)}, \ldots, S_{s_{k}}^{(n)}+ A_{s_{k}}^{(n)}, S_{t}^{(n)}+A_{t}^{(n)} )$. We conclude that 
\begin{equation}
\mathbb{P}\left[I_{t}^{(n)} \in \cdot \mid X_0, \ldots, X_{m},  I_{s_1}^{(n)}, \ldots, I_{s_{k}}^{(n)}\right] = \mathbb{P}\left[I_{t}^{(n)} \in \cdot \mid X_0, \ldots, X_{m},  I_{s_{j+1}}^{(n)}, \ldots, I_{s_{k}}^{(n)}\right],
\end{equation}
which means, we can assume that $s_1,\ldots, s_k$ are all in the same sub-interval $(T_m,T_{m+1})$. 

Let us define $a_m(\cdot):= \sum\limits_{i=0}^m g_i(\cdot) X_i$, and
\begin{equation}
    \Delta_q:=\sum_{i=1}^{k}  c_{i,q}   I_{s_i}^{(n)} = \sum_{i=1}^{k}  c_{i,q}   a_m(s_i) + \sum_{i=1}^{k}  c_{i,q}  g_{m+1}(s_i) X_{m+1} + \sum_{i=1}^{k}  c_{i,q}  A_{s_i}^{(n)},
\end{equation}
for $q=1,\ldots,k-1$, where the coefficients $(c_{i,q})$ are chosen such that
\begin{equation} \label{choice}
\sum\limits_{i=1}^{k}  c_{i,q}  K_A(s_i,t)=0 \quad \mbox{ and } \quad \mathbb \det
\begin{pmatrix}
c_{1,1}  & \ldots & c_{1,k-1}  \\
\vdots &  & \vdots  \\
c_{k-1,1}  & \ldots & c_{k-1,k-1}   \\
\end{pmatrix} \neq 0.
\end{equation}
This guarantees the following, where the notation "$\mid (X_0, \ldots, X_{m})$" means conditional on $(X_0, \ldots, X_{m})$.
\begin{enumerate}
    \item $\mathbb{P}\left[I_{t}^{(n)} \in \cdot \mid  X_0, \ldots, X_{m}, I_{s_1}^{(n)}, \ldots, I_{s_{k}}^{(n)}\right] \\ = \mathbb{P}\left[I_{t}^{(n)} \in \cdot \mid X_0, \ldots, X_{m}, \Delta_1, \ldots, \Delta_{k-1},I_{s_k}^{(n)} \right]$.
    \item $(\Delta_1, \ldots, \Delta_{k-1}) \mid (X_0, \ldots, X_{m})$ is a Gaussian vector. To see this, we observe that for all $q =1,\ldots k-1$,
    \begin{equation}
        \sum\limits_{i=1}^{k}  c_{i,q}  K_A(s_i,t)=0  \implies \sum\limits_{i=1}^{k}  c_{i,q}  g_{m+1}(s_i)=0,
    \end{equation}
    where we used Eq. (\ref{subcond2}). Hence, $\Delta_q=\sum\limits_{i=1}^{k}  c_{i,q}   I_{s_i}^{(n)}=\sum\limits_{i=1}^{k}  c_{i,q}   a_m(s_i)+ \sum\limits_{i=1}^{k}  c_{i,q}   A_{s_i}^{(n)}$ for all $q=1,\ldots,k-1$, which implies that $(\Delta_1, \ldots, \Delta_{k-1}) \mid (X_0, \ldots, X_{m})$ is a Gaussian vector.
    \item $A_{t}^{(n)} \ind (\Delta_1, \ldots, \Delta_{k-1}) \mid (X_0, \ldots, X_{m})$, since $\sum\limits_{i=1}^{k}  c_{i,q}  K_A(s_i,t)=0$.
    \item $A_{s_k}^{(n)} \ind (\Delta_1, \ldots, \Delta_{k-1}) \mid (X_0, \ldots, X_{m})$, since $(A_t^{(n)})$ is Markov.
\end{enumerate}
To conclude, we need to show
\begin{equation} \label{ind2}
    I_{t}^{(n)} \ind (\Delta_1, \ldots, \Delta_{k-1}) \mid (X_0, \ldots, X_{m}) \text{ and }  I_{s_k}^{(n)} \ind (\Delta_1, \ldots, \Delta_{k-1}) \mid (X_0, \ldots, X_{m}).
\end{equation}
Since $(I_t^{(n)})=(S_t^{(n)}+A_t^{(n)})$, and $(S_t^{(n)}) \ind (A_t^{(n)})$, we have (\ref{ind2}) if
\begin{equation}
   A_{t}^{(n)} \ind (\Delta_1, \ldots, \Delta_{k-1})\mid (X_0, \ldots, X_{m}) \text{ and }  A_{s_k}^{(n)} \ind (\Delta_1, \ldots, \Delta_{k-1})\mid (X_0, \ldots, X_{m}),
\end{equation}
which is guaranteed by conditions (\ref{choice}).
\end{proof}
\begin{rem}
If $(A_t^{(n)})$ is standard, see Definition \ref{standard}, then Eq. (\ref{subcond2}) is equivalent to 
\begin{equation}
    g_j(x) \1_{[T_{j-1},T_{j}]}(x)= f_j(x)\1_{[T_{j-1},T_{j}]}(x).
\end{equation}
This makes standard RAPs automatically conditionally Markov.
\end{rem}
\begin{rem}
Depending on the coupling of $X$, the probability $\mathbb{P}[I_{t}^{(n)} \in \cdot \mid X_0, \ldots, X_{m(s)}, I_{s}^{(n)}]$ might simplify further. For instance, if $X$ has continuous marginals and is distributed according to Kantorovich's coupling, then 
\begin{equation}
    \mathbb{P}\left[I_{t}^{(n)} \in \cdot \mid X_0, \ldots, X_{m(s)}, I_{s}^{(n)} \right]= \mathbb{P}\left[I_{t}^{(n)} \in \cdot \mid X_0, I_{s}^{(n)} \right]
\end{equation}
because $X_1,\ldots,  X_{m(s)}$ are all deterministic functions of $X_0$.
\end{rem}
\begin{rem} \label{remreverse}
It is important to note that the conditionally-Markov property is not symmetric in time. Define $\mathcal G_t^I=\sigma ( I_u^{(n)} \mid t \leqslant u \leqslant T_n) $, and let $s < t$ in $[T_0,T_n]$. Then, to obtain
\begin{equation}
    \P\left[I_s^{(n)} \in \cdot \mid  \mathcal G^I_t \right] = \P\left[I_s^{(n)}  \in \cdot \mid I_t^{(n)}, X_{k(t)}, X_{k(t)+1}, \ldots, X_n  \right]
\end{equation}
where $k(t)=\min\limits\{i \in \N \mid T_i \geqslant t\}$, Condition (\ref{subcond1}) needs to be replaced by 
\begin{equation}
   g_j(x) \1_{[T_{j+1},T_{n}]}(x)=0,
\end{equation}
and Condition (\ref{subcond2}) by 
\begin{equation}
g_j(x)A_1(T_{j}) \1_{[T_{j},T_{j+1}]}(x)= A_1(x)\1_{[T_{j},T_{j+1}]}(x)
\end{equation}
for all $ j=0,\ldots, n-1$, and for all $ x \in [T_0,T_n]$.
\end{rem}
\begin{ex} \label{exf}
We give an example of a non-standard $X$-RAP on $[T_0,T_2]$ that is $\{T_0,T_2\}_*$-conditionally Markov, where $X_0\eql X_1  \eql X_2 \eql \textrm{Uni}\,(\{-1,1\})$ are mutually independent. Consider the interpolating coefficients
\begin{align}
f_0(t)&= \frac{T_1-t}{T_1-T_0} \1_{\left[T_0,\frac{T_1+T_2}{2}\right]}(t) - \frac{T_2-t}{T_1-T_0} \1_{\left(\frac{T_1+T_2}{2},T_2\right]}(t), \\
f_1(t)&= \frac{t-T_0}{T_1-T_0} \1_{[T_0,T_1]}(t) +  \frac{T_2-t}{T_2-T_1} \1_{(T_1,T_2]}(t), \\
f_2(t)&= \frac{t-T_1}{T_2-T_1} \1_{\left[T_1,T_2\right]}(t).
\end{align}
Let $(A_t^{(2)})_{t \in [T_0,T_2]}$ be the arcade process built with these interpolating coefficients and driven by a standard Brownian motion. As shown in Example \ref{nonstarcade}, $(A_t^{(2)})$ is Markov. For Eq. (\ref{subcond1}) to be satisfied, one only needs to impose $g_2(t)\1_{[T_0,T_1]}(t)=0$. For Eq. (\ref{subcond2}) to be satisfied, one requires that 
\begin{align}
&g_1(t)\1_{[T_0,T_1]}(t)= \frac{t-T_0}{T_1-T_0}\1_{[T_0,T_1]}(t),\\ &g_2(t)\1_{\left[T_1,\frac{T_1+T_2}{2}\right]}(t)=\left( \frac{t-T_1}{T_2-T_1} + \frac{(t-T_1)T_0}{(T_1-T_0)^2}\right) \1_{\left[T_1,\frac{T_1+T_2}{2}\right]}(t),\\
&g_2(t)\1_{\left(\frac{T_1+T_2}{2}, T_2 \right]}(t) = \left(\frac{t-T_1}{T_2-T_1} + \frac{(T_2  -t )T_0}{(T_1-T_0)^2}\right) \1_{\left(\frac{T_1+T_2}{2}, T_2 \right]}(t).
\end{align}
Outside of the considered intervals, the functions $g_i$ may take any values as long as they remain interpolating coefficients. Theorem \ref{thmsemimarkov} does not impose a condition on $g_0$. For example, we could choose $g_i=f_i$ outside the above intervals. Hence, all three conditions are fulfilled and this $X$-RAP is $\{T_0,T_2\}_*$-conditionally Markov. As we can see from the paths-simulation below, this process is different from a randomised stitched Brownian arcade on the second arc. The noise has been diminished to make the paths more informative. Simulating the signal function highlights the following: $X_0$ will determine the fate of the signal function on $[T_1,T_2]$, since this RAP remembers the previously matched random variables when changing arcs. On the first arc, where the process is simply a randomised Brownian bridge, to go from $X_0=-1$ to $X_1=-1$, for instance, there is only one way---a straight line. On the second arc, to go from $X_1=-1$ to $X_2=-1$, there are two ways. The evolution of the signal function is determined by the value that $X_0$ takes. This is illustrated by the paths of the signal function below: the blue path and the green path both take value $-1$ at $T_1$ and value $1$ at $T_2$, but have different values in $T_0$. Hence, they differ on $[T_1,T_2]$, as observed.
\begin{figure}[H]
   \centering
   \begin{minipage}[t]{0.49\textwidth}
     \vtop{\null\hbox{\includegraphics[width=\textwidth]{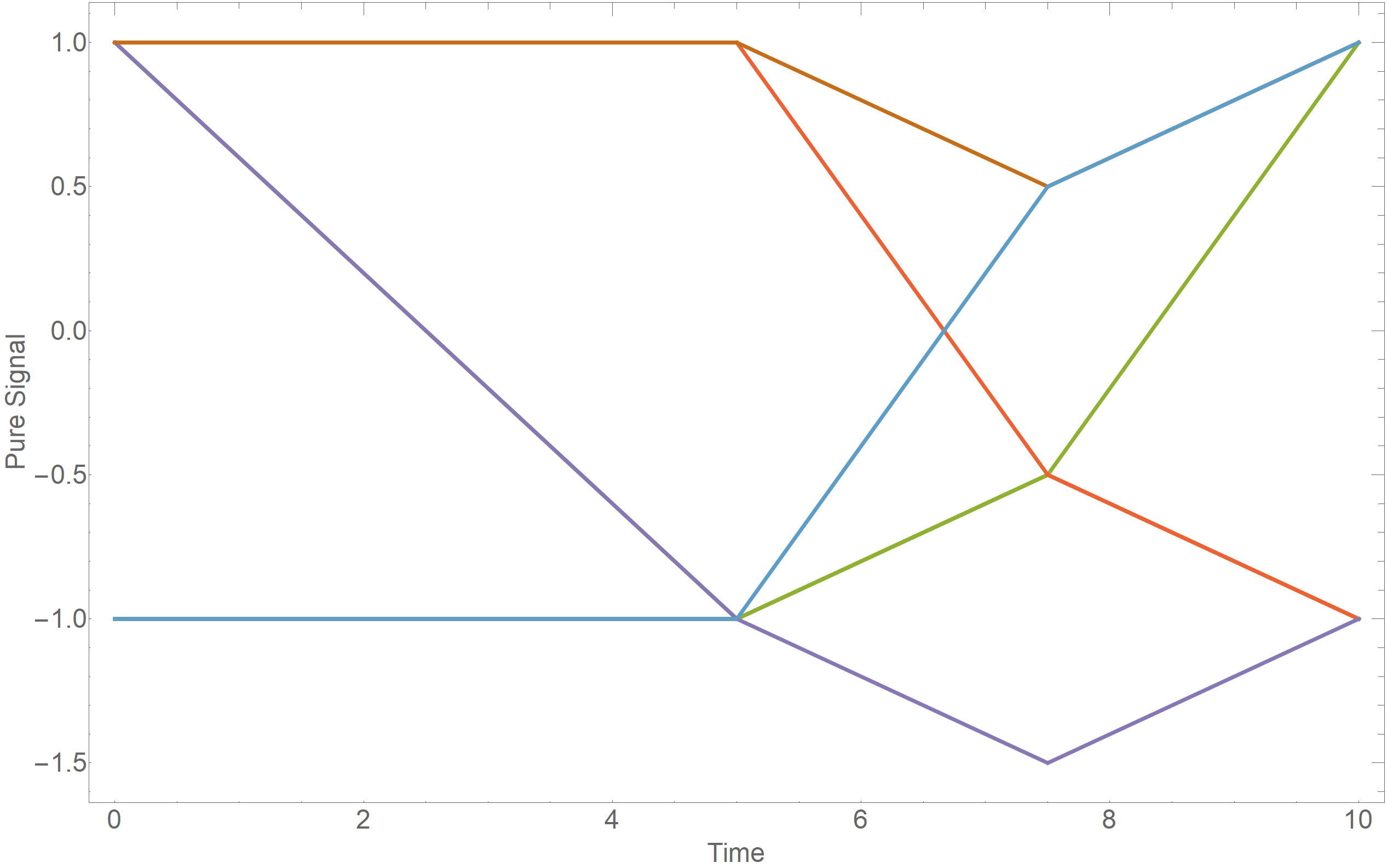}}}
     \caption{Paths simulation of the signal function of a non-standard $X$-RAP, where $\{T_0,T_2\}_*=\{0,5,10\}$, cf., Example \ref{exf}.}
   \end{minipage}
   \hfill
   \begin{minipage}[t]{0.49\textwidth}
     \vtop{\null\hbox{\includegraphics[width=\textwidth]{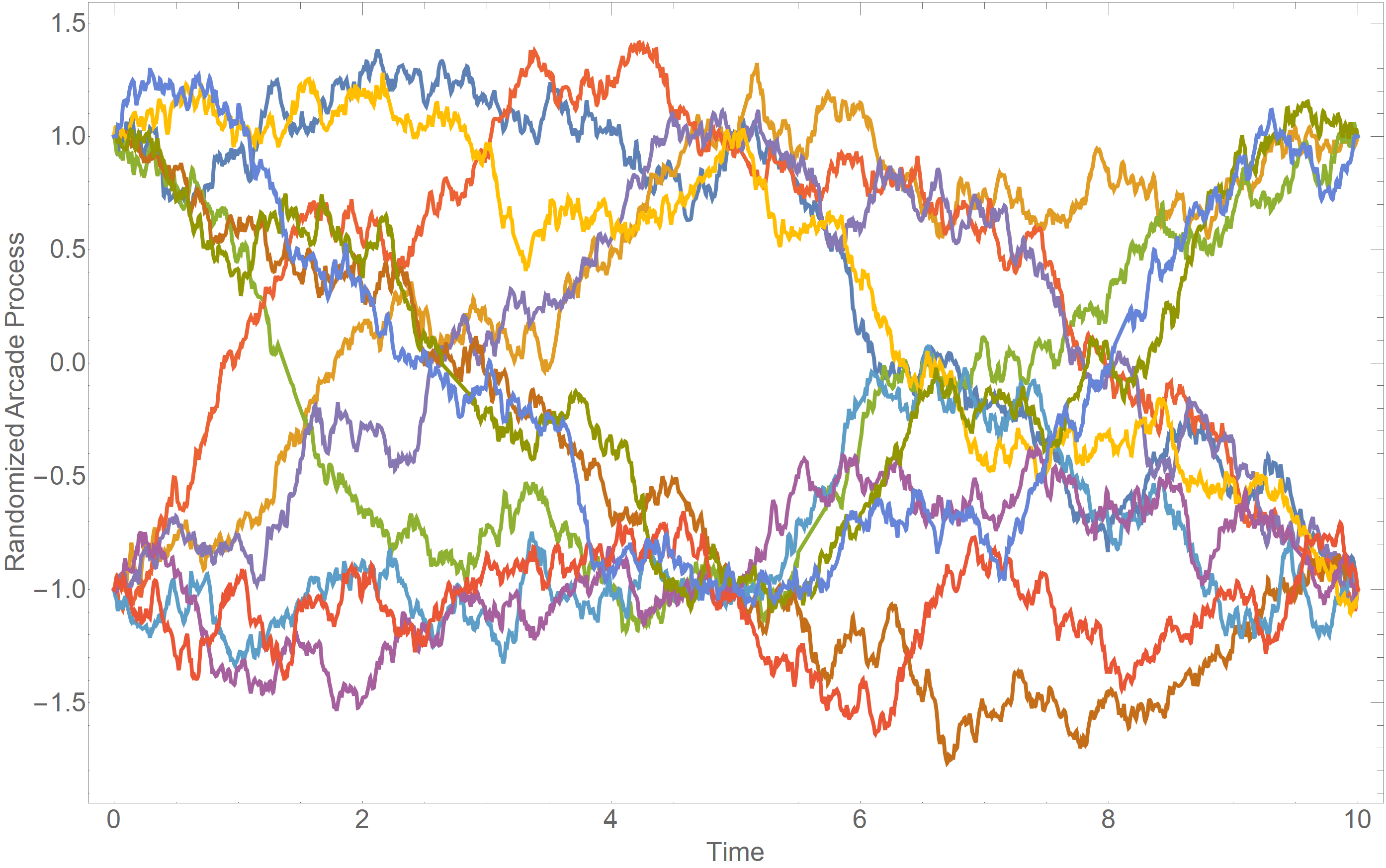}}}
     \caption{Paths simulation of a non-standard $X$-RAP, where the noise is rescaled by 0.3 and $\{T_0,T_2\}_*=\{0,5,10\}$, cf., Example \ref{exf}.}
   \end{minipage}
\end{figure}
\end{ex}
\section{Filtered arcade martingales} \label{FAMchapter}
In this section, given a martingale vector $X=(X_0,\ldots,X_n)$, i.e., a finite discrete-time martingale, we construct continuous-time martingales with respect to the filtration generated by an $X$-RAP, which interpolate between the convexly ordered components of the random vector $X$. These martingales solve an underlying stochastic filtering problem, and extend the martingale class of the information-based theory developed in \cite{BHM1}, see also \cite{BHM2}. We call such martingales \emph{filtered arcade martingales} (FAMs). We introduce the following notation:
\begin{enumerate}
    \item $\mathcal P(\R^n)$ is the set of Borel probability measures on $\R^n$, for $n\in \N_0$.
    \item $\mathcal P_1(\R^n)$ is the set of Borel probability measures on $\R^n$, for $n\in \N_0$, with finite first moment.
    \item $\Pi(\mu_0,\mu_1,\ldots, \mu_n):=  \{ \pi \in \mathcal P(\R^{n+1}) \| \mu_{i-1} \text{ is the ith marginal measure of } \pi, \text{ for all }  \\ i=1,\ldots, n+1 \}$ is the set of couplings of $\{\mu_0,\mu_1,\ldots, \mu_n\} \subseteq \mathcal P (\R)$.
    \item $\mathcal M (\mu_0,\mu_1,\ldots, \mu_n):=\{\pi \in \Pi(\mu_0,\mu_1,\ldots, \mu_n) \cap \mathcal P_1(\R^{n+1}) \| (X_0,X_1,\ldots, X_n) \sim \pi \implies \\ \E[X_n \| X_0,\ldots, X_m] = X_m \text{ for all } m=0,\ldots, n\}$ is the set of martingale couplings of $\{\mu_0,\mu_1,\ldots, \mu_n\} \subseteq \mathcal P (\R)$.
\end{enumerate}
We recall that $\mathcal M (\mu_0,\mu_1,\ldots, \mu_n) \neq \emptyset$ if and only if $\int_\R f(x) \d \mu_0(x) \leqslant \int_\R f(x) \d \mu_1(x)\leqslant \ldots \leqslant \int_\R f(x) \d \mu_n(x)$ for any convex function $f$ on $\R$, i.e, the measures $\mu_0,\mu_1,\ldots, \mu_n$ are convexly ordered. In that case, we write $\mu_0\leqslant_{\mathrm{cx}}\mu_1\leqslant_{\mathrm{cx}}\ldots\leqslant_{\mathrm{cx}} \mu_n$.
\subsection{The one-arc FAM} \label{1arccase}
Let $X=(X_0,X_1)$ be a random vector distributed according to a martingale coupling $\pi^X \in \mathcal M (\mu_0,\mu_1)$, where $(\mu_0,\mu_1) \in \mathcal P_1 (\R) \times \mathcal P_1 (\R)$ and $\mu_0 \leqslant_{\mathrm{cx}}\mu_1$. It follows that $X_0 \sim \mu_0$, $X_1 \sim \mu_1$, and $\E[X_1 \mid X_0 ] \eqas X_0$. Given an $X$-RAP $(I_t^{(1)})_{t \in [T_0,T_1]}$ on the partition $\{T_0,T_1\}_*$, we aim at constructing an $(\F^I_t)$-martingale $(M_t)_{t\in[T_0, T_1]}$ such that $M_{T_0}\eqas X_0$ and $M_{T_1}\eqas X_1$. The BHM framework developed in \cite{BHM1} is recovered when $X_0 =\E[X_1]=0$ and $T_0=0$.
\begin{defn} \label{1FAM}
Given an $X$-RAP $(I_t^{(1)})$, a one-arc FAM for $X\sim \pi^X$ on $[T_0,T_1]$ is a stochastic process defined by $M_t:=\E[X_1 \mid \F_t^I ]$.
\end{defn}
The choice of the coupling $\pi^X$, alongside the interpolating coefficients and the driver of $(I_t^{(1)})$, dictate the behaviour of the paths of a FAM, since they tell us how $X_0$, $X_1$, and the information process $(I_t^{(1)})$ are related to each other.
\begin{prop}
The FAM $(M_t)_{t\in[T_0, T_1]}$ is an $(\F_t^I)$-martingale that interpolates between $X_0$ and $X_1$, almost surely.
\end{prop}
\begin{proof}
The process $(M_t)$ is adapted to $(\mathcal F_t^I)$ by construction. We show that it is a martingale that interpolates between $X_0$ and $X_1$ as follows.
\begin{enumerate}
    \item $\E[|M_t|] < \infty$ for all $t\in [T_0,T_1]$ by Jensen inequality, since $\E[|X_1|] < \infty$.
    \item For $s<t$, $\E[M_t \mid \F_s^I] = M_s$ by the tower property of the conditional expectation.
    \item $M_{T_0} =\E[X_1 \mid X_0] \eqas X_0$, since $\pi^X \in \mathcal M (\mu_0,\mu_1)$ and $I_{T_0}^{(1)}=X_0$.
    \item $M_{T_1}= \E[X_1 \mid \F_{T_1}^I] \eqas X_1$, since $\sigma(X_1) \subset \F_{T_1}^I$ by construction of $(I_t^{(1)})$.
\end{enumerate}
Hence, $(M_t)$ is an $(\F_t^I)$-martingale that interpolates almost surely between $X_0$ and $X_1$ over the interval $[T_0,T_1]$.
\end{proof}
\begin{rem}
The process $(M_t)_{t\in [T_0,T_1]}$ is also a martingale with respect to its natural filtration, denoted $(\F_t^M)_{t\in[T_0, T_1]}$. For $s<t$, we have
\begin{align}
    \E [M_t \mid \F^M_s]\eqas   \E[\E [M_t \mid \F^M_s] \mid \F^I_s] \eqas \E[\E [M_t \mid \F^I_s] \mid \F^M_s] \eqas\E[ M_s\mid  \F^M_s] \eqas M_s.
\end{align}
\end{rem}
To assist with the derivation of the outcomes in the remainder of this section, we give a brief overview of the goals and forthcoming results as a guide. First, we recall the martingale representation theorem (MRT), as presented in \cite{Revuz}, Theorem 4.2.
\\
\\
{\bf Martingale representation theorem}, \cite{Revuz}.
{\it
Let $(W_t)_{t \geqslant 0}$ be a standard Brownian motion on a filtered probability space $(\Omega, \mathcal{F},\left(\mathcal{F}_t\right)_{t\ge 0}, \mathbb{P})$ satisfying the usual conditions.\footnote{The usual conditions are:
\begin{enumerate}
    \item Right-continuity of the filtration:
    \begin{equation*}
    \mathcal{F}_t=\bigcap_{s>t} \mathcal{F}_s \quad \text { for all } t \geqslant 0 .
    \end{equation*}
    \item Completeness of the probability space: The $\sigma$-algebra $\mathcal{F}_0$ contains all $\mathbb{P}$-null sets of $\F$. That is, if $A \in \F$ and $\mathbb{P}(A)=0$, then $A \in \F_0$.
    \item Adaptedness of the Brownian motion: The Brownian motion $(W_t)$ is adapted to the filtration $(\F_t)$.
\end{enumerate}
} Then, every square integrable $(\F_t$)-martingale $(M_t)_{t\ge 0}$ can be represented by
\begin{equation}
M_t = M_0+\int_0^t \phi_s \rd W_s
\end{equation}
almost surely, where $(\phi_t)$ is a predictable process such that $\int_0^T \phi_s^2 \rd s<\infty$ almost surely for all $T>0$.
}

The MRT assumes that the filtration $(\F_t)$ be generated by a Brownian motion (or at least rich enough to support one). If $(\F_t)$ is more general or includes information not generated by Brownian motion (e.g., jump processes, discrete-time processes, or other sources of randomness), then:
\begin{itemize}
    \item A Brownian motion may not exist with respect to the filtration $(\F_t)$.
    \item Even if a Brownian motion exists, it may not generate the entire filtration; in this case the MRT would not apply.
\end{itemize}
The main goals of this section are: 
\begin{itemize}
    \item Show that even though $(\F_t^I)$ is not generated by a Brownian motion, under appropriate conditions, one may write 
    \begin{equation}
        M_t \eqas X_0 + \int_0^t \sigma_s \rd W_s,
    \end{equation}
    where $(W_t)$ is an $(\F_t^I)$-adapted Brownian motion.
    \item Identify the processes $(\sigma_t)$ and $(W_t)$, i.e., give their explicit, almost sure expressions.
\end{itemize}
To achieve these goals, we proceed as follows.
\begin{enumerate}
    \item In Proposition \ref{relations}, we bring to the fore a relationship between $\E[I_t^{(1)} \mid \F^I_s]$, $M_s$, and $\E[ \E [A_t^{(1)} \mid \F^A_s] \mid \F^I_s]$ for any pair $(s,t) \in [T_0,T_1]^2$ such that $s\leqslant t$. This relationship will play a key role in finding the $(\F_t^I)$-adapted Brownian motion $(W_t)$. 
    \item In the case where $M_t=\E[X_1 \mid \F_t^I]=\E[X_1 \mid X_0, I_t^{(1)}]$ holds, i.e., in the case that $(I_t^{(1)})$ is a conditional-Markov process, we apply the Bayes rule to $(M_t)$ in Proposition \ref{Bayes} to obtain an integral representation that depends on the conditional density of $(I_t^{(1)})$ given $X_0$ and $X_1$, and the conditional distribution function of $X_1$ given $X_0$.
    \item We then apply Itô's formula to the Bayes integral representation in Proposition \ref{Ito}. This allows us to obtain the general SDE satisfied by a FAM that is driven by a conditional-Markov RAP, under the appropriate conditions implied by Itô's formula. It then turns out that the SDE is composed by a $\rd t$-term, a $\rd I_t^{(1)}$-term, and a $\rd [I^{(1)}]_t\,$-term.
    \item We specialise the SDE in two steps. First, we add the assumption that the driver of $(I_t^{(1)})$ is Gaussian in Corollary \ref{Itocoro1}. This allows us to calculate all the terms in the SDE explicitly. Then, we add the assumption that $(I_t^{(1)})$ is a standard RAP, see Corollary \ref{Itocoro2}, which allows us to calculate $[I^{(1)}]_t$ explicitly and further simplify the terms in the SDE. These now only have a $\rd t$-term and a $\rd I_t^{(1)}$-term, in which a common factor appears.
    \item Removing that common factor, Proposition \ref{propinnovations} shows that what is left is an $(\F^I_t)$-adapted Brownian motion, up to rescaling. Thus, the common factor yields the volatility process $(\sigma_t)$. Hence, one has achieved one's goal. This plan is rather standard in a non-linear stochastic filtering context, see, e.g., \cite{FKK}.
\end{enumerate}
We now begin with the first step of the described plan, above. Since equalities involving conditional expectations are understood in the almost sure sense, we will omit the a.s. ``symbol'' in "$\eqas$" from now on in equalities involving FAMs. 

A relationship that we shall use later between the RAP $(I_t^{(1)})$, its noise process $(A_t^{(1)})$, and its associated FAM $(M_t)$ is the following:
\begin{prop} \label{relations}
If $I_t^{(1)}= g_0(t)X_0 + g_1(t) X_1 + A_t^{(1)}$ and $M_t=\E[X_1 \mid \F_t^I ]$, then 
\begin{equation} \label{rel2}
    \E[I_t^{(1)} \mid \F^I_s]=g_0(t) X_0 + g_1(t) M_s + \E[ \E [A_t^{(1)} \mid \F^A_s] \mid \F^I_s]
\end{equation}
for any pair $(s,t) \in [T_0,T_1]^2$ such that $s\leqslant t$. Furthermore, if $(A_t^{(1)})$ is Gauss-Markov with $K_A(x,y)= A_1(\min(x,y))A_2(\max(x,y))$, then
\begin{align}
    \E[I_t^{(1)} \mid \F^I_s]&=  \left(g_0(t) - \frac{A_2(t)}{A_2(s)} g_0(s) \right)  X_0 + \left(g_1(t)- \frac{A_2(t)}{A_2(s)} g_1(s)\right) M_s\nn \\
    & \quad + \frac{A_2(t)}{A_2(s)} I_s  + \mu_A(t) +\frac{A_2(t)}{A_2(s)} \mu_A(s)
\end{align}
for any $(s,t) \in [T_0,T_1]^2$ such that $s\leqslant t$.
\end{prop}
\begin{proof}
Let $(s,t) \in [T_0,T_1]^2$ be such that $s\leqslant t$. We notice that 
\begin{align}
    \E[I_t^{(1)} \mid \F^I_s]&= g_0(t) X_0 + g_1(t) M_s + \E[A_t^{(1)} \mid \F^I_s] \nn \\
    &= g_0(t) X_0 + g_1(t) M_s + \E[ \E [A_t^{(1)} \mid \F^A_s] \mid \F^I_s],
\end{align}
since $\E[A_t^{(1)} \mid \F^I_s] = \E[ \E [A_t^{(1)} \mid X_0,X_1, \F^A_s] \mid \F^I_s] =\E[ \E [A_t^{(1)} \mid \F^A_s] \mid \F^I_s]$, where we used the fact that $(A_t^{(1)}) \ind (X_0,X_1)$ by Definition \ref{RAP}. If $(A_t^{(1)})$ is Gauss-Markov, then 
\begin{equation}
    \E [A_t^{(1)} \mid \F^A_s] = \E [A_t^{(1)} \mid A_s^{(1)}] = \mu_A(t) + \frac{K_A(s,t)}{\sigma_A^2(s)} \left( A_s^{(1)}  - \mu_A(s) \right).
\end{equation}
Hence, by the linearity property of the conditional expectation, we have
\begin{equation} \label{rel4}
    \E[ \E [A_t^{(1)} \mid \F^A_s] \mid \F^I_s] = \mu_A(t) + \frac{K_A(s,t)}{\sigma_A^2(s)} \left( \E[A_s^{(1)} \mid \F^I_s]  - \mu_A(s) \right).
\end{equation}
Moreover,
\begin{equation}
    I_t^{(1)} = \E[I_t^{(1)} \mid \F^I_t] = \E[A_t^{(1)} \mid \F^I_t] + g_0(t) X_0+g_1(t) M_t,
\end{equation}
which implies 
\begin{equation} \label{rel1}
    \E[A_t^{(1)} \mid \F^I_t]= I_t^{(1)} - g_0(t) X_0-g_1(t) M_t.
\end{equation}
Then, inserting Eq. (\ref{rel1}) into Eq. (\ref{rel4}), and recalling that $K_A(s,t)=A_1(s)A_2(t)$, gives
\begin{align}
    \E[I_t^{(1)} \mid \F^I_s]&=  \left(g_0(t) - \frac{A_2(t)}{A_2(s)} g_0(s) \right)  X_0 + \left(g_1(t)- \frac{A_2(t)}{A_2(s)} g_1(s)\right) M_s\nn \\
    & \quad + \frac{A_2(t)}{A_2(s)} I_s  + \mu_A(t) +\frac{A_2(t)}{A_2(s)} \mu_A(s). \label{rel3}
\end{align}
\end{proof}
If $(I_t^{(1)})$ is $\{T_0,T_1\}$-conditionally Markov, then $M_t = \E [ X_1 \mid X_0, I_t^{(1)} ] $. In this case, the dynamics of $(M_t)$ can be obtained using Bayes' rule and Itô's formula under mild assumptions. In the remainder of this section, we assume that the driver of $(I_t^{(1)})$ has a density function for $t \in [T_0,T_1]$. In what follows, when we write $M_t = \E [ X_1 \mid X_0, I_t^{(1)} ]$, we assume that $(I_t^{(1)})$ is $\{T_0,T_1\}$-conditionally Markov.
\begin{prop} \label{Bayes}
Let $M_t = \E [ X_1 \mid X_0, I_t^{(1)} ]$ be a one-arc FAM restricted to $t\in(T_0,T_1)$. Then
\begin{equation} \label{Bayeseq}
    M_t= \frac{\int_\R y f^{I_t^{(1)} \mid X_0,X_1=y} (I_t^{(1)})  \rd F^{X_1 \mid X_0} (y)}{\int_\R f^{I_t^{(1)} \mid X_0,X_1=y} (I_t^{(1)})  \rd F^{X_1 \mid X_0} (y)},
\end{equation}
where $f^{I_t^{(1)} \mid X_0,X_1}$ is the conditional density of $(I_t^{(1)})$ given $X_0$ and $X_1$, and $F^{X_1 \mid X_0}$ is the conditional distribution function of $X_1$ given $X_0$. In particular,
\begin{enumerate}
    \item If $(X_0,X_1)$ is a continuous random vector, then 
    \begin{equation}
    M_t = \frac{\int_\R y f^{I_t^{(1)} \mid X_0,X_1=y} (I_t^{(1)})   f^{X_1 \mid X_0}(y) \rd y}{\int_\R f^{I_t^{(1)} \mid X_0,X_1=y} (I_t^{(1)}) f^{X_1 \mid X_0}(y) \rd y}
\end{equation}    
\item If $(X_0,X_1)$ is a discrete random vector, then 
\begin{equation}
    M_t = \frac{\sum\limits_{y} y f^{I_t^{(1)} \mid X_0,X_1=y} (I_t^{(1)}) \P[X_1=y \mid X_0]  }{\sum\limits_{y}  f^{I_t^{(1)} \mid X_0,X_1=y} (I_t^{(1)}) \P[X_1=y \mid X_0]} .
\end{equation}
\end{enumerate}
\end{prop}
\begin{proof}
By the Bayes rule,
\begin{align}
    \P [ X_1 \leqslant y \mid X_0, z  \leqslant I_t^{(1)} \leqslant z + \epsilon ] &= \frac{\P [ z  \leqslant I_t^{(1)} \leqslant z + \epsilon \mid X_0,  X_1 \leqslant y ] \P [ X_1 \leqslant y \mid X_0]}{ \P [ z  \leqslant I_t^{(1)} \leqslant z + \epsilon \mid X_0 ]}\nn\\
    &= \frac{\P [ z  \leqslant I_t^{(1)} \leqslant z + \epsilon \mid X_0,  X_1 \leqslant y ] \P [ X_1 \leqslant y \mid X_0]}{ \int_\R \P [ z  \leqslant I_t^{(1)} \leqslant z + \epsilon \mid X_0 , X_1 =y] \rd F^{X_1\mid X_0}(y) }.
\end{align}
This means that, by taking the limit when $\epsilon \rightarrow 0$, 
\begin{equation}
    F^{X_1\mid X_0,I_t^{(1)} =z}(y) = \frac{ \frac{\rd}{\rd z}\P [I_t^{(1)} \leqslant z \mid X_0, X_1 \leqslant y]  F^{X_1\mid X_0}(y)}{ \int_\R f^{I_t^{(1)} \mid X_0,X_1 = y} (z)\rd F^{X_1\mid X_0}(y) },
\end{equation}
which implies 
\begin{equation}
    \rd F^{X_1\mid X_0,I_t^{(1)} =z}(y) = \frac{f^{I_t^{(1)} \mid X_0,X_1 =y} (z)\rd F^{X_1\mid X_0}(y) }{ \int_\R f^{I_t^{(1)} \mid X_0,X_1 = y} (z)\rd F^{X_1\mid X_0}(y) }.
\end{equation}
Inserting the expression for $\rd F^{X_1\mid X_0,I_t^{(1)} =z}(y)$ into $M_t$, we obtain 
\begin{equation}
    M_t= \int_{\R} y  \rd F^{X_1\mid X_0,I_t^{(1)}}(y) = \frac{\int_\R y f^{I_t^{(1)} \mid X_0,X_1=y} (I_t^{(1)})  \rd F^{X_1 \mid X_0} (y)}{\int_\R f^{I_t^{(1)} \mid X_0,X_1=y} (I_t^{(1)})  \rd F^{X_1 \mid X_0} (y)}.
\end{equation}
\end{proof}
To simplify the expressions, we introduce the following notation:
\begin{equation}
    u(t, z, X_0, y) :=  f^{I_t^{(1)} \mid X_0,X_1=y} (z), \quad u_t(t, z, X_0, y) := \frac{\partial}{\partial t}u(t, z, X_0, y),
\end{equation}
\begin{equation}
    u_z(t, z, X_0, y) := \frac{\partial }{\partial z}u(t, z, X_0, y), \quad u_{z z} (t, z, X_0, y) := \frac{\partial^2 }{\partial z^2}u(t, z, X_0, y),
\end{equation}
\begin{equation}
    K_{\boldsymbol{\cdot}} (t,z,X_0):=\int_\R  u_{\boldsymbol{\cdot}}  (t, z, X_0, y) \rd F^{X_1 \mid X_0} (y),
\end{equation}
\begin{equation}
    V_{\boldsymbol{\cdot}}  (t,z,X_0):=\int_\R y u_{\boldsymbol{\cdot}}  (t, z, X_0, y)\rd F^{X_1 \mid X_0} (y).
\end{equation}
Thus, under the conditions of Proposition $\ref{Bayes}$, we may write Eq. (\ref{Bayeseq}) as
\begin{equation}
     M_t =  \frac{V(t,I_t^{(1)},X_0)}{K(t,I_t^{(1)},X_0)}.
\end{equation}
Several examples of one-arc FAMs are provided in \ref{appendixB}. Next, we we go about deriving the stochastic differential equation for a class of one-arc FAMs.
\begin{prop} \label{Ito}
Let $M_t = \E [ X_1 \mid X_0, I_t^{(1)} ]$ be a one-arc FAM. If the $X$-RAP $(I_t^{(1)})$ is a semimartingale such that $(t,x) \rightarrow \frac{V(t,x,X_0)}{K(t,x,X_0)}$ is $C^2(((T_0,T_1) \setminus N) \times Im(I^{(1)}))$ where $N \subset (T_0,T_1)$ contains finitely many elements, then
\begin{align}
    \rd M_t =& \frac{V_t(t,I_t^{(1)},X_0)-M_t K_t(t,I_t^{(1)},X_0) }{K(t,I_t^{(1)},X_0)} \rd t + \frac{V_z(t,I_t^{(1)},X_0)-M_t K_z(t,I_t^{(1)},X_0) }{K(t,I_t^{(1)},X_0)} \rd I^{(1)}_t \nn \\ 
    &+ \Bigg(  \frac{M_t K_z^2(t,I_t^{(1)},X_0) -K_z(t,I_t^{(1)},X_0)V_z(t,I_t^{(1)},X_0)}{K^2(t,I_t^{(1)},X_0)} \nn \\
     & \hspace{5cm}+ \frac{V_{z z}(t,I_t^{(1)},X_0)-M_t K_{z z}(t,I_t^{(1)},X_0) }{2K(t,I_t^{(1)},X_0)}  \Bigg)  \rd [I^{(1)}]_t.
\end{align}
for $t \in (T_0,T_1)$.
\end{prop}
\begin{proof}
    This is verified by a straightforward application of Itô's formula. 
\end{proof}
\begin{rem}
The condition $(t,x) \rightarrow \frac{V(t,x,X_0)}{K(t,x,X_0)}$ is $C^2(((T_0,T_1) \setminus N) \times Im(I^{(1)})))$, where $N \subset (T_0,T_1)$ contains finitely many elements, imposes implicit integrability conditions on $(X_0,X_1)$.
\end{rem}
Since Gaussian RAPs have been studied in detail in previous sections, we specialise Proposition \ref{Ito} to this particular subclass. 
\begin{coro} \label{Itocoro1}
Under the conditions in Proposition \ref{Ito}, if the conditional probability distribution of $I_t^{(1)}$ given $(X_0,X_1)$ is $\mathcal N (g_0(t) X_0 + g_1(t) X_1 + \mu_A(t), \sigma_A^2(t))$ for all $t\in (T_0,T_1)$, we have
\begin{align}
    &\frac{V_t(t,I_t^{(1)},X_0)-M_t K_t(t,I_t^{(1)},X_0) }{K(t,I_t^{(1)},X_0)} = -\frac{(\E[X_1^3 \mid X_0, I_t^{(1)}] - M_t^3) g_1(t) \left(\frac{g_1(t)}{\sigma_A(t)}\right)'}{\sigma_A(t)} \nn \\
    & \hspace{1cm} + \Bigg(  \frac{M_t g_1(t) \left(\frac{g_1(t)}{\sigma_A(t)}\right)' + (I_t^{(1)} - g_0(t) X_0- \mu_A(t))\left( \left(\frac{g_1(t)}{\sigma_A(t)}\right)' + g_1(t)\left(\frac{1}{\sigma_A(t)}\right)'\right) }{\sigma_A(t)}  \nn \\
     & \hspace{6cm}- \frac{(X_0 g_0'(t) + \mu_A'(t))g_1(t)  }{\sigma_A^2(t)}  \Bigg) \times \var [ X_1 \mid X_0, I_t^{(1)} ],\\
     &\frac{V_z(t,I_t^{(1)},X_0)-M_t K_z(t,I_t^{(1)},X_0) }{K(t,I_t^{(1)},X_0)} = \frac{g_1(t)}{\sigma_A^2(t)} \var [ X_1 \mid X_0, I_t^{(1)} ], \\
    &\frac{M_t K_z^2(t,I_t^{(1)},X_0) -K_z(t,I_t^{(1)},X_0)V_z(t,I_t^{(1)},X_0)}{K^2(t,I_t^{(1)},X_0)}  \nn \\
    & \hspace{4cm}= \frac{I_t^{(1)} -g_0(t)X_0 - \mu_A(t)- g_1(t)M_t }{\sigma_A^4(t)} g_1(t) \var [ X_1 \mid X_0, I_t^{(1)} ],\\
    & \frac{V_{z z}(t,I_t^{(1)},X_0)-M_t K_{z z}(t,I_t^{(1)},X_0) }{2K(t,I_t^{(1)},X_0)} = \frac{\E[X_1^3 \mid X_0, I_t] - M_t^3}{2\sigma_A^4(t)}  g_1^2(t) \nn \\
    &\hspace{4cm} - \frac{2(I_t^{(1)} -g_0(t)X_0  - \mu_A(t)) + g_1(t)M_t} {2\sigma_A^4(t)} g_1(t)\var [ X_1 \mid X_0, I_t^{(1)} ].
\end{align}
\end{coro}
\begin{proof}
Let $Z_t:= I_t^{(1)} -g_0(t)X_0  - \mu_A(t)$ and $J_t:= X_0\, g_0'(t) + \mu_A'(t)$. The following computations of the partial derivatives $V_t, K_t, V_z,V_{zz}, K_{zz}$ in the case where $u(t,z,X_0,y)= (1/\sqrt{2\pi \sigma_A^2(t)}) \\ \exp [-(z-g_0(t) X_0 - g_1(t) y - \mu_A(t))^2/(2\sigma^2_A(t))]$ lead to the result:
\begin{enumerate}
    \item \begin{align}
        &V_t(t,I_t^{(1)},X_0) = \frac{g_1(t) \sigma_A'(t) - g_1'(t) \sigma_A(t) }{ \sigma_A^3(t)} g_1(t) K (t,I_t^{(1)},X_0) \E[X_1^3 \mid X_0, I_t^{(1)}] \nn \\
        &\hspace{5cm} + \frac{Z_t J_t\sigma_A(t)  + \sigma_A'(t)(Z_t^2 - \sigma_A^2(t) )}{\sigma_A^3(t)} V(t,I_t^{(1)},X_0) \nn \\
        &\quad - \frac{g_1(t) (J_t \sigma_A(t) + 2 Z_t \sigma_A'(t)) - Z_t \sigma_A(t) g_1'(t) }{\sigma_A^3(t)} K(t,I_t^{(1)},X_0) \E[X_1^2 \mid X_0, I_t^{(1)}],
    \end{align}
    \item \begin{align}
        K_t(t,I_t^{(1)},X_0) &= \frac{g_1(t) \sigma_A'(t) - g_1'(t) \sigma_A(t) }{ \sigma_A^3(t)} g_1(t) K (t,I_t^{(1)},X_0) \E[X_1^2 \mid X_0, I_t^{(1)}] \nn \\
        &\quad + \frac{Z_t J_t\sigma_A(t)  + \sigma_A'(t)(Z_t^2 - \sigma_A^2(t) )}{\sigma_A^3(t)} K(t,I_t^{(1)},X_0) \nn \\
        &\quad - \frac{g_1(t) (J_t \sigma_A(t) + 2 Z_t \sigma_A'(t)) - Z_t \sigma_A(t) g_1'(t) }{\sigma_A^3(t)} K(t,I_t^{(1)},X_0) M_t,
    \end{align}
    \item \begin{equation}
        V_z(t,I_t^{(1)},X_0) = \frac{g_1(t) \E[X_1^2 \mid X_0, I_t^{(1)}] K (t,I_t^{(1)},X_0)  - Z_t V (t,I_t^{(1)},X_0)} {\sigma_A^2(t)},
        \end{equation}
    \item \begin{equation}
        K_z(t,I_t^{(1)},X_0) = \frac{g_1(t) V (t,I_t^{(1)},X_0)  - Z_t K (t,I_t^{(1)},X_0)} {\sigma_A^2(t)},
        \end{equation}
    \item \begin{align}
        V_{zz}(t,I_t^{(1)},X_0) &= \frac{g_1^2(t) E[X_1^3 \mid X_0, I_t^{(1)}] - 2 Z_t g_1(t) \E[X_1^2 \mid X_0, I_t^{(1)}]}{\sigma_A^4(t)} K (t,I_t^{(1)},X_0) \nn \\
        &\quad + \frac{Z_t^2 - \sigma_A^2(t)}{\sigma_A^4(t)} V(t,I_t^{(1)},X_0),
    \end{align}
    \item \begin{align}
        K_{zz}(t,I_t^{(1)},X_0) &= \frac{g_1^2(t) E[X_1^2 \mid X_0, I_t^{(1)}]K(t,I_t^{(1)},X_0) - 2 Z_t g_1(t) V(t,I_t^{(1)},X_0)}{\sigma_A^4(t)} \nn \\
        &\quad \times K (t,I_t^{(1)},X_0) + \frac{Z_t^2 - \sigma_A^2(t)}{\sigma_A^4(t)} K(t,I_t^{(1)},X_0).
    \end{align}
\end{enumerate}
\end{proof}
Keeping the notation $Z_t= I_t^{(1)} -g_0(t)X_0  - \mu_A(t)$ and $J_t= X_0\, g_0'(t) + \mu_A'(t)$ introduced in the proof, and introducing $U_t= \E[X_1^3 \mid X_0, I_t] - M_t^3$, the SDE satisfied by $(M_t)$ can be rewritten as
\begin{align}
    \rd M_t =& U_t\left(\frac{g_1^2(t)}{2\sigma_A^4(t)} \rd [I^{(1)}]_t - \frac{ g_1(t) \left(\frac{g_1(t)}{\sigma_A(t)}\right)'}{\sigma_A(t)} \rd t \right) + \frac{g_1(t)}{\sigma_A^2(t)} \var [ X_1 \mid X_0, I_t^{(1)} ] \nn \\
    &\times \Bigg( \Bigg(  M_t \sigma_A(t) \left(\frac{g_1(t)}{\sigma_A(t)}\right)' + Z_t \frac{\sigma_A(t)}{g_1(t)}\left( \left(\frac{g_1(t)}{\sigma_A(t)}\right)' + g_1(t)\left(\frac{1}{\sigma_A(t)}\right)'\right)  - J_t  \Bigg) \rd t \nn \\
    & \hspace{5cm}- \frac{3g_1(t)M_t }{2\sigma_A^2(t)} \rd [I^{(1)}]_t +\rd I^{(1)}_t \Bigg).
\end{align}
If, furthermore, $(I_t^{(1)})$ is a standard RAP (see Definition \ref{standardRAP}), its driver $(D_t)$ is Gauss-Markov with $K_D(x,y)=H_1(\min(x,y))H_2(\max(x,y))$, where $H_1$ and $H_2$ are continuous functions on $[T_0,T_1]$ such that $H_1/H_2$ is positive and non-decreasing on $[T_0,T_1]$. Then, as shown in \cite{Kassis2},
\begin{equation}
    [I^{(1)}]_t = [D]_t=\int_{T_0}^t H_2(s) \rd H_1(s) - \int_{T_0}^t H_1(s) \rd H_2(s),
\end{equation}
where the right-hand side is interpreted as a difference of Riemann-Stieltjes integrals. The right-hand side exists since
\begin{equation}
    \int_{T_0}^t H_2(s) \rd H_1(s) - \int_{T_0}^t H_1(s) \rd H_2(s) =\int_{T_0}^t H_2^2(s) \rd \left(\frac{H_1}{H_2}\right)(s),
\end{equation}
and $H_1/H_2$ is continuous and monotone, so of bounded variation and differentiable almost everywhere. Recalling that in the standard RAP case we have 
\[
g_1(x) = \frac{H_1(x)H_2(T_0) - H_1(T_0)H_2(x)}{H_1(T_1)H_2(T_0) - H_1(T_0) H_2(T_1)},
\]
$A_1(x)=g_1(x)H_2(T_1)$, and $ A_2(x)= (H_1(T_1)/H_2(T_1))H_2(x) - H_1(x)$, $\sigma_A^2(x)=A_1(x)A_2(x)$, we obtain the following expression for the SDE of the martingale $(M_t)_{t\in[T_0,T_1]}$.
\begin{coro} \label{Itocoro2}
Under the conditions of Proposition \ref{Ito}, if $(I_t^{(1)})$ is a standard $X$-RAP with driver covariance $K_D(x,y)=H_1(\min(x,y))H_2(\max(x,y))$, then 
\begin{dmath*}
    \rd M_t= \frac{\var [ X_1 \mid X_0, I_t^{(1)} ]}{H_1(T_1)H_2(t)-H_1(t)H_2(T_1)} 
    \times \left( \left( \frac{Z_t ( H_1'(t)H_2(T_1)  - H_1(T_1) H_2'(t)) - M_t ( H_1'(t)H_2(t)  - H_1(t) H_2'(t))}{H_1(T_1)H_2(t) - H_1(t) H_2(T_1) }  - J_t \right) \rd t + \rd I^{(1)}_t \right),
\end{dmath*}
where $Z_t= I_t^{(1)} -g_0(t)X_0  - \mu_A(t)$ and $J_t= X_0 g_0'(t) + \mu_A'(t)$.
\end{coro}
\begin{proof}
The result follows from the following calculations:
\begin{align}
        &1.\quad\frac{g_1^2(t)(H_1'(t)H_2(t)  - H_1(t) H_2'(t))}{2\sigma_A^4(t)}  - \frac{ g_1(t) \left(\frac{g_1(t)}{\sigma_A(t)}\right)'}{\sigma_A(t)}=0,\\
        &2.\quad M_t \sigma_A(t) \left(\frac{g_1(t)}{\sigma_A(t)}\right)' + Z_t \frac{\sigma_A(t)}{g_1(t)}\left( \left(\frac{g_1(t)}{\sigma_A(t)}\right)' + g_1(t)\left(\frac{1}{\sigma_A(t)}\right)'\right)\nn \\
        &\hspace{8cm} - \frac{3g_1(t)M_t (H_1'(t)H_2(t)  - H_1(t) H_2'(t))}{2\sigma_A^2(t)} \nn \\
        &\quad\,\,\,=\frac{Z_t ( H_1'(t)H_2(T_1)  - H_1(T_1) H_2'(t)) - M_t ( H_1'(t)H_2(t)  - H_1(t) H_2'(t))}{H_1(T_1)H_2(t) - H_1(t) H_2(T_1) }.
    \end{align}
\end{proof}
Introducing the notation 
\begin{dmath} \label{NSDE}
    \rd N_t{:=} \left( \frac{Z_t ( H_1'(t)H_2(T_1)  - H_1(T_1) H_2'(t)) - M_t ( H_1'(t)H_2(t)  - H_1(t) H_2'(t))}{H_1(T_1)H_2(t) - H_1(t) H_2(T_1) }  - J_t \right) \rd t + \rd I^{(1)}_t,
\end{dmath}
we can then write
\begin{equation}
    M_t= X_0 + \int_{T_0}^t \frac{\var [ X_1 \mid X_0, I_t^{(1)} ]}{H_1(T_1)H_2(t)-H_1(t)H_2(T_1)} \rd N_t.
\end{equation}
As shown below, the process $(N_t)_{t\in [T_0,T_1]}$, defined by the SDE (\ref{NSDE}) and initial condition $N_{T_0}=0$, is a martingale, and can be used to construct a \emph{standard Brownian motion on $[T_0,T_1]$}.
\begin{defn}
Let $I \subseteq \R^+$ be an interval or the real positive line. A stochastic process $(Y_t)_{t\in I}$ is said to be standard Brownian motion on an interval $[a,b] \subseteq I$ if there exists a standard Brownian motion $(B_t)_{t \geqslant 0}$ such that $Y_t\eqas B_{t-a}$ for all $t\in [a,b]$.
\end{defn}
\begin{prop} \label{propinnovations}
Under the conditions in Proposition \ref{Ito}, the process \\ $(W_t)_{t \in [T_0,T_1]}$ defined by 
\begin{equation} \label{innovations}
    W_t= \int_{T_0}^t \frac{1}{\sqrt{H_1'(s)H_2(s) - H_1(s) H_2'(s)}} \rd N_s
\end{equation}
is an $(\F_t^I)$-adapted standard Brownian motion on $[T_0,T_1]$.
\end{prop}
\begin{proof}
We begin with showing that $(N_t)_{t \in [T_0,T_1]}$ is an $(\F_t^I)$-martingale. We introduce the following notation: $h_1(t):= H_1'(t)H_2(T_1)  - H_1(T_1) H_2'(t)$, $h_2(t):= H_1'(t)H_2(t) - H_1(t) H_2'(t)$, $h_3(t):= H_1(T_1)H_2(t) - H_1(t) H_2(T_1)$, and
$$
S(X_0,T_0,t):= \int_{T_0}^t \frac{(g_0(u) X_0 + \mu_A(u))h_1(u)}{h_3(u)} \rd u. 
$$
Then, it follows that
\begin{equation}
    N_t = \int_{T_0}^t \frac{I^{(1)}_u h_1(u) - M_u h_2(u)}{h_3(u)} \rd u- S(X_0,T_0,t) - g_0(t)X_0 - \mu_A(t) + I^{(1)}_t.
\end{equation}
Let $(s,t) \in [T_0,T_1] ^2$ such that $s < t$. We next show that $\E [N_t \mid \F_s^I ] = N_s$. By the linearity property of the conditional expectation, we have
\begin{align} \label{master}
    &\E [N_t \mid \F_s^I ]= \int_{T_0}^s \frac{I^{(1)}_u h_1(u) - M_u h_2(u)}{h_3(u)} \rd u + \int_{s}^t \frac{\E[I^{(1)}_u \mid \F_s^I] h_1(u) }{h_3(u)} \rd u  \nn \\
    &\quad- M_s\int_{s}^t \frac{  h_2(u)}{h_3(u)} \rd u - (S(X_0,T_0,s) + S(X_0,s,t)) -g_0(t) X_0 - \mu_A(t) + \E [ I^{(1)}_t \mid \F_s^I ].
\end{align}
By Proposition \ref{relations}, the second and the last terms in the above expression of $\E [N_t \mid \F_s^I ]$ can be expressed as follows: 
\begin{align}
    &\E[I_t^{(1)} \mid \F^I_s]=  \left(g_0(t) - \frac{A_2(t)}{A_2(s)} g_0(s) \right)  X_0 + \left(g_1(t)- \frac{A_2(t)}{A_2(s)} g_1(s)\right) M_s\nn \\
    & \hspace{8cm} + \frac{A_2(t)}{A_2(s)} I_s^{(1)}  + \mu_A(t) +\frac{A_2(t)}{A_2(s)} \mu_A(s), \\
    &\int_{s}^t \frac{\E[I^{(1)}_u \mid \F_s^I] h_1(u) }{h_3(u)} \rd u = X_0 \int_s^t \frac{g_0(u)h_1(u)}{h_3(u)} \rd u + M_s\int_s^t \frac{g_1(u)h_1(u)}{h_3(u)} \rd u  \nn \\
    & + \int_s^t \frac{\mu_A(u) h_1(u)}{h_3(u)} \rd u+   \frac{\left( I_s^{(1)} - g_1(s)M_s - g_0(s)X_0 + \mu_A(s) \right)}{A_2(s)}\int_s^t \frac{A_2(u) h_1(u)}{h_3(u)} \rd u.
\end{align}
Moreover, we notice that 
\begin{equation}
    \frac{1}{A_2(s)}\int_s^t \frac{A_2(u) h_1(u)}{h_3(u)} \rd u = \frac{H_1(t)-H_1(s) - \frac{H_1(T_1)}{H_2(T_1)} \left( H_2(t) - H_2(s) \right)}{A_2(s)}
    = 1- \frac{A_2(t)}{A_2(s)}.
\end{equation}
Hence, 
\begin{align}
    &\E[I_t^{(1)} \mid \F^I_s] + \int_{s}^t \frac{\E[I^{(1)}_u \mid \F_s^I] h_1(u) }{h_3(u)} \rd u \nn \\
    &= (g_0(t)-g_0(s) ) X_0 + (g_1(t)-g_1(s))  M_s + X_0 \int_s^t \frac{g_0(u)h_1(u)}{h_3(u)} \rd u \nn \\
    &\hspace{2.3cm} + M_s\int_s^t \frac{g_1(u)h_1(u)}{h_3(u)} \rd u + \int_s^t \frac{\mu_A(u) h_1(u)}{h_3(u)} \rd u +I_s^{(1)}  + \mu_A(s) + \mu_A(t) \nn \\
    &= (g_0(t)-g_0(s) ) X_0 + (g_1(t)-g_1(s))  M_s + S(X_0,s,t) \nn \\
    &\hspace{6cm} + M_s\int_s^t \frac{g_1(u)h_1(u)}{h_3(u)} \rd u +I_s^{(1)} + \mu_A(s) + \mu_A(t).
\end{align}
Next, we observe that
\begin{align}
    \int_s^t \frac{g_1(u)h_1(u)- h_2(u)}{h_3(u)} \rd u &= \int_s^t \frac{H_1'(u)H_2(T_0) - H_1(T_0)H_2'(u)}{H_1(T_1)H_2(T_0) - H_1(T_0) H_2(T_1)}\rd u= g_1(t)-g_1(s),
\end{align}
which allows one to write 
\begin{align} \label{work}
    &\E[I_t^{(1)} \mid \F^I_s] + \int_{s}^t \frac{\E[I^{(1)}_u \mid \F_s^I] h_1(u) }{h_3(u)} \rd u \nn \\
    &= (g_0(t)-g_0(s) ) X_0 +  M_s\int_{s}^t \frac{  h_2(u)}{h_3(u)} \rd u M_s + S(X_0,s,t) +I_s^{(1)}  + \mu_A(s) + \mu_A(t).
\end{align}
Inserting Eq. (\ref{work}) into Eq. (\ref{master}) gives: 
\begin{equation}
    \E [N_t \mid \F_s^I ] = \int_{T_0}^s \frac{I^{(1)}_u h_1(u) - M_u h_2(u)}{h_3(u)} \rd u- S(X_0,T_0,s) - g_0(s)X_0 - \mu_A(s) + I^{(1)}_s=N_s.
\end{equation}
Moreover, $\E [ \, \abs{N_t} ] < \infty$ for all $t \in [T_0,T_1]$. Hence $(N_t)_{t \in [T_0,T_1]}$ is an $(\F_t^I)$-martingale. We now show that $(W_t)$ is an $(\F_t^I)$-adapted standard Brownian motion on $[T_0,T_1]$. We compute the quadratic variation of $(W_t)$ and obtain
\begin{equation}
    [W]_t= \int_{T_0}^t \frac{1}{h_2(s)} \rd [N]_s = \int_{T_0}^t \frac{1}{h_2(s)} \rd [I^{(1)}]_s = \int_{T_0}^t \frac{h_2(s)}{h_2(s)} \rd s = t-T_0.
\end{equation}
Since $(N_t)$ is an $(\F_t^I)$-martingale and $\E[[W]_t] < \infty$, $(W_t)$ is an $(\F_t^I)$-martingale. Hence, by Lévy's characterisation theorem (Theorem 7.1 in \cite{Revuz}), $(W_t)_{t \in [T_0,T_1]}$ is an $(\F_t^I)$-adapted standard Brownian motion on $[T_0,T_1]$.
\end{proof}
The process $(W_t)_{t \in [T_0,T_1]}$ is referred to as the innovations process associated with the martingale $(M_t)$. It follows that we can write $(M_t)$ as an integral with respect its innovations process.
\begin{coro} \label{FAMSDEfinal}
Under the conditions of Proposition \ref{Ito}, it holds that
\begin{equation} \label{GMFAM}
    M_t = X_0 + \int_{T_0}^t \frac{\var [ X_1 \mid X_0, I_s^{(1)} ]\sqrt{H_1'(s)H_2(s)  - H_1(s) H_2'(s) } }{H_1(T_1)H_2(s)-H_1(s)H_2(T_1)} \rd W_s.
\end{equation}
\end{coro}
In the case that $(I^{(1)}_t)$ is an $X$-randomised anticipative Brownian bridge on $[T_0,T_1]$, that is,
\begin{equation}
    I_{t}^{(1)}=  B_t - \frac{T_1-t}{T_1-T_0} (B_{T_0}-X_0) - \frac{t-T_0}{T_1-T_0} (B_{T_1}-X_1),
\end{equation}
where $(B_t)_{t\geqslant 0}$ is a standard Brownian motion, the expressions become significantly simpler.
\begin{coro} \label{Ito2}
If $(I^{(1)}_t)_{t \in [T_0,T_1]}$ is an $X$-randomised anticipative Brownian bridge on $[T_0,T_1]$, the following holds:
\begin{enumerate}
    \item The process $(M_t)_{t\in[T_0,T_1]}$ satisfies the SDE
\begin{equation}
     M_t = X_0+ \int_{T_0}^t \frac{\var [ X_1 \mid X_0, I_t^{(1)} ]}{T_1-t} \left(  \frac{  I_t^{(1)} - M_t  }{T_1-t} \rd t+  \rd I^{(1)}_t \right).
\end{equation} 
    \item The process $(W_t)_{t \in [T_0,T_1]}$, defined by 
\begin{equation}
     W_t :=   \int_{T_0}^t\frac{ I_u^{(1)} - M_u  }{T_1-u} \rd u +   I^{(1)}_t - X_0,
\end{equation}
is an $(\F_t^I)$-adapted standard Brownian motion on $[T_0,T_1]$.
\end{enumerate}
\end{coro}
\begin{proof}
It suffices to set $H_1(x)=x$, $H_2(x)=1$, $g_0(t)=(T_1-t)/(T_1-T_0)$, and $g_1(t)=(t-T_0)/(T_1-T_0)$ in Eqs. (\ref{GMFAM}) and (\ref{innovations}).
\end{proof}
When $X_0=0$, $X_1$ is centred around $0$, and $T_0=0$, the randomised anticipative Brownian bridge gives rise to the martingale developed in \cite{BHM1}, i.e., 
\begin{equation}
    M_t=  \int_{0}^t \frac{\var [ X_1 \mid  I_s^{(1)} ]}{T_1-s} \rd W_s.
\end{equation}
Using Eq. (\ref{GMFAM}), we give examples of FAMs where the underlying RAP is standard (see Definition \ref{standardRAP}).

In general, one can expect that for a given coupling $\pi^X$, $M_t=\E[X_1 \mid \F_t^I]$ has no explicit analytical expression, even in the case where $M_t=\E[X_1 \mid X_0, I_t^{(1)}]$. Nonetheless, we give an example where $M_t=\E[X_1 \mid X_0, I_t^{(1)}]$ and its SDE are explicit. 
\begin{ex} \label{ex418}
Let $X_0 \sim \textrm{Uni}\,(\{-1,1\})$ and $X_1 \sim \frac{1}{4}\delta_{-2}+\frac{1}{2}\delta_{0} + \frac{1}{4}\delta_{2}$ such that 
\begin{equation}
X_1 \mid X_0  = \left\{
    \begin{array}{ll}
          X_0 +1  & \mbox{with probability } \frac{1}{2},\\ \\
         X_0 - 1  & \mbox{with probability } \frac{1}{2}.
    \end{array}
\right.
\end{equation}
Clearly, $\E[X_1 \mid X_0 ] = X_0$. Let $(B_t)_{t \geqslant0}$ be a standard Brownian motion and
\begin{equation}
    I_{t}^{(1)}= B_t - \frac{T_1-t}{T_1-T_0} (B_{T_0}-X_0) - \frac{t-T_0}{T_1-T_0} (B_{T_1}-X_1).
\end{equation}
Moreover, let $\phi$ denote the density function of the Gaussian measure \\ $\mathcal N\left(0, (T_1-t)(t-T_1)/(T_1-T_0)\right)$, and recall that $g_1(x)= (t-T_0)/(T_1-T_0)$. By Proposition \ref{Bayes}, we have:
\begin{align}
    M_t&= \frac{(X_0+1)\phi(I_t^{(1)}-X_0-g_1(t))  \frac{1}{2} + (X_0-1)\phi(I_t^{(1)}-X_0+g_1(t))  \frac{1}{2}}{\phi(I_t^{(1)}-X_0-g_1(t))  \frac{1}{2} + \phi(I_t^{(1)}-X_0+g_1(t))  \frac{1}{2}}, \\
    &= X_0 + \frac{\phi(I_t^{(1)}-X_0-g_1(t)) - \phi(I_t^{(1)}-X_0+g_1(t))}{\phi(I_t^{(1)}-X_0-g_1(t)) + \phi(I_t^{(1)}-X_0+g_1(t))}, \\
    &= X_0 + \tanh \left( \frac{I_t^{(1)}-X_0}{T_1-t}\right).
\end{align}
The associated SDE is given by the integral form
\begin{equation}
    M_t= X_0+ \bigintss_{T_0}^t \frac{ \sech^2 \left( \frac{I_s-X_0}{T_1-s} \right) }{T_1-s} \rd W_s.
\end{equation}
\begin{figure}[H]
   \centering
   \begin{minipage}[t]{0.49\textwidth}
    \includegraphics[width=\textwidth]{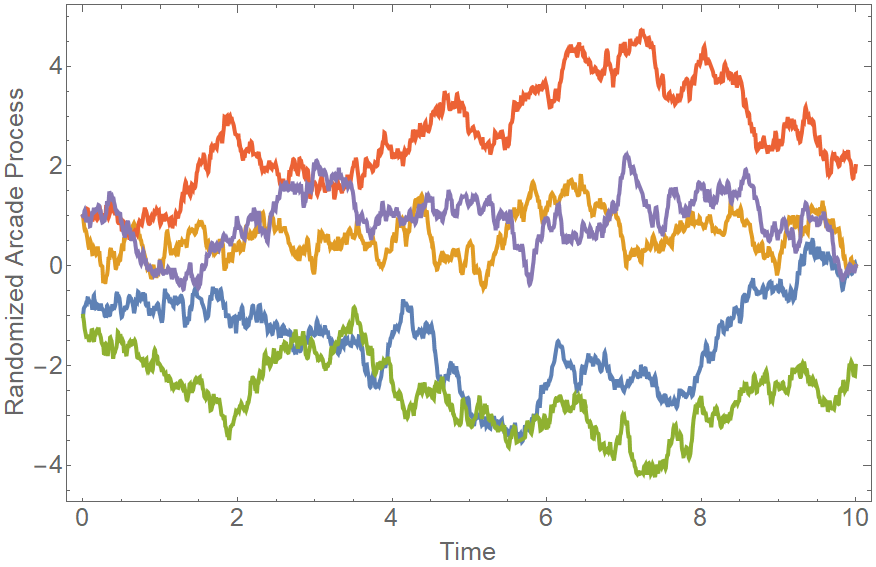}
     \caption{Paths of the $X$-RAP $(I_{t}^{(1)})$, cf., Example \ref{ex418}.}
   \end{minipage}
   \hfill
   \begin{minipage}[t]{0.49\textwidth}
     \includegraphics[width=\textwidth]{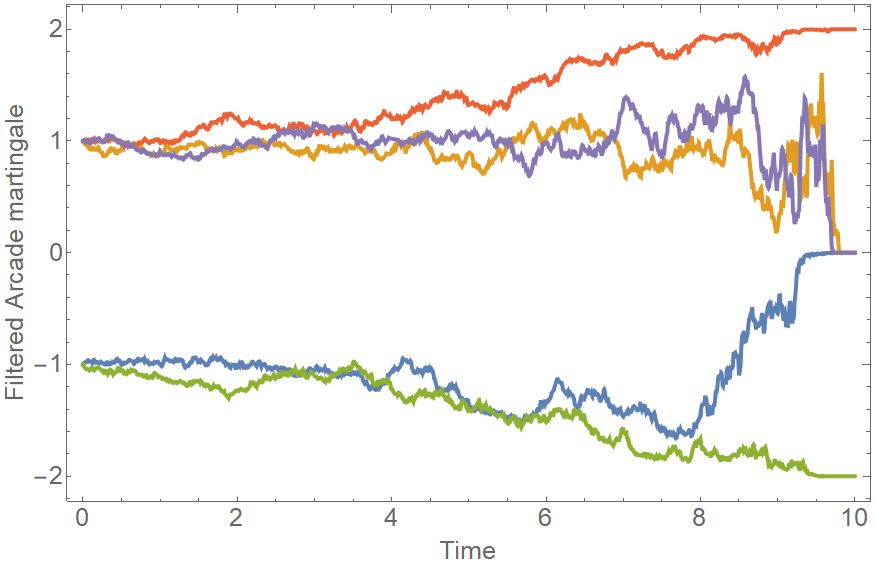}
     \caption{Paths of the associated FAM $(M_t)$, cf., Example \ref{ex418}.}
   \end{minipage}
 \end{figure}
\end{ex}
\subsection{Filtered arcade reverse martingales}
Before moving on to the n-arc FAMs, one can apply the ideas developed so far to a different set of problems. Let $X=(X_0,X_1)$ be a random vector distributed according to a coupling $\pi^X \in \Pi (\mu_0,\mu_1)$, where $(\mu_0,\mu_1) \in \mathcal P_1 (\R) \times \mathcal P_1 (\R)$. Instead of predicting $X_1$ using the information revealed by an $X$-RAP, one might want to predict $X_0$ given different scenarios of the future generated by an $X$-RAP. We illustrate this idea with a brief thought experiment. 

At the start of a farming season, a farmer might wonder how the production of his crops should be kicked off to meet the target yield by the end of the season, while taking into account the adverse weather forecast for the mid-season. If $X_0$ is the total amount produced by the end of the first day and $X_1$ is the total production at the end of the last day of the season, the farmer has a good idea about what $\pi^X$ is given his experience over many seasons. So, how should the farmer proceed? The answer is to use a filtered arcade reverse martingale (FARM).

\begin{defn}
 Given an $X$-RAP $(I_t^{(1)})$, a one-arc FARM for $X\sim \pi^X$ on $[T_0,T_1]$ is a stochastic process defined by $M_t^-=\E[X_0 \mid \mathcal G_t^I ]$, where $\mathcal G_t^I:= \sigma ( I_u^{(1)} \mid t \leqslant u \leqslant T_1)$ and $t \in [T_0,T_1]$.
\end{defn}
Unlike a FAM, a FARM is not a martingale, and is not necessarily interpolating between $X_0$ and $X_1$ either. Instead, it is a reverse martingale (also called backward martingale) with respect to the reverse filtration $({\mathcal G}_t^I)$, which follows from the tower property of conditional expectation: $\E[M_s^- \mid {\mathcal G}_t^I]=\E[\E[X_0 \mid {\mathcal G}_s^I ]\mid {\mathcal G}_t^I]=\E[X_0 \mid {\mathcal G}_t^I ]=M_t^-$ for any $(s,t) \in [T_0,T_1]^2$ such that $s\leqslant t$. A FARM only interpolates between $X_0$ and $X_1$ if $(X_0,X_1)$ is a discrete-time reverse martingale, i.e., $\E[X_0\mid X_1]=X_1$, since $M_{T_0}^- = X_0$ and $M_{T_1}^-=\E[X_0\mid X_1]$ by construction. 

To write a FARM $M_t^-$ as a function of $I_t^{(1)}$ and $X_1$, one needs another relaxation of the Markov property, as hinted at in Remark \ref{remreverse}. 
\begin{defn}
Let $\mathcal I  \subseteq \R^+$ be a real interval and $ \tau_0>\tau_1 > \ldots  \geqslant \inf \mathcal I$ such that $\tau=\{\tau_0, \tau_1, \ldots\} \subset \mathcal I$. The set $\tau$ may be finite or contain countably many elements. A stochastic process $(Y_t)_{t \in \mathcal I}$ is called reverse $\tau$-conditionally Markov if 
\begin{equation}
    \P\left[Y_s \in \cdot \mid  \mathcal G^Y_t \right] = \P\left[Y_{s} \in \cdot \mid Y_t, Y_{\tau_{m(t)}}, \ldots, Y_{\tau_0}  \right]
\end{equation}
for any $(s,t) \in \mathcal I^2$ such that $s \leqslant t$, where $\mathcal G^Y_t= \sigma ( Y_u \mid t \leqslant u \leqslant \sup \mathcal I)$ and $\tau_{m(s)} = \min\limits_{i\in \N} \{\tau_i \mid \tau_i \geqslant t\}$.
\end{defn}
Similarly to the FAM, a FARM can be expressed by
\begin{equation}
    M_t^-= \frac{\int_\R y f^{I_t^{(1)} \mid X_1,X_0=x} (I_t^{(1)})  \rd F^{X_0 \mid X_1} (x)}{\int_\R f^{I_t^{(1)} \mid X_1,X_0=x} (I_t^{(1)})  \rd F^{X_0 \mid X_1} (x)},
\end{equation}
for $t \in (T_0,T_1)$ if $(I_t^{(1)})$ is reverse $\{T_1,T_0\}$-conditionally Markov. In the previous section, we used standard RAPs to significantly simplify the SDEs that FAMs satisfy. For FARMs, the corresponding notion is the one of the \textit{reverse standard RAP}. 
\begin{defn}
A RAP $(I_t^{(n)})$ is said to be a reverse standard RAP if its noise process $(A_t^{(n)})$ is a standard AP, and if for all $x\in[T_0,T_n]$ and $j=0,\ldots,n-1$ it holds that $g_j(x) \1_{[T_{j+1},T_{n}]}(x)=0$, and $g_j(x) \1_{[T_{j},T_{j+1}]}(x)= f_j(x)\1_{[T_{j},T_{j+1}]}(x)$.
\end{defn}
Following the same procedure used for FAMs to derive their SDEs, while replacing the conditionally-Markov property by the reverse conditionally-Markov property, and standard RAPs by reverse standard RAPs, one finds that there exists a reverse standard Brownian motion $(W_t^-)_{t\in[T_0,T_1]}$ adapted to the reverse filtration $(\mathcal G_t^I)$ such that 
\begin{equation}
    M_t^- = X_0 + \int_{T_0}^t \frac{\var [ X_0 \mid I_s^{(1)}, X_1 ] \sqrt{H_1'(s)H_2(s)  - H_1(s) H_2'(s) }}{H_1(T_1)H_2(s)-H_1(s)H_2(T_1)} \rd W_s^-.
\end{equation}
\subsection{The n-arc FAM}
Let $X=(X_0,X_1,\ldots, X_n)$ be a random vector distributed according to a martingale coupling $\pi^X \in \mathcal M(\mu_0,\mu_1, \ldots, \mu_n)$, where $\mu_i \in \mathcal P_1 (\R)$ for $i=0,\ldots, n$ and $\mu_{i}\leqslant_{\mathrm{cx}}\mu_{i+1}$ for $i=0,\ldots, n-1$. Given an $X$-RAP $(I_t^{(n)})_{t \in [T_0,T_n]}$ on the partition $\{T_0,T_n\}_*$, we aim at constructing an $(\F^I_t)$-martingale $(M_t)_{t\in[T_0, T_n]}$ such that $M_{T_i}\eqas X_i$, for $i=0,\ldots, n$. A major difference in the $n$-arc case is that $(I_t^{(n)})$ must be $\{T_0,T_n\}_*$-conditionally Markov to guarantee the interpolation of $(M_t)$.
\begin{defn} \label{nFAM}
Given an $X$-RAP $(I_t^{(n)})_{t \in [T_0,T_n]}$ that is $\{T_0,T_n\}_*$-conditionally Markov, an $n$-arc FAM for $X\sim \pi^X$ is a stochastic process of the form $M_t = \E[X_n \mid \mathcal F_t^I] = \E[ X_n \mid X_0, \ldots, X_{m(t)}, I_t^{(n)} ]$ where $m(t)= \max \{i \in \N \mid T_i \leqslant t \}$.
\end{defn}
By the tower property of conditional expectation, $(M_t)_{t \in [T_0,T_n]}$ is a martingale with respect to $(\F_t^I)_{t \in [T_0,T_n]}$. Moreover, $M_{T_i}= \E[ X_n \mid X_0, \ldots, X_{i} ] = X_i $ since $\pi^X \in \mathcal M(\mu_0,\mu_1, \ldots, \mu_n)$.  Hence, $(M_t)$ is an interpolating martingale with respect to $(\F_t^I)$. We notice that, without the conditionally-Markov property, $M_{T_i}= \E[X_n \mid \mathcal F_{T_i}^I]$ is not necessarily a.s. equal to $X_i$. In other words, $X$ is not necessarily a martingale with respect to $(\F_t^I)$. More precisely, if we consider the filtration $(\F_t^X)_{t \in [T_0,T_n]}$ generated by the martingale step process $(X_t)_{t \in [T_0,T_n]}$ equal to $X_i$ on $[T_i, T_{i+1})$ for $i=0,\ldots, n-1$ and to $X_n$ on $[T_{n-1},T_n]$, we have that $\F_t^X \subseteq \F_t^I$ for all $t\in[T_0,T_n]$ by definition of a RAP, but $(\F_t^X)$ is not immersed\footnote{We say that $\left(\mathcal{F}_t\right)$ is immersed in $\left(\mathcal{G}_t\right)$ under $\mathbb{P}$ if every $\left(\mathcal{F}_t\right)$-martingale is also a $\left(\mathcal{G}_t\right)$-martingale. The immersion property ensures that enlarging the filtration (i.e., gaining more information) does not alter the martingale property of processes adapted to the smaller filtration, meaning that the additional information in $\left(\mathcal{G}_t\right)$ does not allow for "better prediction" of the future of $\left(\mathcal{F}_t\right)$-martingales.} in $(\F_t^I)$ under $\P$ in general. By restricting to RAPs that satisfy the $\{T_0,T_n\}_*$-conditionally-Markov property, $(\F_t^X)$ becomes immersed in $(\F_t^I)$ under $\P$. Notice that for $n=1$, $(\F_t^X)$ is immersed in $(\F_t^I)$ under $\P$ in general. This is why imposing the conditionally-Markov property in the definition of a one-arc FAM is not necessary.

As for the case $n=1$, in the remainder of this section, we assume that the driver $(D_t)$ of $(I_t^{(n)})$ has a density function on $[T_0,T_n]$.
\begin{prop}
Let $M_t = \E[ X_n \mid X_0, \ldots, X_{m(t)}, I_t^{(n)} ]$ be an $n$-arc FAM restricted to $(T_0,T_n)_*$. Then,
\begin{equation}
    M_t = \frac{\int_\R y f^{I_t^{(n)} \mid X_0,\ldots,X_{m(t)},X_n=y} (I_t^{(n)})    \rd F^{X_n \mid X_{0}, \ldots, X_{m(t)}}}{\int_\R f^{I_t^{(n)} \mid  X_0,\ldots,X_{m(t)},X_n=y} (I_t^{(n)})  \rd F^{X_n \mid X_{0}, \ldots, X_{m(t)}}},
\end{equation}
where $F^{X_n \mid X_{0}, \ldots, X_{m(t)}}$ is the distribution function of $X_n$ conditional on random vector $(X_{0}, \ldots, X_{m(t)})$. In particular:
\begin{enumerate}
    \item If $(X_{0}, \ldots, X_{m(t)},X_n)$ is a continuous random variable, with density function $f^{X_n \mid X_{0}, \ldots, X_{m(t)}}$, then 
    \begin{equation}
    M_t = \frac{\int_\R y f^{I_t^{(n)} \mid X_0,\ldots,X_{m(t)},X_n=y} (I_t^{(n)})   f^{X_n \mid X_{0}, \ldots, X_{m(t)}}(y) \rd y}{\int_\R f^{I_t^{(n)} \mid  X_0,\ldots,X_{m(t)},X_n=y} (I_t^{(n)}) f^{X_n \mid X_{0}, \ldots, X_{m(t)}}(y) \rd y}.
\end{equation}    
\item If $(X_{0}, \ldots, X_{m(t)},X_n)$ is a discrete random variable, then 
\begin{equation}
    M_t = \frac{\sum\limits_{y} y f^{I_t^{(n)} \mid X_{0}, \ldots, X_{m(t)},X_n=y} (I_t^{(n)}) \P[X_n=y \mid X_{0}, \ldots, X_{m(t)}]  }{\sum\limits_{y}  f^{I_t^{(n)} \mid X_{0}, \ldots, X_{m(t)},X_n=y} (I_t^{(n)}) \P[X_n=y \mid X_{0}, \ldots, X_{m(t)}]} .
\end{equation}
\end{enumerate}
\end{prop}
The proof follows along the lines of the one-arc case. Introducing the notations
\begin{equation}
    u(I_t^{(n)},t,x_{m +1},\ldots,x_{n-1},y)= f^{I_t^{(n)} \mid X_0,\ldots,X_{m(t)}, X_{m(t)+1}= x_{m +1}, \ldots, X_{n-1}=x_{n-1}, X_n=y} (I_t^{(n)})
\end{equation}
and $\tilde F:=F^{X_{m(t)+1}, \ldots, X_{n-1}, X_n \mid X_{0}, \ldots, X_{m(t)}}$, it is often more convenient to write
\begin{dmath}
    M_t = \frac{\int_{\R^{n-m(t)}} y\, u(I_t^{(n)},t,x_{m +1},\ldots,x_{n-1},y)  \rd \tilde F ( x_{m +1}, \ldots,x_{n-1},y)}{\int_{\R^{n-m(t)}} u(I_t^{(n)},t,x_{m +1},\ldots,x_{n-1},y) \rd \tilde F ( x_{m +1}, \ldots,x_{n-1},y)}.
\end{dmath}
We can apply Itô's formula in the same fashion as in the one-arc case. An interesting pattern appears if the signal function of $(I_t^{(n)})$ satisfies $g_j(x) \1_{[T_0,T_{j-1}]}(x)=0$ for all $ j=1,\ldots, n$, and for all $ x \in [T_0,T_n]$, which significantly simplifies the $n$-arc case.
\begin{prop} \label{special}
Let $M_t = \E[ X_n \mid X_0, \ldots, X_{m(t)}, I_t^{(n)} ]$ be an $n$-arc FAM, where the signal function of $(I_t^{(n)})$ satisfies $g_j(x) \1_{[T_0,T_{j-1}]}(x)=0$ for all $ j=1,\ldots, n$, and for all $ x \in [T_0,T_n]$. Then, $M_t=\E[ X_{m(t)+1} \mid X_0, \ldots, X_{m(t)}, I_t^{(n)} ]$.
\end{prop}
\begin{proof}
The result is straightforward for $t \in \{T_0,T_n\}_*$, so we treat the case $t\in (T_0,T_n)_*$. Let $\varphi_t$ be the density function of $A_t^{(n)}$. Then,
\begin{align}
    u(I_t^{(n)},t,x_{m +1},\ldots,x_{n-1},y) &= \varphi_t\left( I_t^{(n)} - \sum_{i=0}^{m(t)} g_i(t) X_i -  g_{m(t)+1}(t) x_{m+1})\right),
\end{align}
since $g_j(x) \1_{[T_0,T_{j-1}]}(x)=0$ for all $ j=1,\ldots, n$, and for all $ x \in [T_0,T_n]$. Denoting by $\tilde F$ the conditional distribution of $(X_{m(t)+1},\ldots, X_n)$ given $(X_0,\ldots, X_{m(t)})$, i.e., 
\begin{equation}
    \tilde F(x_{m+1},\ldots, x_n)= F^{X_{m(t)+1}, \ldots, X_n | X_0,\ldots, X_{m(t)}} (x_{m+1},\ldots, x_n),
\end{equation}
we have
\begin{align}
        &\int_{\R^{n-m(t)}} u(I_t^{(n)},t,x_{m +1},\ldots,x_{n-1},y) \rd \tilde F (x_{m+1},\ldots, x_{n-1},y) \nn \\
        &\hspace{3cm} =\int_\R u(I_t^{(n)},t,x_{m +1},\ldots,x_{n-1},y) \rd F^{X_{m(t)+1} | X_0,\ldots, X_{m(t)}} (x_{m+1}).
\end{align}
It also follows that
\begin{align}
    &\int_{\R^{n-m(t)}} y  \rd \tilde F (x_{m+1},\ldots, x_n) \nn \\
    &= \int_{\R^{n-m(t)}} y \rd F^{X_n|  X_0, \ldots, X_{m(t)}, X_{m(t)+1}=x_{m+1},\ldots, X_{n-1}=x_{n-1} } (y) \nn \\
    & \hspace{5cm} \rd F^{X_{m(t)+1},\ldots, X_{n-1} |  X_0, \ldots, X_{m(t)} }(x_{m+1},\ldots, x_{n-1} )  \nn\\ 
    &= \int_{\R^{n-m(t)-1}} \E[X_n \mid  X_0, \ldots, X_{m(t)}, X_{m(t)+1}=x_{m+1},\ldots, X_{n-1}=x_{n-1} ] \nn \\
    & \hspace{5cm} \rd F^{X_{m(t)+1},\ldots, X_{n-1} |  X_0, \ldots, X_{m(t)} }(x_{m+1},\ldots, x_{n-1} ) \nn\\
    &= \int_{\R^{n-m(t)-1}} x_{n-1} \rd F^{X_{m(t)+1},\ldots, X_{n-1} |  X_0, \ldots, X_{m(t)} }(x_{m+1},\ldots, x_{n-1} ),
\end{align}
where we use the martingale property of $(M_t)$. Applying the same argument $m(t)$ times, one obtains
\begin{align}
    &\int_{\R^{n-m(t)}} y  \rd \tilde F (x_{m+1},\ldots, x_n) = \int_{\R} x_{m+1} \rd F^{X_{m(t)+1} |  X_0, \ldots, X_{m(t)}} (x_{m+1}).
\end{align}
Thus,
\begin{align}
    &\int_{\R^{n-m(t)}} y u(I_t^{(n)},t,x_{m +1},\ldots,x_{n-1},y) \rd \tilde F (x_{m+1},\ldots, x_{n-1},y) \nn \\
    &\hspace{2cm}= \int_{\R} x_{m+1}u(I_t^{(n)},t,x_{m +1},\ldots,x_{n-1},y) d F^{X_{m(t)+1} |  X_0, \ldots, X_{m(t)}} (x_{m+1}),
\end{align}
which leads to
\begin{align}
    M_t &= \frac{\int_{\R^{n-m(t)}} y u(I_t^{(n)},t,x_{m +1},\ldots,x_{n-1},y) \rd \tilde F (x_{m+1},\ldots, x_{n-1},y)}{\int_{\R^{n-m(t)}}  u(I_t^{(n)},t,x_{m +1},\ldots,x_{n-1},y) \rd \tilde F (x_{m+1},\ldots, x_{n-1},y)} \nn \\ \nn \\
    &= \frac{\int_{\R} x_{m+1}u(I_t^{(n)},t,x_{m +1},\ldots,x_{n-1},y) \rd F^{X_{m(t)+1} |  X_0, \ldots, X_{m(t)}} (x_{m+1})}{\int_\R u(I_t^{(n)},t,x_{m +1},\ldots,x_{n-1},y) \rd F^{X_{m(t)+1} | X_0,\ldots, X_{m(t)}} (x_{m+1})} \nn \\ \nn \\
    &= \int_{\R} x_{m+1} \rd F^{X_{m(t)+1} |  X_0, \ldots, X_{m(t)}, I_t^{(n)}} (x_{m+1})\nn \\ \nn \\
    &= \E[ X_{m(t)+1} \mid X_0, \ldots, X_{m(t)}, I_t^{(n)} ],
\end{align}
where we use the Bayes rule in the second-last step.
\end{proof}
This allows one to use the results for the one-arc case, see Section \ref{1arccase} to derive the SDE for the $n$-arc FAM.
\begin{prop} \label{suitable}
Let $M_t = \E [ X_n \mid  X_0, \ldots, X_{m(t)}, I_t^{(n)} ]$ be an $n$-arc FAM. If the $X$-RAP $(I_t^{(n)})$ is standard and a semimartingale, satisfying the condition in Proposition \ref{special}, with driver covariance $K_D(x,y)=H_1(\min(x,y))H_2(\max(x,y))$, such that $(t,x) \rightarrow \E [ X_n \mid  X_0, \ldots, X_{m(t)}, I_t^{(n)}=x ]$ is $C^2(((T_0,T_n)_* \setminus N) \times Im(I_t^{(n)}))$ where $N \subset (T_0,T_n)_*$ contains finitely many elements, then the following holds:
\begin{enumerate}
    \item The process $(M_t)_{t\in[T_0,T_n]}$ satisfies the equation
    \begin{align}
    M_t & = X_0 \nn \\ &+ \int_{T_0}^t \frac{\var [ X_{m(s)+1} \mid  X_0, \ldots, X_{m(s)}, I_t^{(n)} ]\sqrt{H_1'(s)H_2(s)  - H_1(s) H_2'(s))} }{H_1(T_{m(s)+1})H_2(s)-H_1(s)H_2(T_{m(s)+1})} \rd W_s,
    \end{align}
where 
\begin{equation}
    W_t = \int_{T_0}^t \frac{1}{\sqrt{H_1'(s)H_2(s) - H_1(s) H_2'(s)}} \rd N_s,
\end{equation}
\begin{align}
    &\rd N_t \nn \\
    &=\bigg ( \frac{Z_t ( H_1'(t)H_2(T_{m(t)+1})  - H_1(T_{m(t)+1}) H_2'(t)) - M_t ( H_1'(t)H_2(t)  - H_1(t) H_2'(t))}{H_1(T_{m(t)+1})H_2(t) - H_1(t) H_2(T_{m(t)+1}) }   \nn \\
    &\hspace{10cm} - J_t \bigg )\rd t+ \rd I^{(1)}_t,
\end{align}
 $Z_t= I_t^{(1)} - \sum\limits_{i=0}^{m(t)} g_i(t)X_i  - \mu_A(t)$, and $J_t=\sum\limits_{i=0}^{m(t)} g_i'(t)X_i  + \mu_A'(t)$.
    \item The process $(W_t)_{t\in[T_0,T_n]}$ is an $(\mathcal F^I_t)$-adapted standard Brownian motion on $[T_0,T_n]$.
\end{enumerate}
\end{prop}
The proof follows from the one-arc case, see Proposition \ref{Ito} to Corollary \ref{FAMSDEfinal}. Moreover, we note that 
\begin{align}
    &M_t \nn \\
    &= X_0 + \int_{T_0}^t \frac{\var [ X_{m(s)+1} \mid  X_0, \ldots, X_{m(s)}, I_t^{(n)} ]\sqrt{H_1'(s)H_2(s)  - H_1(s) H_2'(s))} }{H_1(T_{m(s)+1})H_2(s)-H_1(s)H_2(T_{m(s)+1})} \rd W_s \nn \\
    &= X_{m(t)} + \int_{T_{m(t)}}^t \frac{\var [ X_{m(s)+1} \mid  X_0, \ldots, X_{m(s)}, I_t^{(n)} ]\sqrt{H_1'(s)H_2(s)  - H_1(s) H_2'(s))} }{H_1(T_{m(s)+1})H_2(s)-H_1(s)H_2(T_{m(s)+1})} \rd W_s.
\end{align}

\section{Outlook on IB-MOT and applications}
In this paper, a new approach is developed that allows for strong interpolation of a sequence of random variables by deploying a generalisation of anticipative stochastic bridges, the RAPs. These processes generate filtrations that are used to construct FAMs, a class of processes that interpolate between the components of a discrete-time martingale, almost surely. For a fixed driver and interpolating coefficients, there is essentially a one-to-one mapping between discrete-time martingales and FAMs, i.e., no two FAMs with the same driver and interpolating coefficients interpolate the same discrete-time martingale. 

FAMs solve the so-called martingale interpolating problem on the space of random variables in a novel way. Relaxing this interpolation problem on the space of probability measures provides a new degree of freedom for FAMs: in law, there are as many FAMs sharing the same driver, interpolating coefficients, and connecting the target measures, as there are elements in the set of martingale couplings of the target measures. Selecting one of the FAMs in law amounts to selecting one of the martingale couplings. This is reminiscent of optimal transport (see \cite{Monge}, \cite{Kantorovich}, \cite{Brenier} for foundational work and \cite{Bogachev}, \cite{Villani1}, \cite{Villani2} for modern treatments), where the aim is to select an optimal coupling. The question of whether there is a connection to optimal transport and, if so, what relation it might be, arises naturally at this stage. 

Hence, in what follows, we investigate FAMs in law in the context of optimal transport and its noisy versions, since FAMs are inherently noisy. This then serves as a segue into considering martingale transport within the filtered arcade martingale paradigm, which we call {\it information-based martingale optimal transport} (IB-MOT), given the FAMs' in-built informational (Bayesian) updating of their trajectories to target. The development of IB-MOT is presented in the working paper \cite{IBMOT}.

The construction of a FAM $(M_t)$ requires a martingale vector $X$ and an $X$-RAP. In real-world situations, one might not have the vector $X$, but only its marginal distributions. In such a case, one needs to select a martingale coupling for these marginal distributions in order to construct a FAM in law. Could the behaviour of the FAMs associated with each possible coupling for the given marginal distributions inform us on which coupling to choose, i.e., is there an optimal coupling for a FAM? For each coupling, a FAM produces the best estimate of the last target random variable given the information generated by the $X$-RAP up until the time at which the estimation is performed. A natural question arises: among all the estimates, which coupling yields the worst-case scenario? This idea of choosing a coupling based on paths connecting its marginal distributions on the space of probability measures reminds us of optimal transport-like problems. Let us briefly summarise, for the purpose of clarifying the context of our train of thought, the salient points of optimal transport.

Optimal transport (OT) dates back to Gaspard Monge in 1781 \cite{Monge}, with significant advancements by Leonid Kantorovich in 1942 \cite{Kantorovich} and Yann Brenier in 1987 \cite{Brenier}. It provides a way of comparing two measures, $\mu$ and $\nu$, defined on the Borel sets of topological Hausdorff spaces $\X$ and $\Y$, respectively. We denote by $\mathcal P(\X)$ and $\mathcal P(\Y)$ the sets of Borel measures on $\X$ and $\Y$, respectively. A popular mental image in the context of optimal transport is a pile of sand, modelled by a measure $\mu$, and a hole, modelled by another measure $\nu$. One wishes to fill the hole with the sand in an optimal manner, by exercising the least amount of effort. To make this statement more precise, one needs a cost function $c: \X \times \Y \rightarrow [0,\infty]$ that measures the cost of transporting a unit mass from $x \in \X$ to $y\in\Y$. The optimal transport problem is concerned with how to transport $\mu$ to $\nu$ whilst minimising the cost of transportation. That is, given $\mu \in \mathcal{P}(\X)$ and $\nu \in \mathcal{P}(\Y)$,
\begin{equation} \label{chap6OT}
    \inf_{\pi \in \Pi(\mu,\nu)} \K(\pi) := \inf_{\pi \in \Pi(\mu,\nu)} \int_{\X \times \Y} c(x,y) \d \pi(x,y).
\end{equation}
This problem possesses many interesting properties. For instance, under mild conditions, it defines a collection of metrics denoted $W_p$ for any $p\geqslant 1$, the Wasserstein $p$-metric, on the space $\mathcal P_p(\X)$ of probability measures on $\X$ with finite $p$th moment. Furthermore, if $\X$ is Euclidean, and $c(x,y)=\norm{x-y}$, there is a one-to-one correspondence between the minimisers $\pi^*$ of $\inf_{\pi \in \Pi(\mu,\nu)} \int_{\X \times \Y} c(x,y)^p \d \pi(x,y)$, and the geodesics in $(\mathcal P_p(\X),W_{p})$. If $(X_0,X_1) \sim \pi^*$, then the law of the process 
\begin{equation}
    \left(\frac{T_1-t}{T_1-T_0} X_0 + \frac{t-T_0}{T_1-T_0} X_1\right)_{t\in[T_0,T_1]}
\end{equation}
is the shortest paths from $\mu$ to $\nu$ in $(\mathcal P_p(\X),W_{p})$. This links optimal transport to interpolation on the space of random variables. There are two main differences between a FAM for $(X_0,X_1)$ and $(X_0(T_1-t)/(T_1-T_0)  + X_1(t-T_0)/(T_1-T_0) )$. The latter process is not a martingale, and does not incorporate noise, resulting in a deterministic interpolation conditional on $(X_0,X_1)$. Could we modify OT in a way to incorporate noise and produce a martingale interpolator instead? Separately, these conditions are satisfied by well-studied offsprings of OT. The entropic regularisation of optimal transport (\cite{Leonard1}, \cite{Peyre1}), which coincides with Schrödinger's problem \cite{Schrodinger} under mild conditions, injects Brownian noise into OT, producing a noisy interpolator. On the other hand, martingale optimal transport (\cite{Beiglbock}, \cite{Beiglbock2}) replaces $\Pi(\mu,\nu)$ by $\mathcal M(\mu,\nu)$ in Eq. (\ref{chap6OT}). However, when imposing the martingale condition in OT, the interpolator corresponding to the coupling is lost, whereas when trying to inject noise in OT, the problem becomes the opposite of a martingale problem: one seeks to find a drift for a diffusion with a constant volatility. 

It is our view that FAMs could be used to reconcile the introduction of noise in martingale optimal transport (MOT) and martingale interpolation. Instead of comparing the distributions of $X_0$ and $X_1$ using the expectation $\E_\pi[c(X_0,X_1)]$, we propose to investigate the expectation of the cumulative cost in time on $[T_0,T_1]$ incurred by estimating $X_1$ by $M_t$. This would result in an information-based version of MOT, since the information flow generated by the $X$-RAP is taken into account in determining the optimal coupling. Mathematically, the problem can be stated as
\begin{equation}
    \sup_{(X_0,X_1) \sim \pi^X \in \mathcal M (\mu,\nu)} \E \left[ \int_{T_0}^{T_1} \frac{[X_1-M_t(X_0,I_t^{(1)})]^2 \sqrt{H_1'(t)H_2(t)  - H_1(t) H_2'(t) }}{H_1(T_1)H_2(t)-H_1(t)H_2(T_1)} \rd t \right],
\end{equation}
see Section \ref{FAMchapter} for the introduced notion. Because we consider the cumulative cost in time on $[T_0,T_1]$, we propose to multiply the cost $[X_1-M_t(X_0,I_t^{(1)})]^2$ by a weight function that informs the problem about the rate of convergence of $M_t$ towards $X_1$. By construction, the cost will always be small towards the end of the time interval, regardless of the coupling. The weight function allows for a fair comparison between the costs arising across different times by removing the bias introduced in the construction of $(M_t)$.

By maximising the expectation of the cumulative cost in time between $X_1$ and $M_t$ over the set of martingale couplings, and constructing the FAM corresponding to the maximising martingale coupling, one obtains the ``worst-case'' FAM in law given the information generated by the underlying RAP. The IB-MOT theory, and applications thereof, is developed in a follow-on paper \cite{IBMOT}.

Information-based stochastic interpolation can also be considered using FARMs $(M_t^-)$. For example, a government body may need to decide on future climate goals in terms of, e.g., atmospheric CO2 emissions. Future goals depend on a succession of intermediate achievements, which are reached with uncertainty. So, these goals can be modelled by random variables with probability distributions in concave order, if a governmental body is interested in progressively decreasing the uncertainty around the mean  until the final goal is reached. One can then look for an optimal reverse-martingale coupling linking the intermediate goals using the same procedure described above while replacing FAMs with FARMs: one aims at maximising the expectation of the cumulative cost in time between $X_0$ and $M_t^-$ over the set of reverse-martingale couplings. Then, constructing the FARM corresponding to the maximising reverse-martingale coupling, one obtains the ``worst-case'' FARM given the information generated by the underlying RAP, allowing one to understand if the goals are likely to be achieved, and if so, what trajectories must be taken in the near future to succeed.

We note that IB-MOT is not the only optimal transport problem that incorporates noise while maintaining the martingale property and producing an interpolator. The martingale Benamou-Brenier problem (MBBP) \cite{Veraguas} and Schrödinger's volatility models (SVMs) \cite{Labordere} are examples of such problems. IB-MOT is an addition to this class of problems, relying on filtering to select the optimal coupling instead.

\section{Conclusions}
Given a finite sequence of convexly ordered random variables that are indexed by pre-specified fixed dates, we show how continuous-time martingales, the so-called \textit{filtered arcade martingales} (FAMs, Definitions \ref{1FAM} and \ref{nFAM}) can be constructed explicitly such that the target random variables are matched at the given dates, almost surely. These interpolating martingales, defined as conditional expectations, filter noisy information about the last target random variable such that the interpolation through the whole sequence of random variables is guaranteed. This martingale-interpolating method relies on a specific class of stochastic processes, the \textit{randomized arcade processes} (RAPs, Definition \ref{RAP}). These processes, which interpolate strongly between the target random variables, generate information flows that reveal each target random variable at its indexing date, thus allowing a FAM to interpolate through the sequence of random variables. FAMs are solutions to filtering problems where the best estimate of the last target random variable is calculated given the filtration generated by the observation process, i.e., the RAP.

A randomised arcade process is the sum of two independent components: a signal function and a noise process, the latter is the so-called \textit{arcade process} (AP, Definition \ref{AP}). An AP is a functional of a stochastic driving process, and interpolates strongly between zeros. A well-known example of an AP over one period of time is the anticipative Brownian bridge (Example \ref{ABB}). Since APs play the role of noise, a natural class of APs are the ones driven by a Gaussian process. We give sufficient condition for such an AP to be Markovian (Corollary \ref{coromarkov}), and show that there are Markovian APs that are not anticipative representations of Markov bridges (Example \ref{nonstarcade}). As for the signal function, it interpolates between the target random variables, and it is deterministic conditional on these random variables. Since a RAP is a sum of a signal function and an AP, it does not possess the Markov property with respect to its natural filtration, except in trivial cases. Instead, it may have a similar property, the \textit{conditionally-Markov property} (Definition \ref{nearlymarkov}), which plays a central role for the computational tractability of a FAM, and for guaranteeing the interpolation in the case of more than two random variables. We give sufficient conditions for a RAP to satisfy the conditionally-Markov property (Theorem \ref{thmsemimarkov}), and we consider FAMs constructed with such RAPs. Using the Bayes rule and Itô’s lemma, we derive the SDEs satisfied by such FAMs (Propositions \ref{Bayes} and \ref{Ito}). If furthermore, the RAPs are \textit{standard} (Definition \ref{standardRAP}), the obtained FAMs can be expressed in terms of an integral with respect to a Brownian motion adapted to the filtration generated by the RAP underlying the FAM (Proposition \ref{propinnovations}). This Brownian motion corresponds to the innovations process in stochastic filtering theory. 

The FAM theory relies on convexly ordered probability measures, that is, the existence of a martingale coupling for the target variables. In some real life scenarios, one might not have the coupling, but only its marginal distributions. In that case, one must select a martingale coupling for these marginal distributions in order to construct a FAM. As a potential application in the field of optimal transport, that is studied in \cite{IBMOT}, we propose the use of the paths of FAMs to select an optimal coupling for the marginal distributions. This method is inspired by martingale optimal transport and Schrödinger’s problem. It aims at introducing noise in MOT in a similar fashion to how the entropic regularisation of optimal transport introduces noise in optimal transport. The noise contained in FAMs, which depends on the choice of the underlying RAP, is taken into consideration in the selection of an optimal coupling for the target random variables---we aim at maximising the expectation of the cumulative cost in time between the final target random variable and the considered FAM over the set of martingale couplings. The same ideas can be applied to FARMs, where we look instead for an optimal reverse-martingale coupling based on the paths of FARMs. Arcade processes and randomised arcade processes, alongside the filtered arcade martingales, suggest applications in, e.g., financial and insurance mathematics, mathematical biology, statistics, and climate science.

\section*{Acknowledgments}
A. Macrina is grateful for discussions on stochastic interpolation with Camilo A. Garcia Trillos and Jun Sekine, especially in the context of information-based modelling and stochastic filtering, during a research visit at Osaka University in March 2017 as part of a Daiwa Anglo-Japanese Foundation Small Grant, Ref: 3078/12370. Osaka University's hospitality is acknowledged and appreciated. A. Macrina is also grateful for the support by the Fields Institute for Research in Mathematical Sciences through the awards of a Fields Research Fellowship in 2022 and an Elliott-Yui Distinguished Visitorship in 2025. The contents of this manuscript are solely the responsibility of the authors and do not necessarily represent the official views of the Fields Institute. G. Kassis acknowledges the UCL Department of Mathematics for a Teaching Assistantship Award. Furthermore, the authors thank L. P. Hughston for helpful comments and suggestions at an early stage of the work on arcade processes, B. Acciaio and G. Pammer for pointers to optimal transport and stretched Brwonian motion, S. Cohen, J. Guyon, F. Krach, G. W. Peters, and T.-K. L. Wong for useful conversations, and J. Akahori, K. Owari and D. C. Schwarz for feedback. Attendees of the AIFMRM Research Seminar at the University of Cape Town (Aug. 2021) and participants in the SIAM Conference on Mathematics of Data Science (MDS22, Sept. 2022), the 7th International Conference Mathematics in Finance---MiF Kruger Park--- (July 2023, South Africa), the Talks in Financial and Insurance Mathematics at ETH Zurich (Nov. 2023, Switzerland), of the Department of Mathematics Seminar of Ritsumeikan University (Nov. 2023, Japan), and participants in the FAMiLLY Workshop at the University of Liverpool (Dec. 2023, U.K.), the XXV Workshop on Quantitative Finance (Apr. 2024, Italy), the Research Seminar of the Department of Mathematics \& Statistics, University of Ottawa (Nov. 2024, Canada), the Quantitative Methods in Finance (QMF) 2024 International Conference at the University of Technology Sydney (Dec. 2024, Australia), and in the 69th Annual Meeting of the Australian Mathematical Society (Dec. 2025, Australia) are thanked for their comments and suggestions. Last but not least, we are grateful to two anonymous reviewers and an associated editor for constructive comments that have helped with improving this paper.

\appendix
\section{Arcade processes}\label{appendixA}
In this appendix, we include further examples of arcade processes along with their simulations.
\begin{ex} \label{stitched}
For $n>1$, we can generalise the anticipative Brownian bridge by taking 
\begin{equation}
f_0(t)= \frac{T_{1}-t}{T_1-T_{0}} \1_{[T_0, T_1]} (t), \quad f_n(t)= \frac{t-T_{n-1}}{T_n-T_{n-1}} \1_{(T_{n-1}, T_n]} (t),
\end{equation}
and 
\begin{equation}
f_i(t)=\frac{t-T_{i-1}}{T_i-T_{i-1}} \1_{(T_{i-1}, T_{i}]} (t) + \frac{T_{i+1} - t}{T_{i+1}-T_i} \1_{(T_{i}, T_{i+1}]} (t), \quad \text{for } i=1,\ldots, n-1.
\end{equation} 
We call this AP the stitched Brownian AP, for it can be written as
\begin{equation}
A_{t}^{(n)}  = \left\{
    \begin{array}{ll}
         B_t - \frac{T_1-t}{T_1-T_0} B_{T_0} - \frac{t-T_0}{T_1-T_0} B_{T_1}, & \mbox{if } t \in [T_0,T_{1}), \\ \\
         B_t - \frac{T_2-t}{T_2-T_1} B_{T_1} - \frac{t-T_1}{T_2-T_1} B_{T_2}, & \mbox{if } t \in [T_1,T_{2}), \\ 
         \, \vdots \\
         B_t - \frac{T_n-t}{T_n-T_{n-1}} B_{T_{n-1}} - \frac{t-T_{n-1}}{T_n-T_{n-1}} B_{T_n}, & \mbox{if } t \in [T_{n-1},T_{n}].
    \end{array}
\right.
\end{equation}
\begin{figure}[H]
\centering
\includegraphics[width=.5\textwidth]{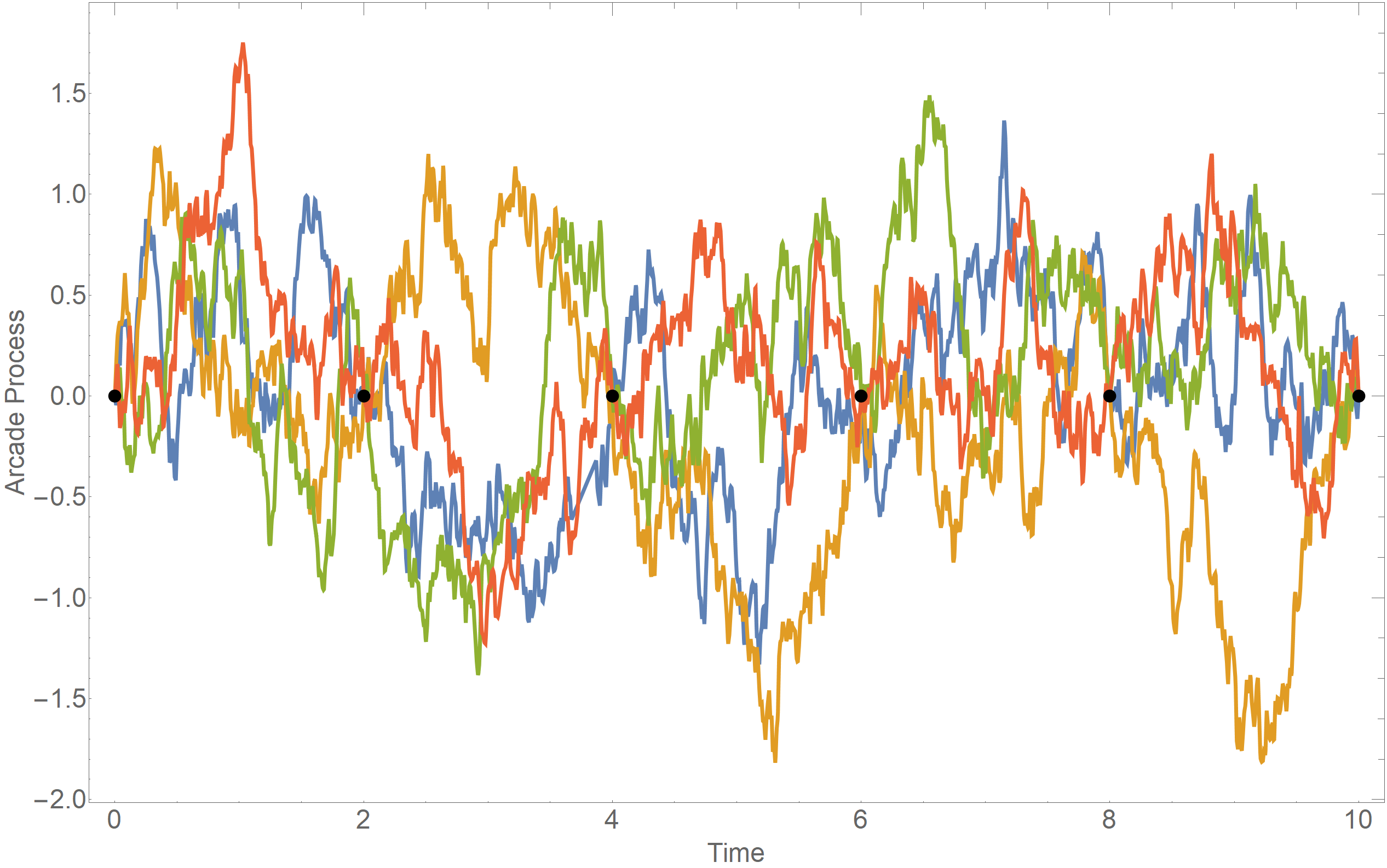}
\caption{Paths simulation of a stitched Brownian AP with $n=5$, using the equidistant partition $\{T_i = 2i \mid i=0,1,\ldots, 5\}$, cf., Example \ref{stitched}.}
\end{figure}
\end{ex}
\begin{ex} \label{ex25}
Another way of generalizing the anticipative Brownian bridge to obtain an AP driven by Brownian motion is by using Lagrange's polynomial interpolation, that is,
\begin{equation}
f_i(t)= \prod\limits_{k=0,k\neq i}^n\frac{ T_k-t}{ T_k-T_i} \quad \text{for } i=0,\ldots, n.
\end{equation}
We may call the resulting AP the Lagrange-Brownian AP, which has the form 
\begin{equation}
    A_t^{(n)}=B_t - \sum_{i=0}^n \prod\limits_{k=0,k\neq i}^n\frac{ T_k-t}{ T_k-T_i} B_{T_i}.
\end{equation}
\begin{figure}[H]
\centering
\includegraphics[width=0.5\textwidth]{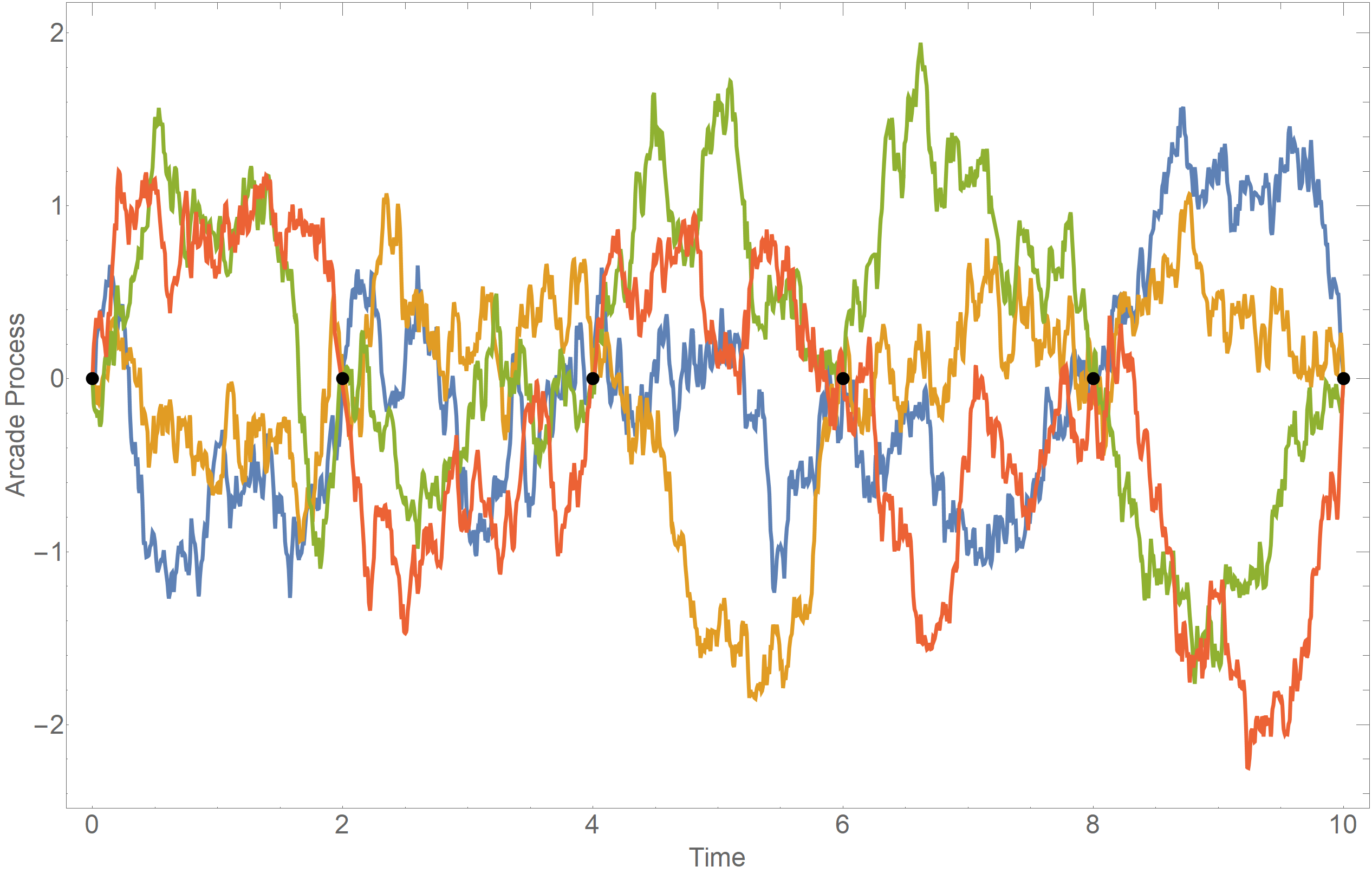}
\caption{Paths simulation of a Lagrange-Brownian AP with $n=5$, using the equidistant partition $\{T_i = 2i \mid i=0,1,\ldots, 5\}$, cf., Example \ref{ex25}.}
\end{figure}
The Lagrange APs inherit Runge's phenomenon from their interpolation coefficients. When $n$ is large, the AP oscillates around the edges of the interval. One can control this effect by applying a transformation to the interpolating coefficients. For instance, the map $x\mapsto |x|^{2(1-|x|)}$ applied to each of the interpolating coefficients $$ f_i(t)= \prod\limits_{k=0,k\neq i}^n\frac{ T_k-t}{ T_k-T_i}$$ for $i=0,\ldots, n$, yields another set of interpolating coefficients $\tilde f_0, \ldots, \tilde f_n$ which no longer suffer from Runge's phenomenon.
\begin{figure}[H]
\centering
\begin{minipage}[b]{0.49\textwidth}
\includegraphics[width=\textwidth]{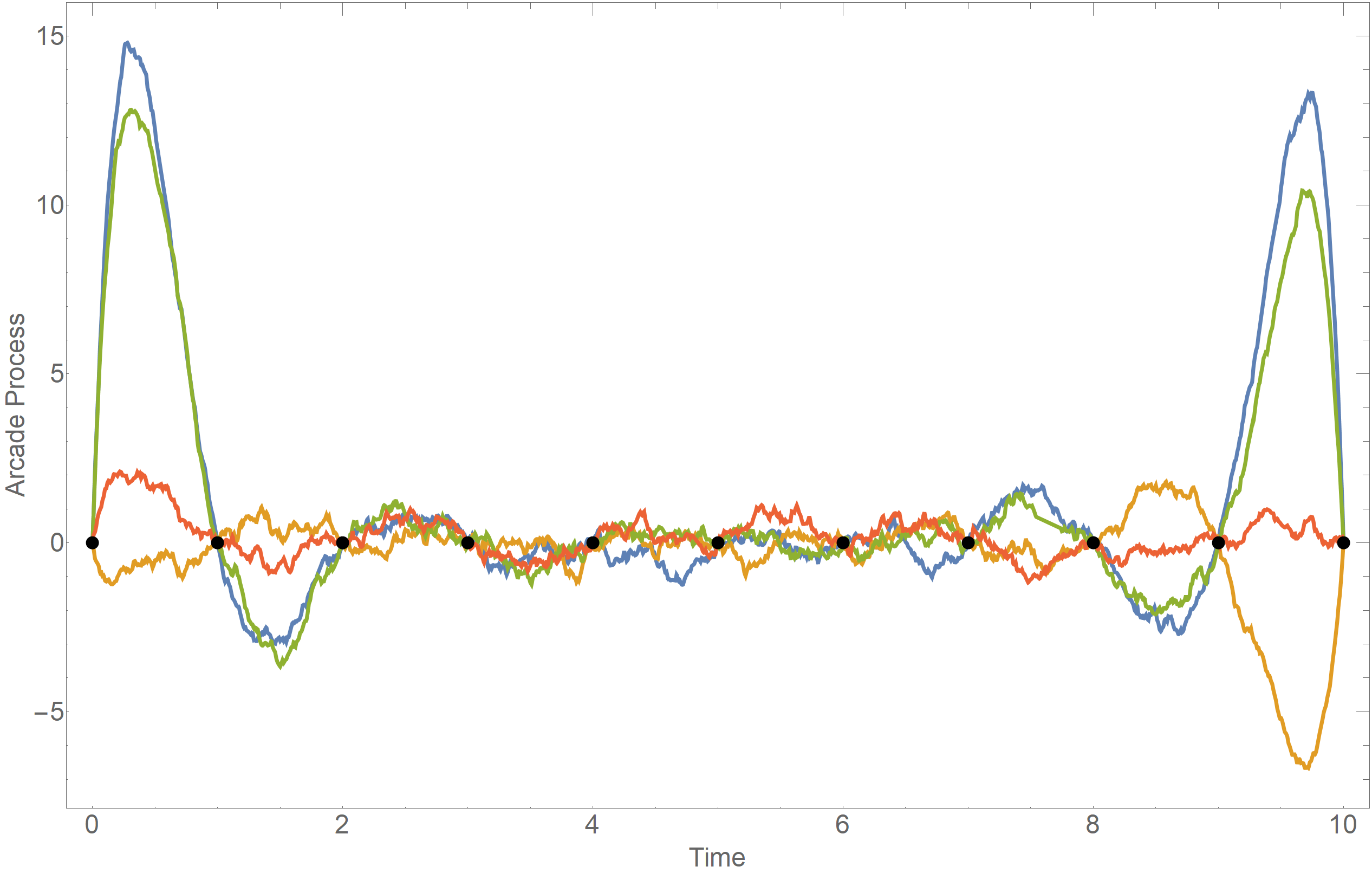}
\caption{Paths simulation of a Lagrange-Brownian AP with $n=10$, using the equidistant partition $\{T_i = i \mid i=0,1,\ldots, 10\}$, illustrating Runge's phenomenon, cf., Example \ref{ex25}.}
\end{minipage}
\hfill
\centering
\begin{minipage}[b]{0.49\textwidth}
\includegraphics[width=\textwidth]{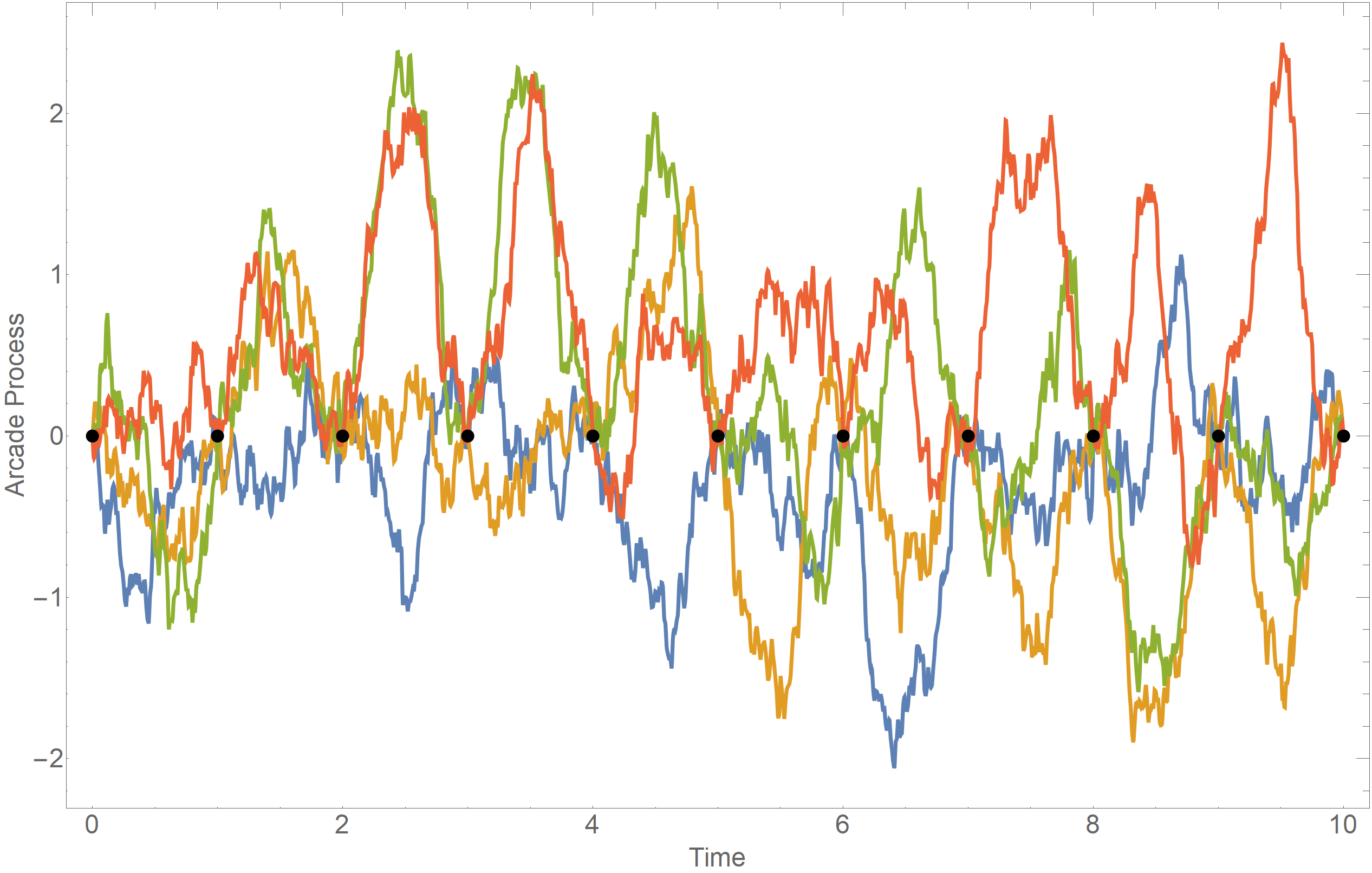}
\caption{Paths simulation of a Brownian AP with $n=10$ and interpolating coefficients $\tilde f_0, \ldots, \tilde f_{10}$, using the equidistant partition $\{T_i = i \mid i=0,1,\ldots, 10\}$, cf., Example \ref{ex25}.}
\end{minipage}
\end{figure}
\end{ex}
\begin{ex} \label{ex26}
Elliptic APs have interpolating coefficients given by
\begin{equation}
    f_0(t)= \sqrt{1 - \left( \frac{t-T_0}{T_1-T_0} \right)^2} \1_{[T_0, T_1]} (t), \quad f_n(t)= \sqrt{1 - \left( \frac{t-T_{n}}{T_n-T_{n-1}} \right)^2} \1_{(T_{n-1}, T_n]} (t),
\end{equation}
\begin{equation}
    f_i(t)= \sqrt{1 - \left( \frac{t-T_{i}}{T_i-T_{i-1}} \right)^2} \1_{(T_{i-1}, T_{i}]} (t) + \sqrt{1 - \left( \frac{t-T_{i}}{T_{i+1}-T_{i}} \right)^2} \1_{(T_{i}, T_{i+1}]} (t),
\end{equation}
for $i=1,\ldots, n-1$.
\begin{figure}[H]
\centering
\includegraphics[width=.5\textwidth]{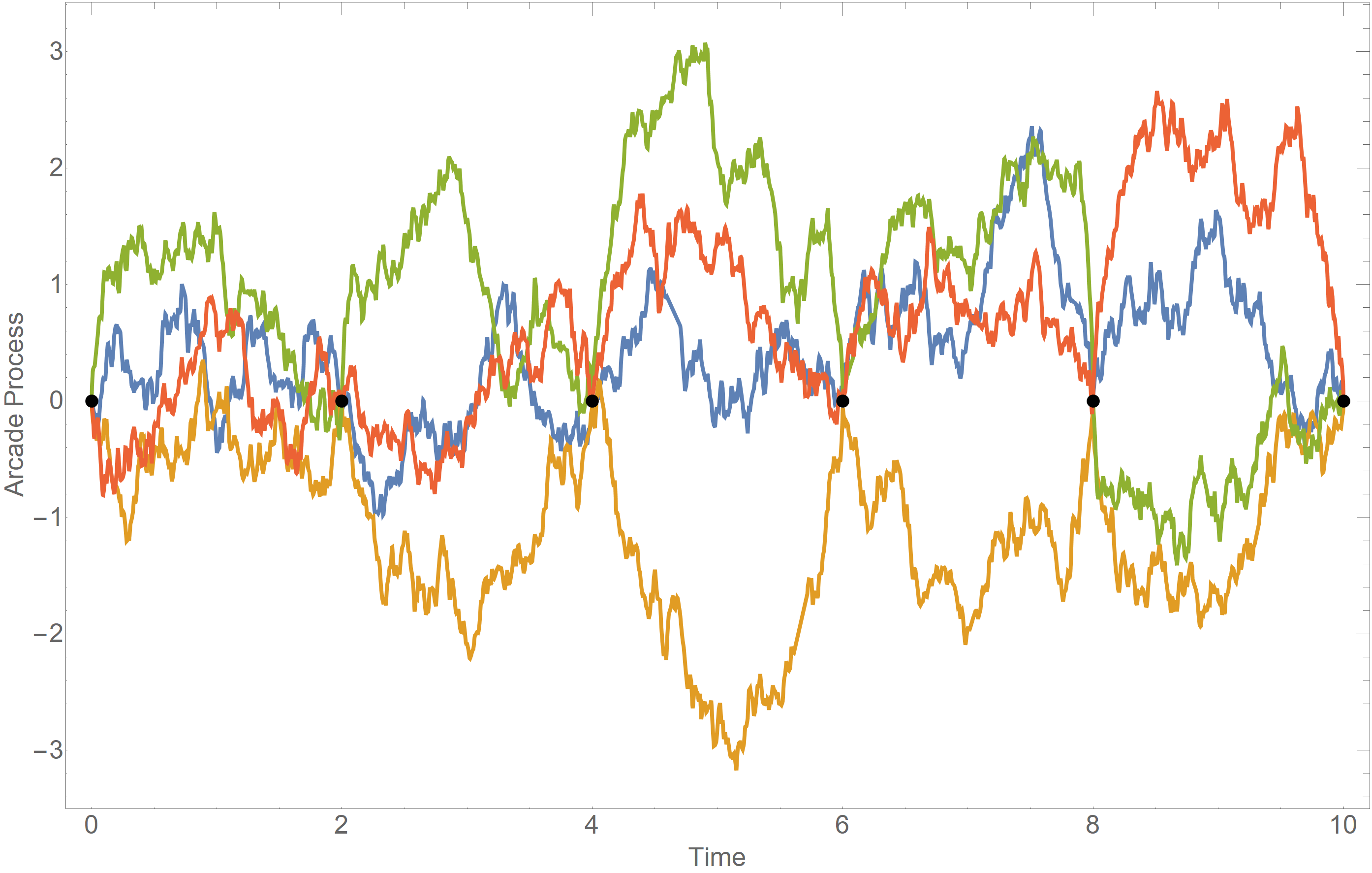}
\caption{Paths simulation of an elliptic-Brownian AP with $n=5$, using the equidistant partition $\{T_i = 2i \mid i=0,1,\ldots, 5\}$, cf., Example \ref{ex26}.}
\end{figure}
\end{ex}
\section{Filtered arcade martingales}\label{appendixB}
In this appendix, we include further examples of filtered arcade martingales.
\begin{ex}
Let $X_0 \sim \mathcal U ([-1,1])$, and $X_1 \sim \mathcal U ([-2,2])$. These probability distributions are convexly ordered, i.e., $X_0 \leqslant_{\mathrm{cx}} X_1$, hence there exists at least one martingale coupling for $(X_0,X_1)$. We choose the coupling defined by 
\begin{equation}
X_1 \mid X_0  = \left\{
    \begin{array}{ll}
        \frac{3}{2} X_0 + \frac{1}{2}  & \mbox{with probability } \frac{3}{4},\\ \\
         -\frac{1}{2} X_0 - \frac{3}{2}  & \mbox{with probability } \frac{1}{4}.
    \end{array}
\right.
\end{equation}
This is a martingale coupling since $\E[X_1 \mid X_0] = X_0$. In fact, it can be shown that this coupling is the solution to a martingale optimal transport problem, see \cite{Sester}. For any $\{T_0,T_1\}$-conditionally Markov $X$-RAP $(I_t^{(1)})$, we have 
\begin{align}
M_t &={\E [ X_1 \mid X_0, I_t^{(1)} ]} \nn \\
& = \frac{(9 X_0 + 3) f^{I_t^{(1)} \mid X_0,X_1=\frac{3}{2} X_0 + \frac{1}{2}} (I_t^{(1)})  -  ( X_0 + 3 ) f^{I_t^{(1)} \mid X_0,X_1=-\frac{1}{2} X_0 - \frac{3}{2} } (I_t^{(1)})}  { 6 f^{I_t^{(1)} \mid X_0,X_1=\frac{3}{2} X_0 + \frac{1}{2}} (I_t^{(1)})  +  2f^{I_t^{(1)} \mid X_0,X_1=-\frac{1}{2} X_0 - \frac{3}{2} } (I_t^{(1)})}.
\end{align}
\end{ex}
\begin{ex}
Let $X_0 \sim \mathcal N (0,1)$ and $X_1 \sim \mathcal N (0,2)$, where $X_1 \mid X_0 \sim \mathcal N (X_0,1)$. For any $\{T_0,T_1\}$-conditionally Markov $X$-RAP $(I_t^{(1)})$, we have 
\begin{equation}
    M_t =\E [ X_1 \mid X_0, I_t^{(1)} ]= \frac{\int_\R y f^{I_t^{(1)} \mid X_0,X_1=y} (I_t^{(1)}) \e^{\frac{-(y-X_0)^2}{2}}  \rd y}  {\int_\R f^{I_t^{(1)} \mid X_0,X_1=y} (I_t^{(1)}) \e^{\frac{-(y-X_0)^2}{2}}  \rd y} .
\end{equation}
\end{ex}
\begin{ex}
Let $D_t=tB_t$ where $(B_t)_{t \geqslant0}$ is a standard Brownian motion. Then $K_D(x,y)= \min(x,y)^2 \max(x,y)$, and so $H_1(x)=x^2$, $H_2(x)=x$. The standard AP on $[T_0,T_1]$ driven by $(D_t)$ is given by
\begin{equation}
    A_t^{(1)}= D_t - \frac{T_1t-t^2}{T_1T_0 - T_0^2} D_{T_0} - \frac{t^2-tT_0}{T_1^2 - T_1T_0} D_{T_1}.
\end{equation}
We have 
\begin{equation}
    K_A(x,y)= \frac{\min(x,y)(\min(x,y)-T_0) \max(x,y) (T_1-\max(x,y))}{T_1-T_0},
\end{equation}
hence $A_1(x)= (x(x-T_0))/(T_1-T_0)$ and $A_2(x)=x(T_1-x)$. A standard $X$-RAP with noise process $(A_t^{(1)})_{t\in [T_0,T_1]}$ is 
\begin{equation}
    I_t^{(1)}= D_t - \frac{T_1t-t^2}{T_1T_0 - T_0^2} (D_{T_0} -X_0) - \frac{t^2-tT_0}{T_1^2 - T_1T_0} (D_{T_1}-X_1).
\end{equation}
We recall that a different interpolating coefficient for $X_0$ could have been chosen without disrupting the standard property of $(I_t^{(1)})$. The quadratic variation of $(I_t^{(1)})$ is given by $\rd [I^{(1)}]_t= H_1'(t)H_2(t)  - H_1(t) H_2'(t) \rd t= t^2 \rd t$. Hence,
\begin{equation}
    M_t = X_0 + \int_{T_0}^t \frac{\var [ X_1 \mid X_0, I_t^{(1)} ] s }{T_1^2 s - s^2 T_1} \rd W_s =  X_0 + \int_{T_0}^t \frac{\var [ X_1 \mid X_0, I_s^{(1)} ]  }{T_1^2  - s T_1} \rd W_s.
\end{equation}
\end{ex}
\begin{ex}
Let $(D_t)$ be the Ornstein-Uhlenbeck process satisfying the SDE
\begin{equation}
    D_{t}=\int_0^t\theta\left(\mu-D_{s}\right) \rd s+\int_0^t\sigma \rd B_{s},
\end{equation}
where $\theta>0, \sigma>0, \mu \in \R$. Then $ K_D(s,t)=\frac{\sigma^2}{2 \theta} \e^{\theta \min(s,t)} \e^{-\theta \max(s,t)}$, and so $H_1(x)= \\\sigma^2/(2 \theta) \exp{[\theta x]}$, $H_2(x)=\exp{[-\theta x]}$. The standard AP driven by $(D_t)$ is 
\begin{equation}
    A_{t}^{(1)}=D_t-\frac{\e^{\theta (T_1-t)}-\e^{-\theta (T_1-t)}}{\e^{\theta (T_1-T_0)}-\e^{-\theta (T_1-T_0)}}  D_{T_0} - \frac{\e^{\theta (t-T_0)}-\e^{-\theta (t-T_0)}}{\e^{\theta (T_1-T_0)}-\e^{-\theta (T_1-T_0)}} D_{T_1}.
\end{equation}
The standard $X$-RAP with noise process $(A_t^{(1)})$ and the same interpolating coefficients for the signal function as in the noise process is 
\begin{equation}
    I_t^{(1)}= D_t - \frac{\e^{\theta (T_1-t)}-\e^{-\theta (T_1-t)}}{\e^{\theta (T_1-T_0)}-\e^{-\theta (T_1-T_0)}} (D_{T_0} -X_0) - \frac{\e^{\theta (t-T_0)}-\e^{-\theta (t-T_0)}}{\e^{\theta (T_1-T_0)}-\e^{-\theta (T_1-T_0)}} (D_{T_1}-X_1).
\end{equation}
The quadratic variation of $(I_t^{(1)})$ is given by $\rd [I^{(1)}]_t= H_1'(t)H_2(t)  - H_1(t) H_2'(t) \rd t= \sigma^2 \rd t$. Hence, the associated FAM takes the form
\begin{equation}
    M_t = X_0 + \int_{T_0}^t \frac{\sigma \var [ X_1 \mid X_0, I_t^{(1)} ]}{\frac{ \sigma^2}{2 \theta} \left(\e^{\theta(T_1-s)} - \e^{\theta(s-T_1)} \right)} \rd W_s = X_0 + \frac{2\theta}{\sigma}\int_{T_0}^t \frac{\var [ X_1 \mid X_0, I_t^{(1)} ]  }{ \e^{\theta(T_1-s)} - \e^{\theta(s-T_1)} } \rd W_s.
\end{equation}
\end{ex}

%


\begin{thebibliography}{99}
\bibitem{Veraguas} Backhoff-Veraguas, J., Beiglböck, M., Huesmann, M., and Källblad, S. (2020) \textit{Martingale Benamou-Brenier: A probabilistic perspective}. The Annals of Probability, 48(5), pp.2258-2289.

\bibitem{Beiglbock} Beiglböck, M. and Juillet, N. (2016) \textit{On a problem of optimal transport under marginal martingale constraints}. The Annals of Probability, 44(1), pp.42-106.

\bibitem{Beiglbock2} Beiglböck, M., Nutz, M., and Touzi, N. (2017) \textit{Complete duality for martingale optimal transport on the line}. The Annals of Probability, 45(5), pp.3038-3074.

\bibitem{Blackwell} Blackwell, D. (1951) \textit{Comparison of experiments}. Proceedings of the Second Berkeley Symposium on Mathematical Statistics and Probability, University of California Press, pp.93-102.

\bibitem{Bogachev} Bogachev, V.I., Kolesnikov, A.V., and Medvedev, K.V. (2012) \textit{The Monge-Kantorovich problem: achievements, connections, and perspectives}. Russian Mathematical Surveys, 67(5), pp.785-890.

\bibitem{Brenier} Brenier, Y. (1987) \textit{Décomposition polaire et réarrangement monotone des champs de vecteurs}. Comptes rendus de l'Académie des Sciences, Paris, Série I, 305, pp.805-808.

\bibitem{BHM1} Brody, D.C., Hughston, L.P., and Macrina, A. (2008) \textit{Information-based asset pricing}. International Journal of Theoretical and Applied Finance, 11(01), pp.107-142.

\bibitem{BHM2} Brody, D.C., Hughston, L.P., and Macrina, A. (2022) \textit{Financial Informatics: An Information-Based Approach to Asset Pricing}. Singapore: World Scientific Publishing Co. Pte. Ltd.

\bibitem{Fitzsimmons} Fitzsimmons, P., Pitman, J., and Yor, M. (1992) \textit{Markovian Bridges: Construction, Palm Interpretation, and Splicing}. In Seminar on Stochastic Processes, 1992, pp 101-134. Springer, New York, NY.

\bibitem{FKK} Fujisaki, M., Kallianpur, G., and Kunita, H. (1972) \textit{Stochastic Differential Equations for the Non Linear Filtering Problem}. Osaka Journal of Mathematics, 9, pp. 19-40.

\bibitem{Labordere} Henry-Labordère, P. (2019) \textit{From (Martingale) Schrödinger Bridges to a New Class of Stochastic Volatility Model}. arXiv. Available at: \url{https://arxiv.org/abs/1904.04554}.

\bibitem{Hirsch} Hirsch, F., Profeta, C., Roynette, B., and Yor, M. (2011) \textit{Peacocks and Associated Martingales, with Explicit Constructions}. Milan: Springer.

\bibitem{Hoyle} Hoyle, E., Hughston, L.P., and Macrina, A. (2011) \textit{Lévy random bridges and the modelling of financial information}. Stochastic Processes and their Applications, 121(4), pp. 856-884.

\bibitem{Kantorovich} Kantorovich, L.V. (1942) \textit{On translation of mass}. Proceedings of the USSR Academy of Sciences, 37, pp. 199–201.

\bibitem{Kassis2} Kassis, G. (2024) \textit{The Quadratic Variation of Gauss-Markov Semimartingales}. Available at: \url{https://doi.org/10.48550/arXiv.2405.18270}.

\bibitem{IBMOT} Kassis, G. and Macrina, A. (2024) \textit{Information-Based Martingale Optimal Transport}. Available at: \url{https://doi.org/10.48550/arXiv.2410.16339}.

\bibitem{Kellerer} Kellerer, H.G. (1972) \textit{Markov-Komposition und eine Anwendung auf Martingale}. Mathematische Annalen, 198, pp. 99–122.

\bibitem{Lecam} Le Cam, L. (1996) \textit{Comparison of Experiments: A Short Review}. IMS Lecture Notes - Monograph Series, 30, pp. 127-138.

\bibitem{Leonard1} Léonard, C. (2012) \textit{From the Schrödinger problem to the Monge-Kantorovich problem}. Functional Analysis, 262, pp. 1879–1920.

\bibitem{Lipster} Liptser, R.S. and Shiryaev, A.N. (2001) \textit{Statistics of random processes: I. General theory}. 2nd ed. Berlin: Springer.

\bibitem{Macrina} Macrina, A. (2006) \textit{An Information-Based Framework for Asset Pricing: X-Factor Theory and its Applications}. PhD thesis, Department of Mathematics, King's College London.

\bibitem{Macrina2} Macrina, A. and Sekine, J. (2021) \textit{Stochastic modelling with randomised Markov bridges}. Stochastics, 93(1), pp. 29–55.

\bibitem{Madan} Madan, D. and Yor, M. (2002) \textit{Making Markov martingales meet marginals: with explicit constructions}. Bernoulli, 8(4), pp. 509-536.

\bibitem{Menguturk} Mengütürk, L. A. and Mengütürk, M. C. (2020) \textit{Stochastic Sequential Reduction of Commutative Hamiltonians}. Journal of Mathematical Physics, 61, 102104.

\bibitem{Monge} Monge, G. (1781) \textit{Mémoire sur la théorie des déblais et des remblais}. De l'Imprimerie Royale.

\bibitem{Peyre1} Peyré, G. and Cuturi, M. (2019) \textit{Computational optimal transport}. Foundations and Trends in Machine Learning, 11(5-6), pp. 355–607.

\bibitem{Revuz} Revuz, D. and Yor, M. (1999) \textit{Continuous Martingales and Brownian Motion}. 3rd ed. Berlin: Springer.

\bibitem{Schrodinger} Schrödinger, E. (1932) \textit{Sur la théorie relativiste de l'électron et l'interprétation de la mécanique quantique}. Annales de l'Institut Henri Poincaré, 2, pp. 269–310.

\bibitem{Sester} Sester, J. (2023) \textit{On intermediate marginals in martingale optimal transportation}. Mathematics and Financial Economics, 17, pp.615-654.

\bibitem{Sherman} Sherman, S. (1951) \textit{On a Theorem of Hardy, Littlewood, Polya, and Blackwell}. Proceedings of the National Academy of Sciences of the United States of America, 37(12), pp. 826-831.

\bibitem{Stein} Stein, C. (1951) \textit{Notes on the Comparison of Experiments}. University of Chicago (mimeographed).

\bibitem{Strassen} Strassen, V. (1965) \textit{The existence of probability measures with given marginals}. Annals of Mathematical Statistics, 36, pp. 423–439.

\bibitem{Villani1} Villani, C. (2003) \textit{Topics in Optimal Transportation}. Providence, RI: American Mathematical Society. ISBN: 978-1-4704-1804-5.

\bibitem{Villani2} Villani, C. (2009) \textit{Optimal Transport: Old and New}. Berlin: Springer. ISBN: 978-3-540-71050-9.
\end{thebibliography}
\end{document}